\documentclass[12pt,a4paper,reqno]{amsart}
\usepackage{amssymb}
\usepackage{amscd}
\usepackage{enumerate}
\usepackage{graphicx}
\usepackage{siunitx}
\usepackage{tikz-cd}
\usepackage{color}
\usetikzlibrary{arrows}
\numberwithin{equation}{section}

\usepackage{mathabx}

\usepackage{mathtools}
\usepackage[tableposition=top]{caption}
\usepackage{booktabs,dcolumn}

\DeclareFontFamily{OT1}{rsfs}{}
\DeclareFontShape{OT1}{rsfs}{n}{it}{<-> rsfs10}{}
\DeclareMathAlphabet{\mathscr}{OT1}{rsfs}{n}{it}

\theoremstyle{plain}

\newtheorem{theorem}{Theorem}[section]
\newtheorem{proposition}[theorem]{Proposition}

\theoremstyle{definition}

\newtheorem{remark}[theorem]{Remark}

\newcommand\R{\mathbb{R}}
\newcommand\Z{\mathbb{Z}}

\newcommand\C{\mathbb{C}}

\newcommand\eps{\varepsilon}

\begin{document}

\title[Blowup for supercritical NLS]{Finite time blowup for a supercritical defocusing nonlinear Schr\"odinger system}

\author{Terence Tao}
\address{UCLA Department of Mathematics, Los Angeles, CA 90095-1555.}
\email{tao@math.ucla.edu}



\begin{abstract}  We consider the global regularity problem for defocusing nonlinear Schr\"odinger systems
$$ i \partial_t + \Delta u = (\nabla_{\R^m} F)(u) + G $$
on Galilean spacetime $\R \times \R^d$, where the field $u\colon \R^{1+d} \to \C^m$ is vector-valued, $F\colon \C^m \to \R$ is a smooth potential which is positive, phase-rotation-invariant, and homogeneous of order $p+1$ outside of the unit ball for some exponent $p >1$, and $G: \R \times \R^d \to \C^m$ is a smooth, compactly supported forcing term.  This generalises the scalar defocusing nonlinear Schr\"odinger (NLS) equation, in which $m=1$ and $F(v) = \frac{1}{p+1} |v|^{p+1}$.  It is well known that in the energy sub-critical and energy-critical cases when $d \leq 2$ or $d \geq 3$ and $p \leq 1+\frac{4}{d-2}$, one has global existence of smooth solutions from arbitrary smooth compactly supported initial data $u(0)$ and forcing term $G$, at least in low dimensions.  In this paper we study the supercritical case where $d \geq 3$ and $p > 1 + \frac{4}{d-2}$.  We show that in this case, there exists a smooth potential $F$ for some sufficiently large $m$, positive and homogeneous of order $p+1$ outside of the unit ball, and a smooth compactly choice of initial data $u(0)$ and forcing term $G$ for which the solution develops a finite time singularity.  In fact the solution is locally discretely self-similar with respect to parabolic rescaling of spacetime.  This demonstrates that one cannot hope to establish a global regularity result for the scalar defocusing NLS unless one uses some special property of that equation that is not shared by these defocusing nonlinear Schr\"odinger systems.

As in a previous paper \cite{tao-nlw} of the author considering the analogous problem for the nonlinear wave equation, the basic strategy is to first select the mass, momentum, and energy densities of $u$, then $u$ itself, and then finally design the potential $F$ in order to solve the required equation.  
\end{abstract}

\maketitle


\section{Introduction}

Let $\C^m$ be a standard finite-dimensional complex vector space, with the real inner product
$$ \langle (z_1,\dots,z_m), (w_1,\dots,w_m) \rangle_{\C^m} \coloneqq \operatorname{Re} \sum_{j=1}^m z_j \overline{w_j}$$
and norm $\|z\|_{\C^m} \coloneqq \langle z, z \rangle_{\C^m}^{1/2}$.

A function $F\colon \C^m \to \R$ is said to be \emph{phase-rotation-invariant and homogeneous of order $\alpha$} for some real $\alpha$ if we have
\begin{equation}\label{homog0}
 F(\lambda v) = |\lambda|^\alpha F(v)
\end{equation}
for all $\lambda \in \C$ and $v \in \C^m$; thus for instance $F(e^{i\theta} v) = F(v)$ for all $\theta \in \R$ and $v \in \C^m$.  In particular, differentiating \eqref{homog0} at $\lambda=1$ we obtain \emph{Euler's identity}
\begin{equation}\label{euler}
\langle v, (\nabla_{\C^m} F)(v)\rangle_{\C^m} = \alpha F(v)
\end{equation}
as well as the variant
\begin{equation}\label{variant}
\langle iv, (\nabla_{\C^m} F)(v)\rangle_{\C^m} = 0
\end{equation}
for all $v \in \C^m$ where a gradient $\nabla_{\C^m} F(v) \in \C^m$ exists.  Here the gradient $\nabla_{\C^m} F(v)$ is defined via duality by the formula
\begin{equation}\label{drv}
 \langle (\nabla_{\C^m} F)(v), w \rangle_{\C^m} = \frac{d}{dt} F(v+tw)|_{t=0}
\end{equation}
for all test directions $w \in \C^m$.
When $\alpha$ is not an integer, it is not possible for such homogeneous functions to be smooth at the origin unless they are identically zero (this can be seen by performing a Taylor expansion of $F$ around the origin).  To avoid this technical issue, we also introduce the notion of $F$ being \emph{phase-rotation-invariant and homogeneous of order $\alpha$ outside of the unit ball}, by which we mean that \eqref{homog0} holds for $\lambda \in \C$ and $v \in \C^m$ whenever $|\lambda|, \|v\|_{\C^m} \geq 1$, or whenever $|\lambda|=1$.

Define a \emph{potential} to be a function $F\colon \C^m \to \R$ that is smooth away from the origin; if $F$ is also smooth at the origin, we call it a \emph{smooth potential}.  We say that the potential is \emph{defocusing} if $F$ is positive away from the origin, and \emph{focusing} if $F$ is negative away from the origin.  In this paper we consider nonlinear Schr\"odinger systems of the form
\begin{equation}\label{nls}
i \partial_t u + \Delta u = (\nabla_{\C^m} F)(u) + G
\end{equation}
where the unknown field $u\colon \R \times \R^d \to \C^m$ is assumed to be smooth, $\Delta = \partial_{x_j} \partial_{x_j}$ is the spatial Laplacian (with the usual summation conventions), $\partial_t, \partial_{x_1}, \dots, \partial_{x_d}$ are the partial derivatives in time and space, $F\colon \C^m \to \R$ is a smooth potential, and $G\colon \R \times \R^d \to \C^m$ is a smooth compactly supported forcing term.  In the homogeneous case $G=0$, this is (formally, at least) a Hamiltonian evolution equation, with Hamiltonian
$$ H(u) \coloneqq \int_{\R^d} \frac{1}{2} \| \nabla u \|_{\R^d \otimes \C^m}^2 + F(u)\ dx$$
which is non-negative when $F$ is defocusing, where the quantity $\| \nabla u \|_{\R^d \otimes \C^m}^2$ is given by the formula
$$ \| \nabla u \|_{\R^d \otimes \C^m}^2 \coloneqq \langle \partial_{x_j} u, \partial_{x_j} u \rangle_{\C^m}$$
with the usual summation conventions.  By Noether's theorem, the phase rotation invariance of this Hamiltonian yields (formally, at least) the conservation of mass $\int_{\R^d} \|u\|_{\C^m}^2\ dx$, while the translation invariance of the Hamiltonian similarly yields conservation of the momentum $2\int_{\R^d} \langle \partial_j u, iu \rangle_{\C^m}\ dx$.

We will restrict attention to potentials $F$ which are phase-rotation-invariant and homogeneous outside of the unit ball of order $p+1$ for some exponent $p>1$.  The well-studied \emph{nonlinear Schr\"odinger equation} (NLS) corresponds to the case when $m=1$ and $F(v) = \frac{|v|^{p+1}}{p+1}$ (for defocusing NLS) or $F(v) = - \frac{|v|^{p+1}}{p+1}$ (for focusing NLS), with the caveat that one needs to restrict $p$ to be an odd integer if one wants these potentials to be smooth at the origin.  

The natural initial value problem to study here is the Cauchy initial value problem, in which one specifies a smooth initial position $u_0\colon \R^d \to \C^m$ and forcing term $G: \R \times \R^d \to \C^m$, as well as the potential $F$, and asks for a smooth solution $u$ to \eqref{nls} with $u(0,x) = u_0(x)$.   To avoid illposedness issues relating to the infinite speed of propagation of the Schr\"odinger equation, we will require the data $u_0$ and $G$ to be compactly supported in space, and restrict attention to solutions $u$ that are in the Schwartz class.

Standard energy methods (see e.g. \cite{cazenave} or \cite{tao-book}) show that for any choice of smooth compactly supported data $u_0: \R^d \to \C^m$ and smooth compactly supported forcing term $G: \R \times \R^d \to \C^m$, one can construct a unique smooth solution $u$ to \eqref{nls} in $(-T_-,T_+) \times \R^d$ for some $0 < T_-, T_+ \leq \infty$ which is Schwartz in space, with $T_-, T_+$ maximal amongst all such solutions.  Furthermore, if $T_+ < \infty$, then $\|u(t) \|_{L^\infty}$ goes to infinity as $t \to T_+$, and similarly for $T_-$.  In these latter situations we say that the initial value problem exhibits finite time blowup.

The \emph{global regularity problem} for a given choice of potential $F$ asks if the latter situation does not occur, that is to say that for every choice of smooth, compactly supported data $u_0,G$ there is a smooth global solution.  

The answer to this question depends in a somewhat complicated way on the dimension $d$, the exponent $p$, and whether the potential $F$ is focusing or defocusing; the literature here is vast and the following discussion is not meant to be comprehensive.  Readers may consult the texts \cite{cazenave}, \cite{borg-book}, \cite{tao-book} for more complete references.  

Consider first the \emph{mass subcritical} case $p < 1 + \frac{4}{d}$.  It is known in this case from Strichartz estimates and contraction mapping arguments (see e.g. \cite{cazenave}) that the initial value problem is globally well-posed in the Sobolev space $H^1(\R^d)$, regardless of whether the potential $F$ is defocusing or not; in the low-dimensional case $d \leq 3$, Strichartz estimates then place the solution locally in the space $L^4_t L^\infty_x(\R \times \R^d)$, which is sufficient when combined with standard persistence of regularity arguments based on the energy method (see e.g. \cite[Proposition 3.11]{tao-book}) shows that solutions remain smooth for all time.  The case $d=4$ can be handled by modifications of the arguments in \cite{rv}.  The global regularity question in higher dimensions $d>4$ is still not fully resolved; note that for the analogous question for the nonlinear wave equation (NLW), it was shown recently in \cite{tao-high} that global regularity can in fact fail in extremely high dimensions $d \geq 11$, even in the ``extremely subcritical'' case when the potential $F$ and all of its derivatives are bounded.

Now consider the case when $p$ is \emph{mass critical or supercritical} in the sense that $p \geq 1 + \frac{4}{d}$, but is also \emph{energy critical or subcritical} in the sense that either $d < 3$, or $p \leq 1 + \frac{4}{d-2}$.  In the case of the focusing NLS, the well known viriel argument of Glassey \cite{glassey} shows that finite time blowup can\footnote{Global regularity can however be restored if one imposes a suitable smallness condition on the data $u_0,G$; see e.g. \cite{cazenave}.} occur (and in fact \emph{must} occur if the initial data has negative Hamiltonian).  If instead the potential is defocusing, then it is known that the initial value problem is globally well-posed in the energy space $H^1(\R^d)$.  In energy-subcritical situations when $d<3$ or $p < 1 + \frac{4}{d-2}$, this claim can again be established from Strichartz estimates and contraction mapping arguments; see e.g. \cite{cazenave}, \cite{borg-book}, \cite{tao-book}.  The energy-critical case when $d \geq 3$ and $p = 1 + \frac{4}{d-2}$ is more delicate; in the case of scalar NLS (in which $m=1$ and $F(u) = \frac{|u|^{p+1}}{p+1}$), the $d=3$ case was established in \cite{gopher} (after several previous partial results), and the higher dimensional cases $d=4$ and $d>4$ were treated\footnote{These papers are primarily concerned with the homogeneous case $G=0$, but one can use the stability properties of NLS (see e.g. \cite{TV}) to extend from the homogeneous case to the inhomogeneous case, at least in the context of $H^1$ global well-posedness.} in \cite{rv} and \cite{visan} respectively.  It is likely that these results can be extended to more general defocusing potentials, though we do not attempt this here.  Again, in low dimensional cases $d \leq 3$, this $H^1$ local well-posedness can be used in conjunction with Strichartz estimates to establish global regularity; see e.g. \cite{borg-book}, \cite{ckstt}, \cite{tao-book}; the $d=4$ case was treated in \cite{rv}.  As before, the status of the global regularity question in higher dimensions $d>4$ is not yet fully resolved.

Finally, we turn to the \emph{energy-supercritical} case when $d \geq 3$ and $p > 1 + \frac{4}{d-2}$, which is the main focus of this paper.  The Glassey viriel argument \cite{glassey} continues to show that finite time blowup can occur here in the focusing case.  In the defocusing case, the situation is less well understood.  There are a number of results \cite{bgt}, \cite{cct}, \cite{carles}, \cite{alazard}, \cite{carles-2}, \cite{burq} that demonstrate that the solution map, if it exists at all, is highly unstable, although one can at least construct global weak solutions, which are not known to be unique; see \cite{gv}, \cite{alazard}, \cite{tao-weak}.  

The main result of this paper is to show that, at least for certain choices of defocusing potential $F$ and data $u_0, G$, one in fact has blowup in finite time.

\begin{theorem}[Finite time blowup]\label{main}  Let $d \geq 3$, let $p > 1 + \frac{4}{d-2}$, and let $m$ be a sufficiently large integer. Then there exists a defocusing smooth potential $F\colon \C^m \to \R$ that is phase-rotation-invariant and homogeneous of order $p+1$ outside of the unit ball, and a smooth compactly supported of initial data $u_0\colon \R^d \to \C^m$ and forcing term $G\colon \R^d \times \R \to \C^m$, such that there is a smooth, compactly supported solution $u\colon [0,1)\times \R^d \to \C^m$ to the nonlinear Schr\"odinger system \eqref{nls} with the property that $\|u(t)\|_{L^\infty(\R^d)}$ goes to infinity as $t \to 1^-$.
\end{theorem}

\begin{table}[ht]
\caption{A somewhat oversimplified summary of whether nonlinear Schr\"odinger systems are necessarily globally well-posed, or can admit finite time blowup solutions, for various criticality types of exponents and for both focusing and defocusing nonlinearities. Theorem \ref{main} establishes the bottom entry on the third column.}\label{blow}
\begin{tabular}{|l|l|l|l|}
\hline 
Mass & Energy & Defocusing & Focusing \\
\hline 
Subcritical & Subcritical & Global well-posedness & Global well-posedness \\
Critical & Subcritical & Global well-posedness & Finite time blowup \\
Supercritical & Subcritical & Global well-posedness & Finite time blowup \\
Supercritical & Critical & Global well-posedness & Finite time blowup \\
Supercritical & Supercritical & Finite time blowup & Finite time blowup \\
\hline
\end{tabular}
\end{table}

When combined with the known uniqueness theory for the equations \eqref{nls} (see e.g. \cite{cazenave}, \cite{tao-book}) we see that there cannot be any smooth global solution to \eqref{nls} with this data that is Schwartz in space (one can relax the Schwartz requirement considerably, but we will not attempt to do so here).  The presence of the forcing term $G$ is an unfortunate artefact of our method, which (due to the absence of finite speed of propagation for Schr\"odinger equations) requires one to use the forcing term to truncate a solution to a homogeneous equation that decays too slowly at infinity.  It is however reasonable to conjecture that the above theorem can be strengthened by making $G$ vanish (with $u$ now being Schwartz in space rather than compactly supported).

We have not attempted to optimise the value of $m$ produced by the arguments in this paper, but it will grow quadratically in the dimension $d$: $m=O(d^2)$.  It would of course be of great interest to set $m$ equal to $1$ in order to have the blowup result apply to the \emph{scalar} defocusing NLS; however our method requires a lot of ``freeness'' to the solution $u$ (in particular invoking a version of the Nash embedding theorem \cite{nash}), and it does not seem possible to adapt it for this purpose.  Nevertheless, Theorem \ref{main} does construct a ``barrier'' against any attempt to prove global regularity for the scalar supercritical NLS, in that any such attempt must crucially rely on some property of the scalar equation that is not enjoyed by the vector-valued equations considered here.  For instance, this theorem rules out any approach to global regularity for scalar supercritical NLS that relies on somehow manipulating the conservation laws of mass, momentum and energy to generate new \emph{a priori} bounds on the solution.

Theorem \ref{main} is an analogue of the recent finite time blowup result by the author \cite{tao-nlw} for vector-valued defocusing NLW equations, and the argument follows broadly similar lines, in particular performing a sequence of ``quantifier elimination'' steps, each of which removes one or more of the unknown fields from the problem.  

The first reduction is to reduce matters to constructing a discretely self-similar solution to a homogeneous NLS system \eqref{nls}, in which $G$ is now zero, the potential $F$ is homogeneous everywhere (not just outside the unit ball), and the solution $u$ obeys the discrete self-similarity relationship $u(4 t, 2 x ) = e^{i\alpha} 2^{-\frac{2}{p-1}} u(t,x)$ (the phase rotation $\alpha$ is needed for technical reasons, but can be ignored for a first reading).  In order to perform this reduction, it will be important that the self-similar solution $u$ remains smooth all the way up to the initial time slice $t=0$ (except at the spacetime origin $(t,x)=(0,0)$ where a singularity occurs).  See Theorem \ref{main-2} for a precise statement of the claim needed.

Now that the forcing term $G$ is eliminated from the problem, the next step is to eliminate the potential $F$, by first locating a self-similar field $u$, and then constructing a homogeneous defocusing potential $F$ to solve the equation \eqref{nls} with that given $u$.  In order for this to be possible, the field $u$ (as well as the ``potential energy'' field $V = F(u)$) have to obey some differential equations (related to the conservation of mass, momentum, and energy and the Euler identity \eqref{euler}), as well as some positivity and regularity hypotheses; see Theorem \ref{main-3} for a precise statement.  The derivation of Theorem \ref{main-2} from Theorem \ref{main-3} relies on a classical extension theorem of Seeley \cite{seeley} that allows one to extend a smooth function on a submanifold with boundary to a smooth function on the entire manifold.

The differential equations alluded to in the previous paragraph can be expressed in terms of the potential energy field $V$ and the ``Gram-type matrix'' $G[u,u]$ of $u$, which is a $(2d+4) \times (2d+4)$ matrix consisting of inner products $\langle D_1 u, D_2 u \rangle_{\C^m}$ for various differential operators 
$$ D_1, D_2 \in \{ 1, i, \partial_{x_1}, \dots, \partial_{x_d}, \partial_t, i \partial_{x_1}, \dots, i \partial_{x_d}, i \partial_t \}.$$
The coefficients of the Gram-type matrix $G[u,u]$ necessarily obey a number of constraints; for instance $G[u,u]$ will be symmetric and positive definite, and one has the Leibniz type identities
$$\partial_{x_j} \langle u,u \rangle_{\C^m} = 2 \langle u, \partial_{x_j} u \rangle_{\C^m}$$ 
and 
$$ \partial_{x_j} \langle iu, \partial_{x_k} u \rangle_{\C^m} - \partial_{x_k} \langle iu, \partial_{x_j} u \rangle_{\C^m} = 2 \langle i \partial_{x_j} u, \partial_{x_k} u \rangle_{\C^m}.$$  
One can then eliminate the field $u$ in favour of the Gram-type matrix by reducing Theorem \ref{main-3} to a statement about the existence of a certain matrix $G$ of fields (as well as a potential field $V$) obeying the above-mentioned constraints and differential equations; see Theorem \ref{main-4} for a precise statement.  In order to reconstruct the field $u$ from the Gram-type matrix $G$, one needs a ``partially complexified'' version of the Nash embedding theorem \cite{nash}; this is the main reason why the target dimension $m$ is required to be large.  Unfortunately, the existing forms of the Nash embedding theorem in the literature are not quite suitable for this application, and we need to adapt the \emph{proof} of that theorem to establish the embedding theorem required (which we formalise as Proposition \ref{gadc}).  The proof of this embedding theorem is given in Appendix \ref{gad-app}.
 
The Gram-type matrix $G$ contains a large number of fields, while simultaneously being required to obey a large number of constraints.  One can cut down the degrees of freedom considerably, as well as the number of constraints, by requiring the Gram-matrix to be homogeneous with respect to parabolic scaling, and also to be rotation-invariant in a certain tensorial sense.  This reduces the number of independent components of $G$ and $V$ to seven scalar fields $g_{1,1}, g_{\partial_r,\partial_r}, g_{\partial_\omega,\partial_\omega}, g_{\partial_r,\partial_t}, g_{1,i\partial_r}, g_{1,i\partial_t}, v$ which obey a certain number of conservation laws, positivity hypotheses, and some additional constraints such as homogeneity; see Theorem \ref{main-4} for a precise statement.  The fields $g_{D_1,D_2}$ for various differential operators $D_1,D_2$ are supposed to be proxies for the inner products $\langle D_1 u, D_2 u\rangle_{\C^m}$, while $v$ is a proxy for the potential energy $V(u)$.  (Strictly speaking, $G$ contains another scalar field $g_{\partial_t,\partial_t}$ (a proxy for $\| \partial_t u \|_{\C^m}^2$) which is independent of the other fields, but it is essentially unconstrained by any of the conservation laws, and can be set to be extremely large and then ignored.)

Amongst the various constraints between the remaining scalar fields is the energy conservation law, which in this notation becomes
\begin{equation}\label{energy-cons}
 \partial_t\left(\frac{1}{2} g_{\partial_r, \partial_r} + \frac{d-1}{2} g_{\partial_\omega,\partial_\omega} + v\right) - \left(\partial_r + \frac{d-1}{r}\right) g_{\partial_r, \partial_t} = 0.
\end{equation}
This law can be used in the energy sub-critical case to rule out the type of discretely self-similar solutions we are trying to construct here; with a bit more effort involving an additional Morawetz-type identity arising from momentum conservation, one can also rule out such solutions in the energy-critical case.  However, in the energy-supercritical case it turns out that the conservation law \eqref{energy-cons} is easy to satisfy, basically because the scalar field $g_{\partial_r, \partial_t}$ (representing energy current) that appears in this law has no presence in any of the other conservation laws, allowing the energy to be transported spatially at an essentially arbitrary rate.  In the energy-supercritical case, the total energy becomes infinite, and so it becomes possible to eliminate the field $g_{\partial_r, \partial_t}$ and the energy conservation equation \eqref{energy-cons}, reducing one to a variant of Theorem \ref{main-4} with one fewer scalar field and one fewer constraint equation.  
One can similarly use another constraint
$$ g_{1,i\partial_t} + \frac{1}{2} \left(\partial_r^2 + \frac{d-1}{r} \partial_r \right) g_{1,1} - g_{\partial_r,\partial_r} - (d-1) g_{\partial_\omega,\partial_\omega} = (p+1)v$$
(which ultimately arises from the Euler identity \eqref{euler}) to easily eliminate the field $g_{1,i\partial_t}$ (which makes no appearance in any of the other constraints), leaving one with just five remaining scalar fields $g_{1,1}, g_{\partial_r,\partial_r}, g_{\partial_\omega,\partial_\omega}, g_{1,i\partial_r}, v$; see Theorem \ref{main-5} for a precise statement.

An inspection of the remaining constraints reveals that the potential field $v$ and the angular stress $g_{\partial_\omega,\partial_\omega}$ play almost\footnote{This phenomenon is analogous to the well-known fact that when applying separation of variables in polar coordinates to the free Schr\"odinger equation $i \partial_t u + \Delta u = 0$ in which $u(t,r \omega) = v(t,r) Y_\ell(\omega)$ for some spherical harmonic $Y_\ell$ of degree $\ell$, the effect of the spherical harmonic is identical to that of a (defocusing) Coulomb type potential $\frac{\ell(\ell+1)}{r^2}$.} the same roles, and some elementary manipulations allow one to effectively absorb the potential $v$ into the angular stress $g_{\partial_\omega,\partial_\omega}$ (and also the radial stress $g_{\partial_r, \partial_r}$), allowing one to reduce to the case $v=0$; see Theorem \ref{main-6} for a precise statement.  Now there are just four independent scalar fields $g_{1,1}, g_{\partial_r,\partial_r}, g_{\partial_\omega,\partial_\omega}, g_{1,i\partial_r}$ that one needs to locate.

One of the remaining constraints is the momentum conservation law, which can be rewritten as
$$ \partial_r(r^{d-1} g_{\partial_r, \partial_r}) = (d-1) r^{d-2} g_{\partial_\omega,\partial_\omega} + S_1$$
where $S_1$ is the field
$$ S_1 \coloneqq \frac{1}{4} r^{d-1} \left( \partial_r \left(\partial_r^2 + \frac{d-1}{r} \partial_r\right) g_{1,1} + 2 \partial_t g_{1,i\partial_r} \right ).$$
One can integrate this law to obtain a representation of the radial stress $g_{\partial_r,\partial_r}$ as a certain integral involving $g_{\partial_\omega,\partial_\omega}$ and $S_1$.  The requirement that $g_{\partial_r,\partial_r}$ be smooth up to the initial time $t=0$ enforces some asymptotic vanishing conditions on the integrand, while the positive definiteness of the Gram matrix enforces an additional inequality on the integral.  Once these conditions are satisfied, one can then eliminate the field $g_{\partial_r,\partial_r}$ from the problem, leaving only three fields $g_{1,1}, g_{\partial_\omega,\partial_\omega}, g_{1,i\partial_r}$ to construct.  See Theorem \ref{main-7} for a precise statement.

The angular stress $g_{\partial_\omega,\partial_\omega}$ is now only constrained by a nonnegativity condition and by the constraints on the integral involving $g_{\partial_\omega,\partial_\omega}$ and $S_1$ mentioned above.  It is then not difficult to eliminate $g_{\partial_\omega,\partial_\omega}$, and reduce matters to locating just two fields $g_{1,1}, g_{1,i\partial_r}$ that obey a mass conservation law
$$ \partial_t g_{1,1} =  2 \left(\partial_r + \frac{d-1}{r} \right) g_{1,i\partial r}$$
together with a number of technical additional conditions, mostly involving integrals of the quantity $S_1$ mentioned above.  See Theorem \ref{main-8} for a precise statement.

The mass conservation law can be solved explicitly by using the ansatz
\begin{align*}
g_{1,1} &= 2 r^{1-d} \partial_r (r^d W) \\
g_{1,i\partial_r} &= r^{1-d} \partial_t (r^d W)
\end{align*}
for a suitable scalar field $W$.  Now that there is only one field $W$ to choose, it becomes possible to write down an explicit choice of this field that obeys the few remaining constraints required of it; we do so in Section \ref{conclude}.

The author is supported by NSF grant DMS-1266164 and by a Simons Investigator Award.

\section{Notation}

Throughout this paper, the spatial dimension $d$, the target dimension $m$, and the exponent $p$ will be fixed.  Unless otherwise stated, we will always be assuming the energy super-critical hypotheses
\begin{equation}\label{energy-supercrit}
d \geq 3; \quad p > 1 + \frac{4}{d-2}.
\end{equation}
We will also assume that the target dimension $m$ is sufficiently large depending on $d$.
In particular, all the theorems in subsequent sections will implicitly have these hypotheses present (though from Theorem \ref{main-4} onwards, the target dimension $m$ plays no further role as the field $u$ is eliminated at that point).

We use the asymptotic notation $X = O(Y)$ or $X \ll Y$ to denote the estimate $|X| \leq CY$ for some $C$ depending on the above parameters $p,d$.  In some cases we will explicitly allow the implied constant $C$ to depend on additional parameters.

Most of our analysis will take place in the spacetime region
\begin{equation}\label{hd-def}
 H_d \coloneqq ([0,+\infty) \times \R^d) \backslash \{(0,0)\}
\end{equation}
or the one-dimensional variant
\begin{equation}\label{h1-def}
 H_1 \coloneqq ([0,+\infty) \times \R) \backslash \{(0,0)\},
\end{equation}
that is to say on the portion of spacetime consisting of the present $t=0$ and future $t > 0$, but with the spacetime origin $(0,0)$ removed.  On these regions we introduce the parabolic magnitude function $\rho: H_d \to \R$ or $\rho: H_1 \to \R$ defined by
\begin{equation}\label{rho-def}
\rho(t,x) \coloneqq (t^2+|x|^4)^{1/4}
\end{equation}
for $(t,x) \in H_d$, or
\begin{equation}\label{rho1-def}
\rho(t,r) \coloneqq (t^2+r^4)^{1/4}
\end{equation}
for $(t,r) \in H_1$.  We also introduce the discrete scaling operator $T: H_d \to H_d$ by the formula
\begin{equation}\label{tscale}
 T(t,x) := (4t, 2x)
\end{equation}
(thus for instance $\rho \circ T = 2 \rho$)
and let $T^\Z := \{T^n: n \in \Z \}$ be the group of scalings generated by $T$.  A key point is that the quotient space $H_d/T^\Z$ of spacetime by discrete scalings is compact; indeed one can view this space as the set $\{ (t,x) \in H_d: 1 \leq \rho \leq 2 \}$ with the boundaries $\rho=1$ and $\rho=2$ identified. We have chosen to use the scaling $(t,x) \mapsto (4t, 2x)$ in \eqref{tscale} to generate the discrete self-similarity, but this is an arbitrary choice, and one could just as well have used another scaling $(t,x) \mapsto (\lambda_0^2 t, \lambda_0 x)$ for some fixed $\lambda_0>1$.

\section{Reduction to constructing a discretely self-similar solution}\label{reduct}

We begin the proof of Theorem \ref{main}. In analogy with the argument in \cite{tao-nlw}, the first step is to reduce to locating a discretely self-similar solution to a homogeneous nonlinear Schr\"odinger equation. thus eliminating the role of the forcing term $G$.  In the previous paper \cite{tao-nlw}, one could use the finite speed of propagation of nonlinear wave equations to restrict spacetime to a light cone $\{ (t,x): t > 0, |x| \leq t \}$ for the purposes of locating this solution.  In the current context of nonlinear Schr\"odinger equations, one has infinite speed of propagation, and so one can only restrict to the region $H_d$ defined in \eqref{hd-def}.
To get from here to Theorem \ref{main}, one must now apply a spatial cutoff, which is responsible for the forcing term $G$ that is present in this paper but not in the previous work \cite{tao-nlw}.

We turn to the details.  We will derive Theorem \ref{main} from

\begin{theorem}[First reduction]\label{main-2}  There exists a defocusing potential $F\colon \C^m \to \R$ which is phase-rotation-invariant and homogeneous of order $p+1$ and a smooth function $u \colon H_d \to \C^m \backslash \{0\}$ that solves \eqref{nls} (with $G=0$) on its domain and is nowhere vanishing, and also discretely self-similar in the sense that 
\begin{equation}\label{uts}
 u(T(t, x) ) = e^{i\alpha} 2^{-\frac{2}{p-1}} u(t,x)
\end{equation}
for all $(t,x) \in H_d$, and some $\alpha \in \R$, where $T$ is the scaling \eqref{tscale}.
\end{theorem}

A key point here is that $u$ is smooth all the way up to the boundary of the region $H_d$ (except at the spacetime origin $(0,0)$), rather than merely being smooth in the interior.  The exponent $-\frac{2}{p-1}$ is mandated by dimensional analysis considerations; the phase shift $\alpha$ is needed for more technical reasons, representing a ``total charge'' coming from the non-zero momentum density.  It would be natural to consider solutions that are continuously self-similar in the sense that 
$$  u(\lambda^2 t, \lambda x ) = \lambda^{-\frac{2}{p-1} + i \frac{\alpha}{\log 2}} u(t,x) $$
for all $\lambda > 0$ (not just powers of two), but we were unable to construct such a solution.  In the analogous situation for the NLW, such continuously self-similar solutions can be ruled out by \emph{ad hoc} methods for some ranges of $d,p$, as was shown in \cite[Proposition 2.2]{tao-nlw}.  

Let us assume Theorem \ref{main-2} for the moment, and show how it implies Theorem \ref{main}.  Let $F, u$ be as in Theorem \ref{main-2}.   Since $u$ is smooth and non-zero on the compact region $\{ (t,x) \in H_d: 1 \leq \rho \leq 2 \}$, it is bounded from below in this region.  By replacing $u$ with $Cu$ and $F$ with $v \mapsto C^2 F(v/C)$ for some large constant $C$, we may thus assume that
$$ \| u(t,x) \|_{\C^m} \geq 1$$
whenever $(t,x) \in H_d$ with $1 \leq \rho \leq 2 $.  Using the discrete self-similarity property \eqref{uts}, we then have this bound whenever $\rho \leq 2$; in fact we have a lower bound on $\|u(t,x)\|_{\C^m}$ that goes to infinity as $(t,x) \to 0$, ensuring in particular that $\|u(t)\|_{L^\infty(\R^d)}$ goes to infinity as $t \to 0$.

Using a smooth cutoff function, one can find a smooth defocusing potential $F_1 \colon \R^m \to \R$ that is phase-rotation-invariant and agrees with $F$ in the region $\{ v \in \C^m: \|v\|_{\C^m} \geq 1 \}$.  Then $u$ solves \eqref{nls} with this potential in the truncated region $\{ (t,x) \in H_d: \rho \leq 2 \}$, and in particular in the region $\{ (t,x) \in H_d: t, |x| \leq 1 \}$.  Next, let $\varphi: \R^d \to [0,1]$ be a smooth function supported on the ball $\{ x \in \R^d: |x| \leq 1 \}$ that equals one on $\{ x \in \R^d: |x| \leq \frac{1}{2}$, and define the functions $\tilde u: [0,1) \times \R^d \to \C^m$, $\tilde F: \C^m \to \R$, $\tilde G: [0,1) \times \R^d \to \C^m$,  by the formulae
\begin{align*}
\tilde u(t,x) &\coloneqq \overline{u}(1-t, x) \varphi(x) \\
\tilde F(v) &\coloneqq F_1(\overline{v}) \\
\tilde G(t,x) &\coloneqq i \partial_t \tilde u(t,x) + \Delta \tilde u(t,x) - \tilde F(\tilde u(t,x)).
\end{align*}
It is clear that $\tilde F$ is a smooth defocusing potential that is phase-rotation-invariant and homogeneous of degree $p+1$ outside of the unit ball, while $\tilde u, \tilde G$ are smooth functions supported on the regions $\{ (t,x) \in [0,1) \times \R^d: |x| \leq 1 \}$ and $\{ (t,x) \in [0,1) \times \R^d: \frac{1}{2} \leq |x| \leq 1 \}$, with $\|\tilde u(t) \|_{L^\infty}$ going to infinity as $t \to 1$.  This gives Theorem \ref{main} (with $u, F, G$ replaced by $\tilde u, \tilde F, \tilde G$ respectively).

It remains to prove Theorem \ref{main-2}.  This will be the focus of the remaining sections of the paper.  We remark that with the reduction to Theorem \ref{main-2}, we have effectively ``compactified'' spacetime, as the discretely self-similar solution can be viewed as a solution (interpreted geometrically as a section of an appropriate vector bundle) on the smooth compact manifold with boundary $H_d/T^\Z$.

\section{Eliminating the potential}\label{pot}

We now exploit the freedom to select the defocusing potential $F$ from Theorem \ref{main-2} by eliminating it from the equations of motion.  To motivate this elimination, let us formally manipulate the equation
$$ i \partial_t u + \Delta u = (\nabla_{\C^m} F)(u),$$
where $F$ is assumed to be defocusing, phase-rotation-invariant, and homogeneous of order $p+1$, in order to derive equations that do not explicitly involve $F$. 

From \eqref{euler}, \eqref{variant} we have the identities
\begin{equation}\label{v1}
 \langle i \partial_t u + \Delta u, u \rangle_{\C^m} = (p+1) V
\end{equation}
and
\begin{equation}\label{v2}
 \langle i \partial_t u + \Delta u, iu \rangle_{\C^m} = 0
\end{equation}
where we define the \emph{potential energy density} $V$ by
$$ V \coloneqq F(u).$$
Note that the defocusing nature of $F$ makes $V$ non-negative.  From \eqref{drv} and the chain rule we also have the additional identities
\begin{equation}\label{v3}
 \langle i \partial_t u + \Delta u, \partial_{x_j} u \rangle_{\C^m} = \partial_{x_j} V
\end{equation}
and
\begin{equation}\label{v4}
 \langle i \partial_t u + \Delta u, \partial_t u \rangle_{\C^m} = \partial_t V
\end{equation}
for $j=1,\dots,d$.  We have thus obtained $d+3$ equations involving the fields $u,V$ that do not directly involve the nonlinearity $F$.  

\begin{remark}\label{rapid} The equations \eqref{v1}-\eqref{v4} are closely related to the usual conservation laws for the nonlinear Schr\"odinger equation.  Indeed, if we define the pseudo-stress-energy-tensor
\begin{align*}
T_{00} &\coloneqq \| u \|_{\C^m}^2 \\
T_{0j} = T_{j0} &\coloneqq 2 \langle \partial_{x_j} u, iu \rangle_{\C^m} \\
T_{jk} &\coloneqq 4 \langle \partial_{x_j} u, \partial_{x_k} u\rangle_{\C^m} + \delta_{jk} 2(p-1) V - \delta_{jk} \Delta(\|u\|_{\C^m}^2) 
\end{align*}
for $j=1,\dots,d$, where $\delta_{jk}$ is the Kronecker delta, and also define the energy density
$$ E \coloneqq \frac{1}{2} \langle \partial_{x_j} u, \partial_{x_j} u \rangle_{\C^m} + V$$
(with the usual summation conventions) and energy current
$$ J_j \coloneqq - \langle \partial_{x_j} u, \partial_t u \rangle_{\C^m},$$
for $j=1,\dots,k$, then one can easily use \eqref{v2} to deduce the mass conservation law
$$ \partial_t T_{00} + \partial_{x_j} T_{j0} = 0$$
and similarly use \eqref{v1}, \eqref{v3} to deduce the momentum conservation law
$$ \partial_t T_{0k} + \partial_{x_j} T_{jk} = 0$$
for $k=1,\dots,d$.
From \eqref{v1}, \eqref{v4} we can also obtain the energy conservation law
$$ \partial_t E + \partial_{x_j} J_j = 0.$$
Finally, we can rewrite \eqref{v1} in a way that does not explicitly involve second derivatives of $u$ as
\begin{equation}\label{v0}
 \langle i u_t, u \rangle_{\C^m} + \frac{1}{2} \Delta T_{00} - \langle \partial_{x_j} u, \partial_{x_j} u \rangle_{\C^m} = (p+1) V
\end{equation}
Conversely, if we take \eqref{v0} as a definition of the potential energy density $V$, then the above conservation laws can be used to recover \eqref{v2}, \eqref{v3}, \eqref{v4}.  
\end{remark}

Now assume that $u$ obeys the discrete self-similarity hypothesis \eqref{uts} and is nowhere vanishing.  We recall that the complex projective space $\mathbf{CP}^{m-1}$ is the quotient space
$$ \mathbf{CP}^{m-1} \coloneqq (\C^m \backslash \{0\}) / \C^\times$$
of the manifold\footnote{For the purpose of defining tangent spaces, cotangent spaces, differentials, etc., we will view spaces such as $\C^m \backslash \{0\}$ as real manifolds (of dimension $2m$) rather than complex manifolds, although we will certainly also use the complex structure.} $\C^m \backslash \{0\}$ by the action of the multiplicative complex group $\C^\times = \C \backslash \{0\}$ by scalar multiplication.  Let $\pi\colon \C^m \backslash \{0\} \to \mathbf{CP}^{m-1}$ be the projection map; then $\pi \circ u\colon H_d \to \mathbf{CP}^{m-1}$ is a smooth map which is invariant under the action of $T^\Z$, and thus descends to a smooth map $\theta\colon H_d/T^\Z \to \mathbf{CP}^{m-1}$.  We will derive Theorem \ref{main-2} from

\begin{theorem}[Second reduction]\label{main-3}   There exists a smooth nowhere vanishing function $u\colon H_d \to \C^m \backslash \{0\}$ which is discretely self-similar in the sense of \eqref{uts} for some $\alpha \in \R$, and a smooth function $V\colon H_d \to \R$ such that the defocusing property 
\begin{equation}\label{pos}
V > 0
\end{equation}
and the equations of motion \eqref{v1}, \eqref{v2}, \eqref{v3}, \eqref{v4} hold on all of $H_d$.  Furthermore, the map $\theta\colon H_d/T^\Z \to \mathbf{CP}^{m-1}$ defined above is a smooth embedding, that is to say that it is injective and immersed in the sense that the $d+1$ derivatives $\partial_t \theta(t,x), \partial_{x_1} \theta(t,x), \dots, \partial_{x_d}(t,x)$ are linearly independent in the tangent space of $\mathbf{CP}^{m-1}$ at $\theta(t,x)$ for all $(t,x) \in H_d$.
\end{theorem}

Let us assume Theorem \ref{main-3} for now and see how it implies Theorem \ref{main-2}.  Let $d, p, m, u, V, \theta$ be as in Theorem \ref{main-3}.  To prove Theorem \ref{main-2}, it will suffice to produce a defocusing potential $F: \C^m \to \R$, phase-rotation-invariant and homogeneous of degree $p+1$, such that the identity
\begin{equation}\label{itu}
 i \partial_t u + \Delta u = (\nabla_{\C^m} F)(u)
\end{equation}
holds on all of $H_d$.  Since $u$ never vanishes, we can of course remove the origin $0$ from the domain of $F$, working instead on the  manifold $\C^m \backslash \{0\}$.

We now consider the subset $\Gamma$ of $\C^m \backslash \{0\}$ defined by
$$ \Gamma \coloneqq \{ z u(t,x): (t,x) \in H_d; \quad z \in \C^\times \}$$
or equivalently
$$ \Gamma = \pi^{-1} ( \theta( H_d/T^\Z ) ).$$
This is a $(d+2)$-dimensional $\C^\times$-invariant smooth submanifold (with boundary) of $\C^m \backslash \{0\}$.  The values of the potential $F$ and its gradient $\nabla_{\C^m} F$ on $\Gamma$ are determined by the data $u,V$.  Indeed, if $F$ was phase-rotation-invariant, homogeneous of degree $p+1$, and obeyed \eqref{itu}, then from \eqref{euler}, \eqref{v1} and homogeneity we must have
\begin{equation}\label{faz}
 F( z u(t,x) ) = \frac{|z|^{p+1}}{p+1} V(t,x)
\end{equation}
and
\begin{equation}\label{faz-2}
 (\nabla_{\C^m} F)( z u(t,x) ) = |z|^{p-1} z (i \partial_t u(t,x) + \Delta u(t,x))
\end{equation}
for all $(t,x) \in H_d$ and $z \in \C^\times$.  Conversely, if we can locate a defocusing potential $F$ that is phase-rotation-invariant, homogeneous of degree $p+1$, and obeys the identities \eqref{faz}, \eqref{faz-2} on $\Gamma$, then we of course have \eqref{itu} after specialising \eqref{faz-2} to the case $z=1$.

It remains to construct such an $F$.  In view of the constraints \eqref{faz}, \eqref{faz-2}, it is natural to introduce the functions $F_0: \Gamma \to \R$ and $F_1: \Gamma \to \C^m$ by the formulae
\begin{equation}\label{f0}
 F_0( z u(t,x) ) \coloneqq \frac{|z|^{p+1}}{p+1} V(t,x)
\end{equation}
and
\begin{equation}\label{f1}
 F_1( z u(t,x) ) \coloneqq |z|^{p-1} z (i \partial_t u(t,x) + \Delta u(t,x))
\end{equation}
for all $(t,x) \in H_d$ and $z \in \C^\times$.  As we are assuming $\theta$ to be injective, we see that $z u(t,x) = z' u(t',x')$ occurs if and only if $(t',x') = T^n (t,x)$ and $z' = 2^{\frac{2}{p-1}n} z$ for some integer $n$.  On the other hand, from \eqref{uts}, \eqref{v1} we have
$$ V(T^n(t,x)) = 2^{-\frac{2(p+1)}{p-1}n} V(t,x)$$
and similarly from \eqref{uts} we have
$$ i \partial_t u(T^n(t,x)) + \Delta u(T^n(t,x))) = 2^{-\frac{2p}{p-1}n} (i \partial_t u(t,x) + \Delta u(t,x))$$
and so we see that the functions $F_0, F_1$ are well defined.  As $\theta$ is also a smooth embedding, the functions $F_0,F_1$ are also smooth on $\Gamma$; from \eqref{pos} we know that $F_0$ is strictly positive.  By construction we clearly have the homogeneity relations
\begin{equation}\label{homog}
 F_0( z v ) = |z|^{p+1} F_0(v)
\end{equation}
and
$$ F_1(z v) = |z|^{p-1} z F_1(v)$$
for all $v \in \Gamma$ and $z \in \C^\times$.  Our task is to extend $F_0: \Gamma \to \R$ to a defocusing potential $F: \C^m \backslash \{0\} \to \R$ that continues to obey the relation \eqref{homog}, and such that $\nabla_{\C^m} F$ agrees with $F_1$ on $\Gamma$.

At any given point $z u(t,x)$ of $\Gamma$, the tangent space $T_{z u(t,x)} \Gamma$ is spanned (as a real vector space) by the vectors $z u(t,x)$, $iz u(t,x)$, $z \partial_t u(t,x)$, and $z \partial_{x_j} u(t,x)$ for $j=1,\dots,d$.  From \eqref{v1}, \eqref{v2}, \eqref{v3}, \eqref{v4}, \eqref{f0}, \eqref{f1} and linearity, we conclude the identity
\begin{equation}\label{dfw}
 dF_0(v)(w) = \langle F_1(v), w \rangle_{\C^m}
\end{equation}
for any $v \in \Gamma$ and $w \in T_v \Gamma$, where $dF_0(v) \in T_v^* \Gamma$ is the differential of $F_0$ at $v$, or equivalently $dF_0(v)(w)$ is the directional derivative of $F_0$ at $v$ along the tangent vector $w$.  To put it another way, if we use the inner product $\langle,\rangle_{\C^m}$ to identify $\C^m$ with the dual space $(\C^m)^* = T_v^* \C^m$ (viewed as real vector spaces), then $dF_0(v)$ is the projection of $F_1(v)$ to $T_v^* \Gamma$ (using the dual of the inclusion map from $T_v \Gamma$ to $T_v \C^m$).

It will be convenient to normalise out the homogeneity on $F, F_0, F_1$.  Define the normalised functions $F^{(1)}_0: \Gamma \to \R$ and $F^{(1)}_1: \Gamma \to \C^m$ by the formulae
$$ F^{(1)}_0(v) \coloneqq \|v\|_{\C^m}^{-p-1} F_0(v)$$
and
$$ F^{(1)}_1(v) \coloneqq \|v\|_{\C^m}^{-p-1} F_1(v) - (p+1) \|v\|_{\C^m}^{-p-3} F_0(v) v.$$
Then $F^{(1)}_0, F^{(1)}_1$ are smooth, with the homogeneity relations
$$
F^{(1)}_0( z v ) = F^{(1)}_0(v)
$$
and
$$ F^{(1)}_1(z v) = |z|^{-2} z F^{(1)}_1(v)$$
for all $v \in \Gamma$ and $z \in \C^\times$; also, from \eqref{dfw} and the product rule we see that
\begin{equation}\label{af}
 dF^{(1)}_0(v)(w) = \langle F^{(1)}_1(v), w \rangle_{\C^m}
\end{equation}
for any $v \in \Gamma$ and $w \in T_v \Gamma$.  Finally, $F^{(1)}_0$ is clearly everywhere positive.

Since $F^{(1)}_0: \Gamma \to \R$ is invariant under the action of $\C^\times$, it descends to a smooth positive function $F^{(2)}_0: \theta(H_d/T^\Z) \to \R$ on the quotient space $\Gamma/\C^\times = \theta(H_d/T^\Z)$, thus
$$ F^{(2)}_0( \pi(v) ) = F^{(1)}_0(v)$$
for all $v \in \Gamma$.  For any $v \in \Gamma$, we define the covector $F^{(2)}_1(\pi(v)) \in T^*_{\pi(v)} \mathbf{CP}^{m-1}$ by the formula
$$ F^{(2)}_1(\pi(v))(\pi_{*,v}(w)) \coloneqq \langle F^{(1)}_1(v), w \rangle_{\C^m}$$
for all $w \in T_v \C^m \equiv \C^m$, where $\pi_{*,v}: T_v \C^m \to T_{\pi(v)} \mathbf{CP}^{m-1}$ is the projection map.  Note from \eqref{af} and the $\C^\times$-invariance of $F^{(1)}_0$ that $F^{(1)}_1(v)$ is orthogonal to the kernel of $\pi_{*,v}$; this and the homogeneity of $F^{(1)}_0, F^{(1)}_1$ ensure that $F^{(2)}_1$ is well defined and smooth on $\pi(\Gamma) = \theta(H_d/T^\Z)$. From \eqref{af} we see that
\begin{equation}\label{foop}
 dF^{(2)}_0(\tilde v)(\tilde w) = F^{(2)}_1(\tilde v)(\tilde w)
\end{equation}
for all $\tilde v \in \theta(H_d/T^\Z)$ and $\tilde w \in T_{\tilde v} \theta(H_d/T^\Z)$; in other words, $F^{(2)}_1$ agrees with $dF^{(2)}_0$ at any point $\tilde v$ on the compact manifold with boundary $\theta(H_d/T^\Z)$, after restricting to the tangent space $T_{\tilde v} \theta(H_d/T^\Z)$ of that manifold.

One can view $H_d/T^\Z$ as a smooth compact submanifold (with smooth boundary) of $(\R \times \R^d \backslash \{0,0\})/T^\Z$.  The function $\theta: H_d/T^\Z \to \mathbf{CP}^{m-1}$ can be extended smoothly to an open neighbourhood of this submanifold using a classical theorem of Seeley \cite{seeley}; the embedded copy $\theta(H_d/T^\Z)$ of $H_d/T^\Z$ in $\mathbf{CP}^{m-1}$ can then similarly be extended to a slightly larger open manifold of the same dimension $d+1$.  A further application of Seeley's theorem allows one to smoothly extend $F^{(2)}_0$ to this enlargement of $\theta(H_d/T^\Z)$.  Using this extension as well as \eqref{foop} and Fermi normal coordinates (using for instance the Fubini-Study metric on $\mathbf{CP}^{m-1}$), one can then obtain a smooth extension $F^{(3)}_0$ of $F^{(2)}_0$ to an open neighbourhood $U$ of the embedded copy $\theta(H_d/T^\Z)$ of $H_d/T^\Z$ in $\mathbf{CP}^{m-1}$ in such a fashion that $dF^{(3)}_0 = F^{(2)}_1$ on $\theta(H_d/T^\Z)$.  By shrinking $U$ if necessary one can ensure that $F^{(3)}_0$ is positive on all of $U$.  If one then sets $F^{(4)}_0: \mathbf{CP}^{m-1} \to \R$ to be the function defined by
$$ F^{(4)}_0 \coloneqq \varphi F^{(3)}_0 + (1-\varphi)$$
for some smooth cutoff $\varphi: \mathbf{CP}^{m-1} \to [0,1]$ that is supported on $U$ that equals $1$ on a neighbourhood of $\theta(H_d/T^\Z)$, we see that $F^{(4)}_0: \mathbf{CP}^{m-1} \to \R$ is a positive smooth extension of $F^{(2)}_0$ such that $dF^{(4)}_0 = F^{(2)}_1$ on $\theta(H_d/T^\Z)$.

If we now set $F: \C^m \backslash \{0\} \to \R$ to be the function
$$ F(v) \coloneqq \|v\|_{\C^m}^{p+1} F^{(4)}_0(\pi(v))$$
then $F$ is a defocusing potential that is phase-rotation-invariant and homogeneous of degree $p+1$.  By construction, $F$ agrees with $F_0$ on $\Gamma$, and 
$$ d(\|v\|_{\C^m}^{-p-1} F)(v)(w) = \langle F^{(1)}_1(v), w \rangle_{\C^m}$$
for all $v \in \Gamma$ and $w \in T_v \C^m \equiv \C^m$.  By the product rule and construction of $F^{(1)}_1$, this implies that
$$ dF(v)(w) = \langle F_1(v), w \rangle_{\C^m}$$
for all $v \in \Gamma$ and $w \in T_v \C^m$, and thus
$$ \nabla_{\C^m} F = F_1$$
on $\Gamma$, as desired.

It remains to establish Theorem \ref{main-3}.
This will be the focus of the remaining sections of the paper.

\section{Eliminating the field}

In view of Remark \ref{rapid}, the constraints \eqref{v1}-\eqref{v4} that need to be satisfied in Theorem \ref{main-3} can be expressed in terms of the pseudo-stress-energy tensor $T_{00}, T_{0j}, T_{jk}$, as well as the energy density $E$ and the energy current $J_j$.  These quantities in turn depend linearly on the potential energy density $V$ and the components of the $(2d+4) \times (2d+4)$ Gram-type matrix $G[u,u]$, where we define
\begin{equation}\label{gram-type}
G[u,v] \coloneqq ( \langle D_1 u, D_2 v \rangle_{\C^m} )_{D_1,D_2 \in {\mathcal D}},
\end{equation}
for any smooth $u,v: H_d \to \C^m$, where ${\mathcal D}$ is the finite set of differential operators
$$ {\mathcal D} \coloneqq \{ 1, i, \partial_{x_1}, \dots, \partial_{x_d}, \partial_t, i \partial_{x_1}, \dots, i \partial_{x_d}, i \partial_t \}.$$
For our later arguments, it will be crucial to observe that the component $\langle \partial_t u, \partial_t u \rangle_{\C^m} = \langle i \partial_t u, i \partial_t u \rangle_{\C^m}$ of the  Gram-type matrix $G[u,u]$ is \emph{not} used to determine the quantities $T_{00}, T_{0j}, T_{jk}, E, J_j$, and in particular will be allowed to be extremely large compared to the other components of this matrix.

\begin{table}[ht]
\caption{The parabolic order $\operatorname{ord}(D)$ of various differential operators $D$ (or formal differential operators) used in this paper.  Some of the operators in this table will only be defined in subsequent sections.}\label{order}
\begin{tabular}{|l|l|}
\hline 
Operator $D$ & Parabolic order $\operatorname{ord}(D)$ \\
\hline 
$1, i$ & $0$ \\
$\partial_{x_j}$, $i \partial_{x_j}$, $\partial_r$, $i\partial_r$, $\partial_\omega$ & $1$ \\
$\partial_t$, $i\partial_t$ & $2$\\
\hline
\end{tabular}
\end{table}

As in \cite{tao-nlw}, the strategy of proof of Theorem \ref{main-3} will be to eliminate the role of the field $u$ by reformulating the problem in terms of $V$ and the Gram-type matrix $G[u,u]$ (or on closely related quantities such as $T_{00}, T_{0j}, T_{jk}, E, J_j$).    To do this, it is natural to ask what constraints a $(2d+4) \times (2d+4)$ matrix-valued function $G$ on $H_d$ has to obey in order to be expressible as a Gram-type matrix $G[u,u]$ of a smooth field $u: H_d \to \C^m$ obeying the homogeneity condition \eqref{uts}.  Certainly we will have homogeneity relations of the form
$$ \langle D_1 u(4t, 2x), D_2 u(4t, 2x) \rangle_{\C^m} = 2^{-\frac{4}{p-1}-\operatorname{ord}(D_1)-\operatorname{ord}(D_2)} \langle D_1 u(t, x), D_2 u(t, x) \rangle_{\C^m}$$
where the \emph{parabolic order} $\operatorname{ord}(D)$ of a differential operator $D \in {\mathcal D}$ is defined by Table \ref{order}.  Also, it is clear that the matrix $G[u,u]$ is real symmetric and positive semi-definite, with the additional constraint
\begin{equation}\label{da}
 \langle i D_1 u, i D_2 u \rangle_{\C^m} = \langle D_1 u, D_2 u \rangle_{\C^m}
\end{equation}
for $D_1,D_2 = 1,\partial_{x_1},\dots, \partial_{x_d}, \partial_t$. From the product rule we also have the additional constraints
\begin{equation}\label{da-2}
 \langle u, D_1 u \rangle_{\C^m} = \frac{1}{2} D_1 \langle u, u \rangle_{\C^m}
\end{equation}
and
\begin{equation}\label{bongo}
 D_1 \langle u, iD_2 u \rangle_{\C^m} - D_2 \langle u, iD_1 u \rangle_{\C^m} = 2 \langle D_1 u, iD_2 u \rangle_{\C^m}
\end{equation}
for $D_1,D_2 = \partial_{x_1},\dots, \partial_{x_d}, \partial_t$.
Finally we have
\begin{equation}\label{bango}
 \langle i D_1 u, D_2 u \rangle_{\C^m} = -\langle D_1 u, iD_2 u \rangle_{\C^m}.
\end{equation}
for $D_1,D_2 = 1, \partial_{x_1},\dots, \partial_{x_d}, \partial_t$.
One could then hope that these were essentially the complete list of constraints on the Gram-type matrix $G[u,u]$.  In the real case, in which $u$ takes values in the real Euclidean space $\R^m$ (with the usual inner product $\langle,\rangle_{\R^m}$), and the set of operators ${\mathcal D}$ is reduced to the $d+2$ operators
$$ D_{\R} \coloneqq \{ 1, \partial_{x_1},\dots,\partial_{x_d}, \partial_t \},$$
then one can obtain such a claim using the Nash embedding theorem \cite{nash}:

\begin{proposition}\label{gad}  Let $(G_{D_1,D_2})_{D_1,D_2 \in {\mathcal D}_\R}$ be a $(d+2) \times (d+2)$ matrix of smooth functions $G_{D_1,D_2}: H_d \to \R$ obeying the following hypotheses:
\begin{itemize}
\item[(i)] For each $(t,x) \in H_d$, the matrix $(G_{D_1,D_2}(t,x))_{D_1,D_2 \in {\mathcal D}_\R}$ is symmetric and strictly positive definite.
\item[(ii)] One has the scaling law
\begin{equation}\label{scaling}
 G_{D_1,D_2}(4t, 2x) = 2^{-\frac{4}{p_1} - \operatorname{ord}(D_1) - \operatorname{ord}(D_2)} G_{D_1,D_2}(t,x)
\end{equation}
for all $D_1,D_2 \in {\mathcal D}_\R$ and $(t,x) \in H_d$.
\item[(iii)] We have the identity
\begin{equation}\label{div}
G_{1,D_1}(t,x) = G_{D_1,1}(t,x) = \frac{1}{2} G_{1,1}(t,x)
\end{equation}
for all $(t,x) \in H_d$ and $D_1 \in {\mathcal D}_\R \backslash \{1\}$.
\end{itemize}
Suppose also that $m$ is an integer that is sufficiently large depending on $d$.  Then there exists a smooth function $u: H_d \to \R^m$ that is nowhere vanishing and obeying the discrete self-similarity \eqref{uts} with $\alpha=0$, such that
\begin{equation}\label{g}
 G_{D_1,D_2}(t,x) = \langle D_1 u(t,x), D_2 u(t,x) \rangle_{\R^m}
\end{equation}
for all $D_1,D_2\in {\mathcal D}_\R$ and $(t,x) \in H_d$.  Furthermore, the function $\theta: (t,x) \mapsto \frac{u(t,x)}{\|u(t,x)\|_{\R^m}}$, when descended to the quotient space $H_d/T^\Z$, is a smooth embedding.
\end{proposition}

\begin{proof}  Observe from the chain and quotient rules that if $u$ is smooth and obeys \eqref{g}, then $u$ is nowhere vanishing (since $G_{1,1}$ is strictly positive) and  direction map $\theta: (t,x) \mapsto \frac{u(t,x)}{\|u(t,x)\|_{\R^m}}$ obeys the identity
\begin{equation}\label{geo}
  g_{D_1,D_2}(t,x) = \langle D_1 \theta(t,x), D_2 \theta(t,x) \rangle_{\R^m}
\end{equation}
for $(t,x) \in H_d$ and $D_1,D_2 \in {\mathcal D}_\R \backslash \{1\}$, where the functions $g_{D_1,D_2}: H_d \to \R$ are given by the formula
$$ g_{D_1,D_2} \coloneqq \frac{G_{D_1,D_2}}{G_{1,1}} - \frac{G_{1,D_1} G_{1,D_2}}{G_{1,1}^2}.$$
Motivated by this, our strategy will be to construct the direction map $\theta$ obeying \eqref{geo} first, and use this to then reconstruct $u$.

Since $G_{1,1}$ is strictly positive, the $(d+1) \times (d+1)$-matrix $g = (g_{D_1,D_2})_{D_1,D_2 \in {\mathcal D}_\R \backslash \{1\}}$ is smooth and symmetric; from the hypothesis (ii), the matrix $g$ is $T^\Z$-invariant, and can thus (by slight abuse of notation) be viewed as a function on the quotient space $H_d/T^\Z$.  From the identity
$$ \sum_{D_1,D_2 \in {\mathcal D}_\R \backslash \{1\}} g_{D_1,D_2} a_{D_1} a_{D_2} = \sum_{D_1,D_2 \in {\mathcal D}} G_{D_1,D_2} b_{D_1} b_{D_2}$$
for all reals $a_D, D \in {\mathcal D}_\R \backslash \{1\}$, where $b_D \coloneqq \frac{a_D}{G_{1,1}^{1/2}}$ and $b_1 \coloneqq -\frac{\sum_{D \in {\mathcal D}_\R \backslash \{1\}} a_D G_{1,D}}{G_{1,1}^{3/2}}$, and the hypothesis (i), we see that the matrix $g$ is strictly positive.  Thus $(H_d/T^\Z, g)$ can be viewed as a smooth compact $d+1$-dimensional Riemannian manifold with smooth boundary.  If $m_0$ is a large enough integer, we may then apply the Nash embedding theorem (see \cite{nash}, \cite{gunther}) and find a smooth isometric embedding of $(H_d/T^\Z,g)$ into a Euclidean space $\R^{m_0}$.  As observed in \cite[\S 4]{tao-nlw}, any compact region of $\R^{m_0}$ may be isometrically emebdded into the unit sphere $S^{m-1}$ of $\R^m$ if $m \geq 2m_0+2$.  Thus, for $m$ large enough, we may find an isometric embedding $\theta: H_d/T^\Z \to S^{m-1}$ of $(H_d/T^\Z,g)$ into unit sphere $S^{m-1}$ of $\R^m$ (with the induced Euclidean metric); thus $\theta$ is a smooth embedding and obeys the identity \eqref{geo} (after lifting up from $H_d/T^\Z$ to $H_d$).  If one then defines the function $u: H_d \to \R^m$ by the formula
$$ u(t,x) \coloneqq \theta(t,x) G_{1,1}(t,x)^{1/2}$$
then $u$ is smooth, nowhere vanishing, and obeys \eqref{uts} with $\alpha=0$, and from a routine calculation using the product and chain rules (as well as hypothesis (iii)) we have the required identity \eqref{g} for all $D_1,D_2 \in {\mathcal D}_\R \backslash \{1\}$; it is also clear that \eqref{g} holds when $D_1=D_2=1$.  Differentiating the latter identity in space or time using (iii) and the product rule, we obtain the remaining cases of \eqref{g}, and the claim follows.
\end{proof}

One could use the literature on the Nash embedding theorem to extract an explicit value of $m$ as a function of $d$ in the above proposition, but we will not seek to optimise this value here.  For future reference, we observe that the above argument also gives the variant of Proposition \ref{gad} in which the half-space $H_d$ is replaced by the punctured spacetime $\R \times \R^d \backslash \{(0,0)\}$.

In view of the above proposition, one could conjecture a complex analogue of the proposition, in which one uses ${\mathcal D}$ in place of ${\mathcal D}_\R$ and $\C^m$ in place of $\R^m$, with the additional constraints \eqref{bongo}, \eqref{bango} imposed.  This conjecture may well be false in full generality (note that the complex version of the Nash embedding theorem is false, for instance Liouville's theorem prevents compact complex manifolds without boundary from being holomorphically embedded into $\C^m$).  Nevertheless we could adapt the \emph{proof} of the Nash embedding theorem to obtain a partial complex analogue of Proposition \ref{gad}, in which we do not seek to control the $\langle \partial_t u, \partial_t u \rangle_{\C^m}$ component of the Gram-like matrix \eqref{gram-type}, and in which we also have an additional curl-free property of a certain combination of components of this matrix.  While this falls well short of a true complex version of Proposition \ref{gad}, it will suffice for our purposes.  Specifically, we have

\begin{proposition}\label{gadc}  Let $G = (G_{D_1,D_2})_{D_1,D_2 \in {\mathcal D}}$ be a $(2d+4) \times (2d+4)$ matrix of smooth functions $G_{D_1,D_2}: H_d \to \R$ obeying the following hypotheses:
\begin{itemize}
\item[(i)] For each $(t,x) \in H_d$, the matrix $(G_{D_1,D_2}(t,x))_{D_1,D_2 \in {\mathcal D}}$ is symmetric and strictly positive definite.
\item[(ii)] One has the scaling law \eqref{scaling} for all $D_1,D_2 \in {\mathcal D}$ and $(t,x) \in H_d$.
\item[(iii)] We have the identity \eqref{div}, as well as the additional identities
\begin{align}
D_1 G_{1,iD_2}(t,x) - D_2 G_{1,iD_1}(t,x) &= 2G_{D_1,iD_2}(t,x) \label{div-3}
\end{align}
for all $(t,x) \in H_d$ and $D_1,D_2 \in {\mathcal D}_\R \backslash \{1\}$, and
\begin{align}
G_{iD_1,iD_2}(t,x) &= G_{D_1,D_2}(t,x) \label{div-2} \\
G_{D_1,iD_2}(t,x) &= - G_{D_2,iD_1}(t,x) \label{div-4}
\end{align}
for all $(t,x) \in H_d$ and $D_1,D_2 \in {\mathcal D}_\R$ (in particular we have $G_{D_1,iD_1} = 0$).
\item[(iv)]  The vector field $\left(\frac{G_{1, i \partial_{x_j}}}{G_{1,1}}\right)_{j=1}^d$ is curl-free, that is to say
$$ \partial_{x_k} \frac{G_{1, i \partial_{x_j}}}{G_{1,1}}(t,x) - \partial_{x_j} \frac{G_{1, i \partial_{x_k}}}{G_{1,1}}(t,x) = 0 $$
for all $j,k \in 1,\dots,d$ and $(t,x) \in H_d$.
\end{itemize}
Suppose also that $m$ is an integer that is sufficiently large depending on $d$.  Then there exists a smooth function $u: H_d \to \C^m$ that is nowhere vanishing and obeying the discrete self-similarity \eqref{uts} for some $\alpha \in \R$, such that
\begin{equation}\label{gc}
 G_{D_1,D_2}(t,x) = \langle D_1 u(t,x), D_2 u(t,x) \rangle_{\C^m}
\end{equation}
for all $(t,x) \in H_d$ and all $D_1,D_2\in {\mathcal D}$ other than $(D_1,D_2) = (\partial_t, \partial_t), (i \partial_t, i \partial_t)$.
Furthermore, the function $\theta: H_d/T^\Z \to \mathbf{CP}^{m-1}$, formed by descending the map $\pi \circ u: H_d \to \mathbf{CP}^{m-1}$ to $H_d/T^\Z$, is a smooth embedding.
\end{proposition}

\begin{remark}  The condition (iv) differs from the other hypotheses in that it is not necessitated by the conclusions of this theorem.  However, this condition turns out to be convenient in the proof of Proposition \ref{gadc}, as it will allow us to ``gauge transform away'' the $G_{1,i\partial_{x_j}}$ components; see Proposition \ref{gauge}.  However, this additional constraint (iv) will end up not being harmful to our argument, because we will eventually reduce to the case where the matrix $G$ is spherically symmetric in the sense that $G_{1,1}(t,x) = g(t,|x|)$ and $G_{1,i\partial_{x_j}}(t,x) = \frac{x_j}{|x|} h(t,|x|)$ for some functions $g,h$, in which case the condition (iv) is automatically satisfied.
\end{remark}

The proof of Proposition \ref{gadc} is rather lengthy, and the methods of proof (based on the proof of the Nash embedding theorem) are not used elsewhere in the paper.  We therefore defer this proof to Appendix \ref{gad-app}.  
Combining Proposition \ref{gadc} with Remark \ref{rapid}, we thus see that Theorem \ref{main-3} will now follow from the following proposition in which the field $u$ has been eliminated.

\begin{theorem}[Third reduction]\label{main-4}   There exists a smooth $(2d+3) \times (2d+3)$ matrix $G = (G_{D_1,D_2})_{D_1,D_2 \in {\mathcal D}}$ of smooth functions $G_{D_1,D_2}: H_d \to \R$ and an additional smooth function $V: H_d \to \R$ obeying the following properties.
\begin{itemize}
\item[(i)] For each $(t,x) \in H_d$, the matrix $(G_{D_1,D_2}(t,x))_{D_1,D_2 \in {\mathcal D}}$ is symmetric and strictly positive definite.
\item[(ii)] One has the scaling law \eqref{scaling} for all $D_1,D_2 \in {\mathcal D}$ and $(t,x) \in H_d$.
\item[(iii)] We have the identities \eqref{div}, \eqref{div-3}
for all $(t,x) \in H_d$ and $D_1,D_2 \in {\mathcal D}_\R \backslash \{1\}$, and \eqref{div-2}, \eqref{div-4} for all $(t,x) \in H_d$ and $D_1,D_2 \in {\mathcal D}_\R$.
\item[(iv)]  The vector $\left(\frac{G_{1, i \partial_{x_j}}}{G_{1,1}}\right)_{j=1}^d$ is curl-free on $H_d$.
\item[(v)]  One has the defocusing property \eqref{pos}.
\item[(vi)]  If one defines the pseudo-stress-energy tensor
\begin{align*}
T_{00} &\coloneqq G_{1,1} \\
T_{0j} = T_{j0} &\coloneqq - 2 G_{i\partial_{x_j}, 1} \\
T_{jk} &\coloneqq 4 G_{\partial_{x_j}, \partial_{x_k}} + \delta_{jk} 2(p-1) V - \delta_{jk} \Delta G_{1,1} 
\end{align*}
for $j,k=1,\dots,d$, as well as the energy density
$$ E \coloneqq \frac{1}{2} G_{\partial_{x_j},\partial_{x_j}} + V $$
(with the usual summation conventions) and energy current
$$ J_j \coloneqq - G_{\partial_{x_j}, \partial_t}$$
for $j=1,\dots,d$, then one has the identity
\begin{equation}\label{potential}
G_{i\partial_t, 1} + \frac{1}{2} \Delta G_{1,1} - G_{\partial_{x_j}, \partial_{x_j}} = (p+1) V 
\end{equation}
and the conservation laws
\begin{align}
\partial_t T_{00} + \partial_{x_j} T_{j0} &= 0 \label{ping-2} \\
\partial_t T_{0k} + \partial_{x_j} T_{jk} &= 0 \label{ping-3} \\
\partial_t E + \partial_{x_j} J_j &= 0 \label{ping-4}
\end{align}
for $k=1,\dots,d$.
\end{itemize}
\end{theorem}

Note carefully that the components $G_{\partial_t, \partial_t}$, $G_{i\partial_t, i \partial_t}$ of $G$ are not used in the hypotheses (ii)-(vi) above (and only influence (i) through the requirement of being positive definite).  The scaling law \eqref{scaling} only applies directly to the components of $G$, but from the potential identity \eqref{potential} we see that we also have a corresponding scaling law
$$ V(T(t, x)) = 2^{-\frac{4}{p-1}-2} V(t,x)$$
for the potential $V$.

It remains to prove Theorem \ref{main-4}.  This will be the objective of the remaining sections of the paper.

\section{Spherical symmetry and scale invariance}

At first glance, the hypotheses required in Theorem \ref{main-4} of the unknown fields $G, V$ appear to be more complicated than those in previous formulations of the problem, such as Theorem \ref{main-2}.  However, there is one notable way in which the hypotheses of Theorem \ref{main-4} are much better than those in previous formulations: as they only involve linear equalities and inequalities (as well as claims of positive definiteness), the constraints determine a \emph{convex} set in the phase space of possible values for the fields $G, V$.  This can be compared to previous formulations in which the conditions on the unknown field $u$ were quadratic or otherwise nonlinear in nature.  

One consequence of this convexity is that if there is at least one solution $G, V$ to Theorem \ref{main-4}, then there is a solution $G, V$ which is spherically symmetric in a tensorial sense, or more precisely that
\begin{align*}
G_{1,1}(t,x) &= g_{1,1}(t,|x|) \\
G_{\partial_{x_j},\partial_{x_k}}(t,x) &= \frac{x_j x_k}{|x|^2} g_{\partial_r,\partial_r}(t,|x|) + \left(\delta_{jk} - \frac{x_j x_k}{|x|^2}\right) g_{\partial_\omega \partial_\omega}(t,|x|) \\
G_{\partial_{x_j},\partial_t}(t,x) &= \frac{x_j}{|x|} g_{\partial_r,\partial_t}(t,|x|) \\
G_{1,i\partial_{x_j}}(t,x) &= \frac{x_j}{|x|} g_{1,i\partial_r}(t,|x|) \\
G_{1,i\partial_t}(t,x) &= g_{1,i\partial_t}(t,|x|) \\
V(t,x) &= v(t,|x|)
\end{align*}
for some functions $g_{11}, g_{\partial_r,\partial_r}, g_{\partial_\omega, \partial_\omega}, g_{\partial_r, \partial_t}, g_{1,i\partial_r}, g_{1,i\partial_t}$; we omit here for brevity some analogous constraints on the remaining components of $G$ which are either constrained completely by the fields already listed, or (in the case of $G_{\partial_t,\partial_t}$ and $G_{i\partial_t,i\partial_t}$) are not relevant for the theorem.  This is basically because we can average the original solution $G,V$ over rotations (letting the orthogonal group $SO(d)$ act on tensors in an appropriate fashion) and use convexity to obtain a spherically symmetric solution.  For similar reasons (averaging now over dilations rather than rotations), one can assume without loss of generality that the solution $G,V$ is not only \emph{discretely} self-similar in the sense of \eqref{scaling}, but is in fact \emph{continuously} self-similar in the sense that the identity
$$
 G_{D_1,D_2}(\lambda^2 t, \lambda x) = \lambda^{-\frac{4}{p_1} - d_1 - d_2} G_{D_1,D_2}(t,x)
$$ 
holds for all $D_1,D_2 \in {\mathcal D}$, $(t,x) \in H_d$, and $\lambda > 0$, where $d_1,d_2$ denotes the degrees of $D_1,D_2$ respectively as before.

\begin{remark}  Note that the reduction to spherical symmetry of the fields $G, V$ in Theorem \ref{main-4} does \emph{not} mean that we can reduce to spherically symmetric $u$ in the original formulation (Theorem \ref{main}) of the results in this paper, because it is possible for a non-spherically-symmetric field $u$ to have a spherically symmetric Gram matrix (e.g. if $d=2$ and $u$ is equivariant rather than invariant with respect to rotations).  Indeed, a spherically symmetric $u$ would have a vanishing $g_{\partial_\omega,\partial_\omega}$ field, whereas in our construction we will insist instead that this field be positive.  Similarly, we cannot necessarily reduce to solutions in Theorem \ref{main} that are continuously self-similar.
\end{remark}

We now perform these reductions by showing that Theorem \ref{main-4} is a consequence of (and is in fact equivalent to) the following spherically symmetric, continuously self-similar version.  Recall that the domain $H_1$ is given by \eqref{h1-def}.
It will be convenient to make the following definition: we say that a function $F: H_1 \to \R$ \emph{scales like} $\rho^\alpha$ for some $\alpha \in \R$ if one has
\begin{equation}\label{sab}
 F(\lambda^2 t, \lambda r) = \lambda^\alpha F(t,r)
\end{equation}
for all $(t,x) \in H_1$ and $\lambda > 0$.  Here we recall $\rho:H_1 \to \R$ was defined in \eqref{rho1-def}. We also note the following ``factor theorem'' on $H_1$: if $F: H_1 \to \R$ is a smooth function that vanishes on the time axis $r=0$, then the quotient $F(t,r)/r$ has a removable singularity at $r=0$, in the sense that there is a smooth function $G: H_1 \to \R$ such that $F(t,r) = r G(t,r)$ for all $(t,r) \in H_1$ (so that $G$ can be viewed as the smooth completion of $F(t,r)/r$.  Indeed, from the fundamental theorem of calculus, one can take
$$ G(t,r) \coloneqq \int_0^1 (\partial_r F)(t, sr)\ ds.$$
Iterating this, we see that if $k$ is a positive integer, and $F: H_1 \to \R$ is smooth and vanishes to order $k$ on the time axis $r=0$ (in the sense that $F(t,r) = O(r^k)$ as $r \to 0$ for any fixed $t$), then $F/r^k$ has a removable singularity on the time axis.

\begin{table}[ht]
\caption{The scaling exponent of various fields on $H_1$ used in this paper, as well as their parity in $r$ (even or odd).  Some of the fields in this table will only be defined in subsequent sections.}\label{scalingexp}
\begin{tabular}{|l|l|l|}
\hline 
Exponent & Fields & Parity\\
\hline 
$2$ & $t$ & even\\
$1$ & $\rho$ & even \\
$1$ & $r$ & odd \\
\hline
$-\frac{4}{p-1}+d-4$ & $S_1, S_2$ & same as $d$\\
$-\frac{4}{p-1}$ & $g_{1,1}$, $T_{00}$, $W$ & even \\
$-\frac{4}{p-1}-1$ & $g_{1,i\partial_r}$, $T_{0r}$ & odd \\
$-\frac{4}{p-1}-2$ & $g_{\partial_r,\partial_r}$, $g_{\partial_\omega,\partial_\omega}$, $g_{1,i\partial_t}$, $v$, $T_{rr}$, $T_{\omega\omega}$, $E$, $Z$ & even \\
$-\frac{4}{p-1}-3$ & $g_{\partial_r, \partial_t}$, $J_r$ & odd \\
\hline
\end{tabular}
\end{table}

\begin{theorem}[Fourth reduction]\label{main-5}   There exist smooth fields $g_{1,1}, g_{\partial_r,\partial_r}, g_{\partial_\omega,\partial_\omega}, g_{\partial_r,\partial_t}, g_{1,i\partial_r}, g_{1,i\partial_t}, v: H_1 \to \R$ obeying the following properties:
\begin{itemize}
\item[(i)] One has the ``positive definite'' inequalities
\begin{equation}\label{gstorm}
 \left(\frac{1}{2} \partial_r g_{1,1}\right)^2 + g_{1,i\partial_r}^2 < g_{1,1} g_{\partial_r,\partial_r}
\end{equation}
and
\begin{equation}\label{gstorm-2}
g_{1,1}, g_{\partial_r,\partial_r}, g_{\partial_\omega,\partial_\omega} > 0
\end{equation}
pointwise on $H_1$.
\item[(ii)] For each $(D_1,D_2) = (1,1), (\partial_r,\partial_r), (\partial_\omega,\partial_\omega), (\partial_r,\partial_t), (1, i\partial_r), (1, i\partial_t)$, $g_{D_1,D_2}$ scales like $\rho^{-\frac{4}{p-1}-\operatorname{ord}(D_1)-\operatorname{ord}(D_2)}$, where we recall the parabolic order $\operatorname{ord}(D)$ of a differential operator $D \in \{ 1, \partial_r, i \partial_r, \partial_\omega, \partial_t, i\partial_t\}$ is given by Table \ref{order}.  Similarly, we require that $v$ scales like $\rho^{-\frac{4}{p-1}-2}$.  See Table \ref{scalingexp} for a summary of these scaling requirements.
\item[(v)]  One has the defocusing property $v>0$ pointwise on $H_1$.
\item[(vi)]  If one defines the mass density
$$ T_{00} \coloneqq g_{1,1}$$
the radial momentum density
$$ T_{0r} \coloneqq -2 g_{1,i\partial_r}$$
the radial stress
\begin{equation}\label{trr-def}
 T_{rr} \coloneqq 4 g_{\partial_r,\partial_r} + 2(p-1) v - \left(\partial_r^2 + \frac{d-1}{r} \partial_r\right) g_{1,1}
\end{equation}
the angular stress
\begin{equation}\label{too-def}
 T_{\omega\omega} \coloneqq 4 g_{\partial_\omega,\partial_\omega} + 2(p-1) v - \left(\partial_r^2 + \frac{d-1}{r} \partial_r\right) g_{1,1}
\end{equation}
the energy density
\begin{equation}\label{E-def}
 E \coloneqq \frac{1}{2} g_{\partial_r,\partial_r} + \frac{d-1}{2} g_{\partial_\omega,\partial_\omega} + v
\end{equation}
and radial energy current
\begin{equation}\label{jr-def}
 J_r \coloneqq - g_{\partial_r, \partial_t} 
\end{equation}
then one has the potential identity
\begin{equation}\label{potential-r}
g_{1,i\partial_t} + \frac{1}{2} \left(\partial_r^2 + \frac{d-1}{r} \partial_r\right) g_{1,1} - g_{\partial_r, \partial_r} - (d-1) g_{\partial_\omega,\partial_\omega} = (p+1) v 
\end{equation}
and the conservation laws
\begin{align}
\partial_t T_{00} + \left(\partial_r + \frac{d-1}{r}\right) T_{0r} &= 0 \label{ping-2r} \\
\partial_t T_{0r} + \left(\partial_r + \frac{d-1}{r}\right) T_{rr} - \frac{d-1}{r} T_{\omega \omega} &= 0 \label{ping-3r} \\
\partial_t E + \left(\partial_r + \frac{d-1}{r}\right) J_r &= 0. \label{ping-4r}
\end{align}
with a removable singularity at $r=0$ (see Remark \ref{remov} below).
\item[(vii)]  The functions $g_{1,1}, g_{\partial_r, \partial_r}, g_{\partial_\omega, \partial_\omega}, g_{1,i\partial_t}, v$ are even in $r$, while $g_{\partial_r, \partial_t}, g_{1,i\partial_r}$ are odd in $r$ (see Table \ref{scalingexp}).  Furthermore, $g_{\partial_r,\partial_r} - g_{\partial_\omega,\partial_\omega}$ vanishes on the time axis $r=0$.
\end{itemize}
\end{theorem}

\begin{remark}\label{remov}  At first glance, the quantities $T_{rr}, T_{\omega,\omega}$, as well as the equations \eqref{potential-r}, \eqref{ping-2r}, \eqref{ping-3r}, \eqref{ping-4r} appear to have singularities on the time axis $r=0$, due to the factors of $\frac{1}{r}$.  However, these factors are removable due to the symmetry hypotheses in (vii).  Indeed, for each fixed time $t$, one can Taylor expand the even function $g_{1,1}$ as $g_{1,1} = a + br^2 + \dots$, and then one sees that the quantity $\left(\partial_r^2 + \frac{d-1}{r} \partial_r\right) g_{1,1}$ extends smoothly across $r=0$ (which is unsurprising given that this operator is nothing more than the Laplacian on spherically symmetric functions).  Thus $T_{rr}$ and $T_{\omega\omega}$ extend smoothly to $r=0$.  Also, the difference $T_{rr} - T_{\omega\omega}$ vanishes at $r=0$, so the singularity for \eqref{ping-3r} is also removable.  Finally, the functions $T_{0r}, J_r$ are odd in $r$ and so the singularity in \eqref{ping-2r}, \eqref{ping-4r} is also removable.
\end{remark}

\begin{remark}  It is not difficult to use Table \ref{scalingexp} to perform a ``dimensional analysis'' to verify that the requiements in Theorem \ref{main-5}(vi) are consistent with the scaling and parity requirements in Theorem \ref{main-5}(ii), (vii).  One can use the continuous self-similarity (ii) to eliminate the time variable, so that Theorem \ref{main-5} becomes an ODE assertion about the existence of some scalar functions on $\R$.  However, it will be convenient (and more physically natural) to continue to work with both the time variable $t$ and the spatial variable $r$, rather than with just one of these variables.  It is also worth noting that the components $g_{\partial_r,\partial_t}$ and $g_{1,i\partial_t}$ have only a small role to play in the above theorem, basically appearing only in the constraints \eqref{ping-4r} and \eqref{potential-r} respectively; crucially, they do not appear at all in the positive definite constraints in (i), thanks to the previously observed absence of the fields $g_{\partial_t,\partial_t}$ or $g_{i\partial_t,i\partial_t}$.  As such, we will be able to eliminate these fields from the problem in the next section.
\end{remark}

Let us now see how Theorem \ref{main-5} implies Theorem \ref{main-4}.  Let the fields 
$$g_{1,1}, g_{\partial_r,\partial_r}, g_{\partial_\omega,\partial_\omega}, g_{\partial_r,\partial_t}, g_{1,i\partial_r}, g_{1,i\partial_t}, v$$ 
be as in Theorem \ref{main-5}.  Let $A>0$ be a large quantity to be chosen later.  We then define the functions $G_{D_1,D_2}$ for $D_1,D_2 \in {\mathcal D}: H_d \to \R$ and $V: H_d \to \R$ by the formulae in Table \ref{tab}.

\begin{table}[ht]
\caption{Components of $G$ and $V$, and their values at a given point $(t,x)$ of $H_d$; thus for instance $G_{1,1}(t,x)$ and $G_{i,i}(t,x)$ are both set equal to $g_{1,1}(t,|x|)$.  Here $j,k=1,\dots,d$ are arbitrary.}\label{tab}
\begin{tabular}{|l|l|}
\hline 
Fields & Value at $(t,x)$\\
\hline 
$G_{1,1}, G_{i,i}$ & $g_{1,1}(t,|x|)$ \\
$G_{1,i}, G_{i,1}$ & $0$ \\
\hline
$G_{1,\partial_{x_j}}, G_{\partial_{x_j},1}, G_{i,i\partial_{x_j}}, G_{i\partial_{x_j},i}$ & $\frac{1}{2} \partial_{x_j} G_{1,1}(t,x)$ \\
$G_{1,i\partial_{x_j}}, G_{i\partial_{x_j},1}, -G_{i,\partial_{x_j}}, -G_{\partial_{x_j},i}$ & $\frac{x_j}{r} g_{1,i\partial_r}(t,|x|)$ \\
\hline 
$G_{1,\partial_{t}}, G_{\partial_{t},1}, G_{i,i\partial_{t}}, G_{i\partial_{t},i}$ & $\frac{1}{2} \partial_{t} G_{1,1}(t,x)$ \\
$G_{1,i\partial_t}, G_{i\partial_t,1}, -G_{i,\partial_t}, -G_{\partial_t, i}$ & $g_{1,i\partial_t}(t,|x|)$ \\
\hline
$G_{\partial_{x_j},\partial_{x_k}}, G_{i\partial_{x_j},i\partial_{x_k}}$ & $\frac{x_j x_k}{|x|^2} g_{\partial_r,\partial_r}(t,|x|) + \left(\delta_{jk} - \frac{x_j x_k}{|x|^2}\right) g_{\partial_\omega,\partial_\omega}(t,|x|)$ \\
$G_{\partial_{x_j},i\partial_{x_k}}, G_{i\partial_{x_k},\partial_{x_j}}$ & $0$ \\
\hline
$G_{\partial_{x_j},\partial_t}, G_{\partial_t,\partial_{x_j}}, G_{i\partial_{x_j},i\partial_t}, G_{i\partial_t, i \partial_{x_j}}$ & $\frac{x_j}{|x|} g_{\partial_r,\partial_t}(t,|x|)$\\
$G_{\partial_{x_j},i\partial_t}, G_{i\partial_t,\partial_{x_j}}, -G_{i\partial_{x_j}, \partial_t}, -G_{\partial_t, i\partial_{x_j}}$ & $\frac{1}{2} \left(\partial_{x_j} G_{1,i\partial_t}(t,x) - \partial_t G_{1,i\partial_{x_j}}(t,x)\right)$ \\
\hline
$G_{\partial_t, \partial_t}$, $G_{i\partial_t, i\partial_t}$ & $A \rho(t,x)^{-\frac{4}{p-1}-4}$ \\
$G_{\partial_t, i \partial_t}$, $G_{i\partial_t,\partial_t}$ & $0$\\
\hline
$V$ & $v(t,|x|)$\\
\hline
\end{tabular}
\end{table}

It is a classical result of Whitney \cite{whitney} that a smooth function $g(t,r)$ that is even in $r$ can be thought of as a smooth function of $(t,r^2)$ (where the latter is viewed on the half-line $[0,+\infty)$), while an odd function of $t,r$ that is odd in $r$ can be thought of as $r$ times a smooth function of $(t,r^2)$; see e.g. \cite[Corollary 2.2]{tao-high}.  In particular, we see that $g_{1,1}(t,|x|), g_{1,i\partial_t}(t,|x|), v(t,|x|), g_{\partial_r,\partial_r}(t,|x|) - g_{\partial_\omega,\partial_\omega}(t,|x|)$ are smooth functions of $t,x$, while $g_{\partial_r, \partial_t}(t,|x|), g_{1,i\partial_r}(t,|x|)$ are $|x|$ times a smooth function of $t,x$.  Finally, $g_{\partial_r,\partial_r}(t,|x|) - g_{\partial_\omega,\partial_\omega}(t,|x|)$ is $|x|^2$ times a smooth function of $t,x$, due to the hypothesis that $g_{\partial_r,\partial_r} - g_{\partial_\omega,\partial_\omega}$ vanishes on the time axis.  From this and Table \ref{tab}, we can check that all of the functions $G_{D_1,D_2}$, $V$ have removable singularities on the time axis and thus define smooth functions on $H_d$.  

From tedious direct calculation using Table \ref{tab}, we can verify the symmetry $G_{D_2,D_1} = G_{D_1,D_2}$ and the properties claimed in Theorem \ref{main-4}(ii), (iii).  Since
$$ \frac{G_{1, i \partial_{x_j}}}{G_{1,1}}(t,x) = \frac{x_j}{|x|} \frac{g_{1, i \partial_r}}{g_{1,1}}(t,|x|)$$
(away from the time axis at least) we have
$$ \partial_{x_k} \frac{G_{1, i \partial_{x_j}}}{G_{1,1}}(t,x) = \left(\frac{\delta_{jk}}{|x|} - \frac{x_j x_k}{|x|^3} + \frac{x_j x_k}{|x|^2} \partial_r\right) \frac{g_{1, i \partial_r}}{g_{1,1}}(t,|x|);$$
as the right-hand side is symmetric in $j$ and $k$, we have the curl-free property in Theorem \ref{main-4}(iv) (after removing the singularity at the time axis).  The positivity property in Theorem \ref{main-4}(v) is clear.  For in Theorem \ref{main-4}(vi), we observe from the constructions of the various fields that 
\begin{align*}
T_{00}(t,x) &= T_{00}(t,|x|) \\
T_{0j}(t,x) = T_{j0}(t,x) &= \frac{x_j}{|x|} T_{0r}(t,|x|) \\
T_{jk}(t,x) &= \frac{x_j x_k}{|x|^2} T_{rr}(t,|x|) + \left(\delta_{jk} - \frac{x_j x_k}{|x|^2}\right) T_{\omega\omega}(t,|x|) \\
E(t,x) &= E(t,|x|) \\
J_j(t,x) &= \frac{x_j}{|x|} J_r(t,|x|)
\end{align*}
and then it is a routine matter to derive \eqref{potential}, \eqref{ping-2}, \eqref{ping-3}, \eqref{ping-4} from \eqref{potential-r}, \eqref{ping-2r}, \eqref{ping-3r}, \eqref{ping-4r}, again working away from the time axis and then using smoothness to remove the singularity.

The only remaining task is to check in Theorem \ref{main-4}(i); that is to say, we need to verify that for $(t,x) \in H_d$, the matrix $(G_{D_1,D_2}(t,x))_{D_1,D_2 \in {\mathcal D}}$ is strictly positive definite.  In view of \eqref{scaling}, it suffices to do so in a fundamental domain for $H_d /T^\Z$, such as $\{ (t,x): 1 \leq \rho < 2 \}$. By continiuty, we can also avoid the time axis $x=0$ as long as our positive definiteness is uniform in $t,x$.  Henceforth we fix $(t,x)$ in this region and suppress dependence on $t,x$.  If we let $\vec a \coloneqq (a_D)_{D \in {\mathcal D}}$ be a tuple of real numbers, not all zero, our task is to show that
$$ \sum_{D_1,D_2 \in {\mathcal D}} a_{D_1} a_{D_2} G_{D_1,D_2} > \eps |\vec a|^2$$
for some $\eps>0$ uniform in $t,x$.
The left-hand side can be expanded out as
\begin{align*}
& (a_1^2 + a_i^2) G_{1,1} + 2 (a_1 a_{\partial_{x_j}} + a_i a_{i\partial_{x_j}}) G_{1,\partial_{x_j}} +
2(a_1 a_{i\partial_{x_j}} - a_i a_{\partial_{x_j}}) G_{1,i\partial_{x_j}} \\
&\quad  + 2(a_1 a_{\partial_t} + a_i a_{i\partial_t}) G_{1,\partial_t} +
2(a_1 a_{i\partial_t} - a_i a_{\partial_t}) G_{1,i\partial_t} \\
&\quad + 2 (a_{\partial_{x_j}} a_{\partial_{x_k}} + a_{i\partial_{x_j}} a_{i \partial_{x_k}}) G_{\partial_{x_j},\partial_{x_k}}  \\
&\quad + 2 (a_{\partial_{x_j}} a_{\partial_t} + 2a_{i\partial_{x_j}} a_{i \partial_t}) G_{\partial_{x_j},\partial_t} + (a_{\partial_{x_j}} a_{i\partial_t} - a_{i\partial_{x_j}} a_{\partial_t}) G_{\partial_{x_j},i\partial_t} \\
&\quad + (a_{\partial_t}^2 + a_{i\partial_t}^2) G_{\partial_t,\partial_t}
\end{align*}
where we use the usual summation conventions.
If we define 
$$\vec b = ( a_1, a_i, a_{\partial_{x_1}}, \dots, a_{\partial_{x_d}}, a_{i\partial_{x_1}}, \dots, a_{i\partial_{x_d}})$$
to be the spatial components of $\vec a$, then all the cross-terms in the above expression involving one copy of $a_{\partial_t}$ or $a_{i\partial_t}$ and one term from $\vec b$ can be controlled via Cauchy-Schwarz as
$$ O( |\vec b| (a_{\partial_t}^2 + a_{i\partial_t}^2)^{1/2} )$$
where the implied constants can depend on $G$ but are uniform in $t,x$ in the fundamental domain.  On the other hand, from construction of $G_{\partial_t,\partial_t}$, the term $(a_{\partial_t}^2 + a_{i\partial_t}^2) G_{\partial_t,\partial_t}$ is bounded from below by $c A (a_{\partial_t}^2 + a_{i\partial_t}^2)$ for some absolute constant $c>0$.  By the arithmetic mean-geometric mean inequality, it will thus suffice (for $A$ large enough) to obtain the bound
\begin{equation}\label{roar}
\begin{split}
& (a_1^2 + a_i^2) G_{1,1} + 2 (a_1 a_{\partial_{x_j}} + a_i a_{i\partial_{x_j}}) G_{1,\partial_{x_j}} +
2(a_1 a_{i\partial_{x_j}} - a_i a_{\partial_{x_j}}) G_{1,i\partial_{x_j}} \\
&\quad + (a_{\partial_{x_j}} a_{\partial_{x_k}} + a_{i\partial_{x_j}} a_{i \partial_{x_k}}) G_{\partial_{x_j},\partial_{x_k}}  \\
&\quad \geq 2 \eps |\vec b|^2
\end{split}
\end{equation}
for some $\eps>0$ independent of $A$.

If we set $a_{\partial_r}, a_{i\partial_r} \in \R$ and $a_{\partial_\omega}, a_{i\partial_\omega} \in \R^d$ to be the quantities
\begin{align*}
a_{\partial_r} &\coloneqq \frac{x_j}{|x|} a_{\partial_{x_j}} \\
a_{i\partial_r} &\coloneqq \frac{x_j}{|x|} a_{i\partial_{x_j}} \\
a_{\partial_\omega} &\coloneqq \left( a_{\partial_{x_j}} - \frac{x_j}{|x|} a_{\partial_r} \right)_{j=1}^d \\
a_{i\partial_\omega} &\coloneqq \left( a_{i\partial_{x_j}} - \frac{x_j}{|x|} a_{i\partial_r} \right)_{j=1}^d 
\end{align*}
then the left-hand side of \eqref{roar} can be written as
\begin{align*}
& (a_1^2 + a_i^2) g_{1,1} + 2 (a_1 a_{\partial_r} + a_i a_{i\partial_r}) g_{1,\partial_r} +
2 (a_1 a_{i\partial_r} - a_i a_{\partial_r}) g_{1,i\partial_r} \\
&\quad +  (a_{\partial_r}^2+ a_{i\partial_r}^2) g_{\partial_r,\partial_r} + 
 (|a_{\partial_\omega}|^2 + |a_{i\partial_\omega}|^2) g_{\partial_\omega, \partial_\omega}
\end{align*}
where we suppress the dependence on $t$ and $|x|$ in the $g$ terms.  The claim now follows from the Cauchy-Schwarz inequality, the Legendre identity
$$ (a_1 a_{\partial_r} + a_i a_{i\partial_r})^2 + (a_1 a_{i\partial_r} - a_i a_{\partial_r})^2 = (a_1^2 + a_i^2) (a_{\partial_r}^2+ a_{i\partial_r}^2)$$
and the hypotheses \eqref{gstorm}, \eqref{gstorm-2}.

It remains to prove Theorem \ref{main-5}.  This will be the objective of the remaining sections of the paper.

\section{Eliminating the energy conservation law and the potential energy identity}\label{energy-sec}

To motivate the next reduction, assume for the moment that the fields
$$g_{1,1}, g_{\partial_r,\partial_r}, g_{\partial_\omega,\partial_\omega}, g_{\partial_r,\partial_t}, g_{1,i\partial_r}, g_{1,i\partial_t}, v$$
obey the hypotheses and conclusions of Theorem \ref{main-5}, and let $T_{00}, T_{0r}, T_{rr}, T_{\omega\omega}, E, J_r$ be as in that theorem.  The pointwise conservation laws \eqref{ping-2r}, \eqref{ping-3r}, \eqref{ping-4r} can then be written in a familiar integral form.  For instance, multiplying the pointwise mass conservation law \eqref{ping-2r} by $r^{d-1}$ and then integrating on a fixed interval $[0,R]$, one obtains the integral mass conservation identity
$$ \partial_t \int_0^R T_{00}(t,r)\ r^{d-1} dr = - R^{d-1} T_{0r}(t,R) $$
and similarly the pointwise energy conservation law \eqref{ping-4r} gives the integral energy conservation identity
\begin{equation}\label{local-energy}
\partial_t \int_0^R E(t,r)\ r^{d-1} dr = - R^{d-1} J_r(t,R).
\end{equation}
Applying the same manipulations to \eqref{ping-3r} gives a more complicated identity
$$
\partial_t \int_0^R T_{0r}(t,r)\ r^{d-1} dr = - R^{d-1} T_{rr}(t,R) + (d-1) \int_0^R T_{\omega\omega}(t,r)\ r^{d-2} dr;$$
if one sets $d=3$ for sake of discussion, applies \eqref{trr-def}, \eqref{too-def}, and integrates by parts, one obtains the local Morawetz identity
\begin{align*}
\partial_t \int_0^R T_{0r}(t,r)\ r^2 dr &= - R^2 T_{rr}(t,R) - 2 R \partial_r g_{1,1}(t,R) - 2 g_{1,1}(t,R) \\
&\quad + 2 g_{1,1}(t,0) + \int_0^R (8 g_{\partial \omega,\partial \omega}(t,r) + 4(p-1) V(t,r)) r\ dr.
\end{align*}

These sorts of identities are often used in subcritical situations to help establish global regularity of solutions to NLS.  For instance, suppose we are in the energy-subcritical situation where $d < 3$, or $d \geq 3$ and $p < 1 + \frac{4}{d-2}$, rather than in the energy super-critical situation \eqref{energy-supercrit} that is the focus of this paper.   We apply \eqref{local-energy} with $R=1$ (say) to conclude that $\int_0^1 E(t,r)\ r^{d-1} dr$ stays bounded as $t \to 0^+$.  But from the scaling hypothesis (ii) and \eqref{E-def}, the energy density $E$ scales like $\rho^{-\frac{4}{p-1}-2}$, and hence (on setting $\lambda = t^{-1/2}$ and integrating $r$ from $0$ to $1$)
$$ \int_0^{t^{-1/2}} E(1,r)\ r^{d-1} dr = t^{\frac{2}{p-1}-\frac{d-2}{2}} \int_0^1 E(t,r)\ r^{d-1} dr.$$
In the energy-subcritical case, the exponent $\frac{2}{p-1}-\frac{d-2}{2}$ is positive, and hence $\int_0^{t^{-1/2}} E(1,r)\ r^{d-1} dr$ goes to zero as $t \to 0^+$.  In the defocusing setting $v>0$, the energy density $E$ is strictly positive, giving a contradiction.

Now we return to the energy-supercritical situation of Theorem \ref{main-5}.  In this case, the local energy conservation law \eqref{local-energy} does not lead to a contradiction, but still manages to impose a one-dimensional linear constraint on the energy density $E$.  (Note that the energy current $J_r$ is almost arbitrary, since there are almost no constraints on the field $g_{\partial_r,\partial_t}$ in Theorem \ref{main-5} other than through the energy conservation law.)   Namely, from \eqref{local-energy}, the smoothness of $J_r$ on $H_1$, Taylor expansion, and the fundamental theorem of calculus we have the asymptotic
$$ \int_0^1 E(t,r)\ r^{d-1} dr = P_k(t) + O( t^{k+1} )$$
as $t \to 0$, where $k \geq 0$ is an integer to be chosen later, $P_k$ is a polynomial of degree at most $k$, and the implied constant in the $O()$ notation is allowed to depend on $k$ and on the data in Theorem \ref{main-5}.  As $E$ scales like $\rho^{-\frac{4}{p-1}-2}$, we can then conclude the asymptotic
\begin{equation}\label{ss}
 \int_0^R E(1,r)\ r^{d-1} dr = R^{d-2-\frac{4}{p-1}} (P_k(1/R^2) + O( R^{-2k-2} ))
\end{equation}
as $R \to \infty$.  Again using the fact that $E$ scales like $\rho^{-\frac{4}{p-1}-2}$, we also have the asymptotic
\begin{equation}\label{ss2}
 E(1,r) r^{d-1} = r^{d-3-\frac{4}{p-1}} (Q_k(1/r^2) + O( r^{-2k-2} ))
\end{equation}
as $r \to \infty$, for some polynomial $Q_k$ of degree at most $k$.  

Now take $k$ to be the largest integer such that
\begin{equation}\label{dak}
 d - 2 - \frac{4}{p-1} - 2k \geq 0;
\end{equation}
note from the energy-supercriticality hypothesis \eqref{energy-supercrit} that $k$ is non-negative.
If strict inequality holds in \eqref{dak}, then the error term $R^{d-2-\frac{4}{p-1}} O(R^{-2k-2})$ in \eqref{ss} goes to zero at infinity, while the error term $r^{d-3-\frac{4}{p-1}} O(r^{-2k-2})$ in \eqref{ss2} is absolutely integrable in $r$ (for $r$ near zero this follows from the local integrability of $r^{d-3-\frac{4}{p-1}} Q_k(1/r^2)$ and the triangle inequality).  Integrating \eqref{ss2} and comparing with \eqref{ss}, we see that $R^{d-2-\frac{4}{p-1}} P_k(1/R^2)$ must be a primitive of $r^{d-3-\frac{4}{p-1}} Q_k(1/r^2)$, and one has vanishing renormalised total energy in the sense that
\begin{equation}\label{eqk}
 \lim_{R \to \infty} \int_0^R \left(E(1,r) r^{d-1} - r^{d-3-\frac{4}{p-1}} Q_k(1/r^2)\right)\ dr = 0
\end{equation}
since otherwise there would have to be a constant term in $R^{d-2-\frac{4}{p-1}} P_k(1/R^2)$, which is not possible when strict inequality occurs in \eqref{dak}.  If instead equality holds in \eqref{dak}, then the same analysis yields instead that the degree $k$ coefficient of $Q_k$ must vanish (that is to say, $Q_k$ in fact has degree at most $k-1$), since otherwise there would have to be a $\log R$ term present in \eqref{ss}, which is not the case.

As it turns out, though, in the energy-supercritical case the linear constraint that we have just obtained is ``dense'' rather than ``closed'', in the sense that data that does not obey this constraint can be perturbed (in a natural topology) to obey the constraint.  (In other words, the linear functional that defines the constraint is unbounded with respect to a certain natural norm.)  Informally speaking, this will be because for self-similar solutions to an energy-supercritical problem there will be an infinite amount of energy near spatial infinity that is available to ``spend'' to perform such a perturbation.  As such, the constraint can be eliminated entirely; we can also easily eliminate the potential identity \eqref{potential-r} due to the fact that the field $g_{1,i\partial_t}$ appearing in that identity is almost completely unconstrained outside of that identity.  More precisely, we can deduce Theorem \ref{main-5} from

\begin{theorem}[Fifth reduction]\label{main-6}   There exist smooth fields $g_{1,1}, g_{\partial_r,\partial_r}, g_{\partial_\omega,\partial_\omega}, g_{1,i\partial_r}, v: H_1 \to \R$ obeying the following properties:
\begin{itemize}
\item[(i)] One has the positive definite inequalities \eqref{gstorm}, \eqref{gstorm-2} pointwise on $H_1$.
\item[(ii)] For each $(D_1,D_2) = (1,1), (\partial_r,\partial_r), (\partial_\omega,\partial_\omega), (1, i\partial_r)$, $g_{D_1,D_2}$ scales like $\rho^{-\frac{4}{p-1}-\operatorname{ord}(D_1)-\operatorname{ord}(D_2)}$.  Similarly, we require that $v$ scales like $\rho^{-\frac{4}{p-1}-2}$.
\item[(v)]  One has the defocusing property $v>0$ pointwise on $H_1$.
\item[(vi)]  If one defines the mass density
$$ T_{00} \coloneqq g_{1,1}$$
the radial momentum density
$$ T_{0r} \coloneqq -2 g_{1,i\partial_r}$$
the radial stress
$$ T_{rr} \coloneqq 4 g_{\partial_r,\partial_r} + 2(p-1) v - \left(\partial_r^2 + \frac{d-1}{r} \partial_r\right) g_{1,1}$$
and the angular stress
$$ T_{\omega\omega} \coloneqq 4 g_{\partial_\omega,\partial_\omega} + 2(p-1) v - \left(\partial_r^2 + \frac{d-1}{r} \partial_r\right) g_{1,1}$$
then one has the conservation laws \eqref{ping-2r}, \eqref{ping-3r} with removable singularity at $r=0$.
\item[(vii)]  The functions $g_{1,1}, g_{\partial_r, \partial_r}, g_{\partial_\omega, \partial_\omega}, v$ are even in $r$, while $g_{1,i\partial_r}$ is odd in $r$.  Furthermore, $g_{\partial_r,\partial_r} - g_{\partial_\omega,\partial_\omega}$ vanishes on the time axis $r=0$.
\end{itemize}
\end{theorem}

Let us now see how Theorem \ref{main-6} implies Theorem \ref{main-5}.  By Theorem \ref{main-6}, we may find fields $g_{1,1},g_{\partial_r,\partial_r}, g_{\partial_\omega,\partial_\omega}, g_{1,i\partial_r}, v$ obeying the conclusions of that theorem.  Define the energy density $E: H_d \to \R$ by the formula \eqref{E-def}.  Clearly $E$ is smooth and scales like $\rho^{-\frac{4}{p-1}-2}$.  As in the previous discussion, we let $k$ be the largest integer obeying \eqref{dak}, so that $k \geq 0$; then $E(1,r)$ has an asymptotic expansion of the form \eqref{ss2} as $r \to \infty$ for some polynomial $Q_k$ of degree at most $k$.  Let us call the energy density $E$ \emph{good} if the following condition is satisfied:
\begin{itemize}
\item If strict inequality holds in \eqref{dak}, we call $E$ good if we have the asymptotic vanishing property \eqref{eqk}.  (Note that the limit in \eqref{eqk} exists because the integrand will be absolutely integrable, thanks to \eqref{ss2}.)
\item If instead equality holds in \eqref{dak}, we call $E$ good if the degree $k$ component of $Q_k$ vanishes, or equivalently that $Q_k$ has degree at most $k-1$.
\end{itemize}

Let us suppose first that $E$ is good, and conclude the proof of Theorem \ref{main-5}.  Using the data provided by Theorem \ref{main-6}, and comparing the conclusions of that theorem with that of Theorem \ref{main-5}, we see that it will suffice to produce smooth fields $g_{\partial_r,\partial_t}, g_{1,i\partial_t}: H_1 \to \R$ scaling like $\rho^{-\frac{4}{p-1}-3}$ and $\rho^{-\frac{4}{p-1}-2}$ respectively obeying the potential identity \eqref{potential-r} and the energy conservation law \eqref{ping-4r}, where $J_r$ is defined by \eqref{jr-def}; also, we require $g_{\partial_r,\partial_t}$ to be odd in $r$, and $g_{1,i\partial_t}$ to be even in $r$.

It is clear from \eqref{potential-r} how one should construct $g_{1,i\partial_t}$, namely one should set
$$ g_{1,i\partial_t} \coloneqq (p+1) v - 
\frac{1}{2} \left(\partial_r^2 + \frac{d-1}{r} \partial_r\right) g_{1,1} + g_{\partial_r, \partial_r} + (d-1) g_{\partial_\omega,\partial_\omega}.$$
Clearly $g_{1,i\partial_t}$ is smooth on $H_1$ and even in $r$, thanks to (vii).  It is clear from the scaling laws for $v,g_{1,1},g_{\partial_r,\partial_r},g_{\partial_\omega, \partial_\omega}$ that the field $g_{1,i\partial_t}$ scales like $\rho^{-\frac{4}{p-1}-2}$ as required, and the identity \eqref{potential-r} is clear from construction.

In a similar fashion, after using \eqref{jr-def} to rewrite \eqref{ping-4r} as
$$ \partial_t (r^{d-1} E) = \partial_r( r^{d-1} g_{\partial_r,\partial_t} )$$
it is clear from the fundamental theorem of calculus that we should define $g_{\partial_r,\partial_t}$ by the formula
\begin{equation}\label{toby}
g_{\partial_r,\partial_t}(t,R) \coloneqq \frac{1}{R^{d-1}} \int_0^R \partial_t E(t,r)\ r^{d-1} dr 
\end{equation}
in the interior $(0,+\infty) \times \R$ of $H_1$, where we adopt the convention $\int_0^R = -\int_R^0$ when $R$ is negative.  Note that the expression $(t,R) \mapsto \int_0^R \partial_t E(t,r)\ r^{d-1} dr$ is smooth and vanishes to order at least $d$ on the time axis $R=0$ when $t>0$, so the above definition of $g_{\partial_r,\partial_t}(t,R)$ extends smoothly to the entire interior of $H_1$ (including the time axis).  Since $E$ is even in $r$ and scales like $\rho^{-\frac{4}{p-1}-2}$,  $g_{\partial_r,\partial_t}$ is odd in $r$ and scales like $\rho^{-\frac{4}{p-1}-3}$ in the interior of $H_1$.  After defining $J_r$ by \eqref{jr-def}, we see from the fundamental theorem of calculus that \eqref{ping-4r} is obeyed in the interior of $H_1$.  To complete the list of requirements stated in Theorem \ref{main-5}, it will suffice to show that $g_{\partial_r,\partial_t}$ extends smoothly to the boundary component $\{ (0,r): r \neq 0 \}$ of $H_1$.  As $g_{\partial_r,\partial_t}$ is odd in $r$ and scales like $\rho^{-\frac{4}{p-1}-3}$, it suffices to show that $t \mapsto g_{\partial_r,\partial_t}(t,1)$ can be smoothly extended to $t=0$.  From \eqref{toby} we have
$$ g_{\partial_r,\partial_t}(t,1) = \partial_t \int_0^1 E(t,r)\ r^{d-1} dr $$
so it will suffice to show that the function $f \colon t \mapsto \int_0^1 E(t,r)\ r^{d-1} dr$ for $t>0$ can be smoothly extended to $t=0$.

From \eqref{sab} one has
\begin{equation}\label{etrt}
 E(t,r) = t^{-\frac{2}{p-1}-1} E\left(1, \frac{r}{\sqrt{t}} \right)
\end{equation}
so from a change of variables we have
$$ f(t) = t^{\frac{d-2}{2}-\frac{2}{p-1}} \int_0^{t^{-1/2}} E(1,r)\ r^{d-1} dr.$$
Recalling the polynomial $Q_k$ introduced previously, we thus have $f(t) = U_1(t) + U_2(t)$, where
$$ U_1(t) \coloneqq t^{\frac{d-2}{2}-\frac{2}{p-1}} \int_0^{t^{-1/2}} r^{d-3-\frac{4}{p-1}} Q_k(1/r^2)\ dr$$
and
$$ U_2(t) \coloneqq t^{\frac{d-2}{2}-\frac{2}{p-1}} \int_0^{t^{-1/2}} (r^{d-1} E(1,r) - r^{d-3-\frac{4}{p-1}} Q_k(1/r^2))\ dr.$$
As $Q_k$ has degree at most $k$, and at most $k-1$ when equality occurs in \eqref{dak}, the expression $r^{d-3-\frac{4}{p-1}} Q_k(1/r^2)$ is a linear combination of monomials $r^{d-3-\frac{4}{p-1}-2j}$ where $0 \leq j \leq k$, or $0 \leq j \leq k-1$ when equality occurs in \eqref{dak}.  In particular, from \eqref{dak} we see that the exponent in these monomials is strictly greater than $-1$, so the integral is absolutely convergent.  Performing the integral, we see that $U_1(t)$ is a polynomial in $t$ and thus clearly smoothly extendible to $t=0$.  It thus remains to show that $U_2$ is also smoothly extendible to $t=0$.

First suppose that strict inequality occurs in \eqref{dak}.  From \eqref{ss2} we know that the integrand $r^{d-1} E(1,r) - r^{d-3-\frac{4}{p-1}} Q_k(1/r^2)$ is absolutely integrable near $r=\infty$; from the smoothness of $E$ and the absolute integrability of the $U_1$ integrand we also have absolute integrability near $r=0$.  From \eqref{eqk} we thus have
\begin{equation}\label{fair}
 U_2(t) = - t^{\frac{d-2}{2}-\frac{2}{p-1}} \int_{t^{-1/2}}^\infty \left(r^{d-1} E(1,r) - r^{d-3-\frac{4}{p-1}} Q_k(1/r^2)\right)\ dr.
\end{equation}
Making the change of variables $r = \frac{1}{\sqrt{st}}$ and noting from \eqref{etrt} that
$$ E\left(1,\frac{1}{\sqrt{st}}\right) = E(st,1) (st)^{\frac{2}{p-1}+1},$$
this becomes
$$ U_2(t) = -\frac{1}{2} \int_0^1  \frac{E(st,1) - Q_k(st)}{s^{k+1}}\ s^{\frac{2}{p-1} - \frac{d}{2} + k} ds.$$
The function $(s,t) \mapsto E(st,1) - Q_k(st)$ is smooth and vanishes to order $k+1$ at $s=0$ thanks to \eqref{ss2} and rescaling, so the factor $\frac{E(st,1) - Q_k(st)}{s^{k+1}}$ is smooth in $t \in [0,1]$, uniformly in $s \in [0,1]$.  By definition of $k$, the weight $s^{\frac{2}{p-1} - \frac{d}{2}+k}$ is absolutely integrable on $[0,1]$.  From repeated differentiation under the integral sign we conclude that $U_2$ extends smoothly to $[0,1]$ as desired.

Now suppose that equality occurs in \eqref{dak}.  Now we do not necessarily have the vanishing property \eqref{eqk}, so we need to adjust \eqref{fair} to
$$
 U_2(t) = A t^{\frac{d-2}{2}-\frac{2}{p-1}}  - t^{\frac{d-2}{2}-\frac{2}{p-1}} \int_{t^{-1/2}}^\infty (r^{d-1} E(1,r) - r^{d-3-\frac{4}{p-1}} Q_k(1/r^2))\ dr$$
for some quantity $A$ depending on $E, d, p$ but not on $t$.  But in this case $\frac{d-2}{2}-\frac{2}{p-1}$ is an integer, so the monomial $A t^{\frac{d-2}{2}-\frac{2}{p-1}}$ clearly extends smoothly to $t=0$.  Repeating the previous arguments we then obtain the smooth extension of $U_2$ to $t=0$ as required.

We have completed the derivation of Theorem \ref{main-5} from Theorem \ref{main-6} under the hypothesis that $E$ is good.  It remains to handle the situation in which the energy density $E$ produced by Theorem \ref{main-6} is not good.  In this case, we will perturb the data $g_{1,1},g_{\partial_r,\partial_r}, g_{\partial_\omega,\partial_\omega}, g_{1,i\partial_r}, v$ provided by Theorem \ref{main-6} to make the energy density $E$ good, without losing any of the properties listed in Theorem \ref{main-6}.

More precisely, we will consider perturbations of the form
\begin{align*}
\tilde g_{1,1} &\coloneqq g_{1,1} \\
\tilde g_{\partial_r,\partial_r} &\coloneqq g_{\partial_r,\partial_r} - (p-1) Z \\
\tilde g_{\partial_\omega,\partial_\omega} &\coloneqq g_{\partial_\omega,\partial_\omega} - (p-1) Z \\
\tilde g_{1,i\partial r} &\coloneqq g_{1,i\partial_r} \\
\tilde v &\coloneqq v + 2Z
\end{align*}
where $Z: H_1 \to \R$ is a smooth function, even in $r$ and vanishing on the time axis $r=0$, and scaling like $\rho^{-\frac{4}{p-1}-2}$ to be chosen later.  It is clear that this perturbed data $\tilde g_{1,1}, \tilde g_{\partial_r,\partial_r}, \tilde g_{\partial_\omega,\partial_\omega}, \tilde g_{1,i\partial_r}, \tilde v$ continues to obey the scaling properties (ii) and symmetry properties (vii) of Theorem \ref{main-6}; the conservation laws (vi) are also maintained since the densities $T_{00}, T_{0r}, T_{rr}, T_{\omega\omega}$ are completely unchanged by this perturbation.  The positive definite inequalities (i) and the defocusing property (v) might not be preserved in general, but will be maintained if the perturbation $Z$ is sufficiently small in a suitable (scale-invariant) sense which we will make precise later.  Finally, from \eqref{E-def} we see that the perturbed energy density $\tilde E$ is related to the original energy density $E$ by the formula
\begin{equation}\label{firm}
\tilde E = E + \left(2 - \frac{d(p-1)}{2}\right) Z.
\end{equation}
In the energy-critical situation $p > 1 + \frac{4}{d-2}$, we avoid the mass-critical exponent $p = 1 + \frac{4}{d}$ and so the expression $2 - \frac{d(p-1)}{2}$ appearing in \eqref{firm} is non-zero (in fact it is positive).  This gives us substantial flexibility to modify the energy density $E$, and in particular to perturb it to be good.

We turn to the details.
First suppose that strict inequality occurs in \eqref{dak}.  We let $B$ denote the quantity
\begin{equation}\label{A-def}
 B \coloneqq \int_0^\infty (r^{d-1} E(1,r) - r^{d-3-\frac{4}{p-1}} Q_k(1/r^2))\ dr,
\end{equation}
which is well-defined since the integrand is absolutely integrable.  
We introduce a smooth nonnegative even function $\psi: \R \to \R$, supported in $[-2,-1] \cup [1,2]$, and normalised so that
$$ \int_0^\infty \psi(r)\ r^{d-1} dr = 1.$$
We let $R>1$ be a large quantity to be chosen later, and use the perturbation
$$ Z(t,r) \coloneqq -\frac{B}{2-\frac{d(p-1)}{2}} R^{-d} t^{-\frac{2}{p-1}-1} \psi\left( \frac{r}{R t^{1/2}} \right)$$
for $t > 0$, with $Z(t,r)$ vanishing at $t=0$.  By construction, $Z: H_1 \to \R$ is smooth, even, and scales like $\rho^{-\frac{4}{p-1}-2}$, with the function $r \mapsto Z(1,r)$ supported on $[-2R,-R] \cup [R,2R]$ and obeying the normalisation
$$ \int_0^\infty Z(t,r)\ r^{d-1} = \frac{B}{2-\frac{d(p-1)}{2}}.$$
Comparing this with \eqref{firm} and \eqref{A-def} we see that
$$ \int_0^\infty \left(r^{d-1} \tilde E(1,r) - r^{d-3-\frac{4}{p-1}} Q_k(1/r^2)\right)\ dr = 0$$
so that $\tilde E$ is good (note that $\tilde E$ obeys the same asymptotics \eqref{ss2} as $E$ with the same polynomial $Q_k$).  It remains to choose the parameter $R$ so that the perturbed fields $\tilde g_{1,1}, \tilde g_{\partial_r,\partial_r}, \tilde g_{\partial_\omega,\partial_\omega}, \tilde g_{1,i\partial_r}, \tilde v$ continue to obey the properties (i), (v).

We begin with (v) for $\tilde v$.  By hypothesis, the original field $v$ is continuous, scales like $\rho^{-\frac{4}{p-1}-2}$, and everywhere positive, which (by the compactness of $H_1/T^\Z$) implies a pointwise bound
$$ v(t,r) > \eps \rho^{-\frac{4}{p-1} - 2}$$
on $H_1$ for some $\eps>0$.  In order for $\tilde V$ to also obey (v), it thus suffices to obtain the pointwise bound
$$ \frac{B}{2-\frac{d(p-1)}{2}} R^{-d} t^{-\frac{2}{p-1}-1} \psi\left( \frac{r}{R t^{1/2}} \right) \leq \eps \rho^{-\frac{4}{p-1} - 2}.$$
On the support of $\psi( \frac{r}{R t^{1/2}} )$, $t$ is comparable to $(R^{-1} \rho)^2$, so the left-hand side is $O( A R^{\frac{4}{p-1} - d + 2} \rho^{-\frac{4}{p-1} - 2})$.  As we are in the energy-supercritical situation, the exponent $\frac{4}{p-1} - d + 2$ is negative, and so we obtain the required bound if $R$ is large enough.

Similarly, from (i), scaling and compactness we obtain the pointwise bounds
$$ g_{\partial_r,\partial_r}, g_{\partial_\omega,\partial_\omega} > \eps' \rho^{-\frac{4}{p-1} - 2}$$
and
$$ g_{\partial_r,\partial_r} - \frac{ \left(\frac{1}{2} \partial_r g_{1,1}\right)^2 + g_{1,i\partial_r}^2}{g_{1,1}} > \eps' \rho^{-\frac{4}{p-1} - 2}$$
on $H_1$ for some $\eps'>0$, and by arguing as before we see that these properties will be preserved by the perturbation if $R$ is large enough.  The claim follows.

Now suppose instead that \eqref{dak} holds with equality; from energy-supercriticality this implies that $k$ is positive.  We write
\begin{equation}\label{qkk}
 Q_k(s) = Q_{k-1}(s) + C s^k 
\end{equation}
for all $s \in \R$ and some real number $C$, where $Q_{k-1}$ is a polynomial of degree at most $k-1$.  We let $\eta: \R \to \R$ be a smooth nonnegative even function, vanishing near the origin and equal to $1$ near $\pm \infty$, let $R \geq 1$ be a large parameter to be chosen later, and set
\begin{equation}\label{fatr}
 Z(t,r) \coloneqq -\frac{C}{2-\frac{d(p-1)}{2}} |r|^{-d} t^k \eta\left( \frac{r}{R t^{1/2}} \right)
\end{equation}
on $H_1$, with the convention that $\eta(\frac{r}{R t^{1/2}}) = 1$ when $t=0$.
It is clear that $Z: H_1 \to \R$ is smooth, even in $r$, and vanishing near the time axis, and as \eqref{dak} holds with equality we have $Z$ scaling like $\rho^{-\frac{4}{p-1}-2}$ as required.  From \eqref{firm}, \eqref{ss2}, \eqref{qkk}, and \eqref{fatr} we have
$$ \tilde E(1,r) = r^{-2-\frac{4}{p-1}} (Q_{k-1}(1/r^2) + O( r^{-2k-2} ))$$
as $r \to \infty$, where the implied constant in the $O()$ notation can depend on $R$.  In particular, $\tilde E$ is good.  It remains to show that the properties in (i), (v) are maintained by the perturbation.  By repeating the previous arguments, it suffices to ensure that one has the pointwise bound
$$
\frac{C}{2-\frac{d(p-1)}{2}} |r|^{-d} t^k \eta( \frac{r}{R t^{1/2}} ) \leq 
\eps \rho^{-\frac{4}{p-1} - 2} $$
where $\eps>0$ is a quantity not depending on $R$.  But on the support of $\eta(\frac{r}{Rt^{1/2}})$, $|r|$ is comparable to $\rho$ and $t$ is $O( r^2 / R^2 )$, so the right-hand side is $O( C R^{-2k} \rho^{-\frac{4}{p-1} - 2} )$ (since \eqref{dak} holds with equality), and the claim follows by taking $R$ large enough.  
This completes the derivation of Theorem \ref{main-5} from Theorem \ref{main-6}.

It remains to prove Theorem \ref{main-6}.  This will be the objective of the remaining sections of the paper.

\section{Eliminating the potential}

We now make an easy reduction by eliminating the role of the potential energy density $v$.

Let $d \geq 1$ and $p>1$, and suppose we have fields
$g_{1,1}, g_{\partial_r,\partial_r}, g_{\partial_\omega,\partial_\omega}, g_{1,i\partial_r}, v$ obeying the properties claimed in Theorem \ref{main-6}.  If we then define the modified fields $\tilde g_{1,1}, \tilde g_{\partial_r,\partial_r}, \tilde g_{\partial_\omega,\partial_\omega}, \tilde g_{1,i\partial_r}, \tilde v$ by
\begin{align*}
\tilde g_{1,1} &\coloneqq g_{1,1} \\
\tilde g_{\partial_r,\partial_r} &\coloneqq g_{\partial_r,\partial_r} + \frac{p-1}{2} v \\
\tilde g_{\partial_\omega,\partial_\omega} &\coloneqq g_{\partial_\omega,\partial_\omega} + \frac{p-1}{2} v \\
\tilde g_{1,i\partial_r} &\coloneqq g_{1,i\partial_r} \\
\tilde v &\coloneqq 0
\end{align*}
then one easily verifies that these new fields also obey the claims of Theorem \ref{main-6}, except with the defocusing property $v>0$ replaced by $v=0$ (note that the new fields have exactly the same stresses $T_{rr}, T_{\omega\omega}$ as the original fields).  In the converse direction, it turns out that we can replace the defocusing property $v>0$ in Theorem \ref{main-6}(v) by $v=0$.  More precisely, we can deduce Theorem \ref{main-6} from

\begin{theorem}[Sixth reduction]\label{main-7}   Then there exist smooth fields $g_{1,1}, g_{\partial_r,\partial_r}, g_{\partial_\omega,\partial_\omega}, g_{1,i\partial_r}: H_1 \to \R$ obeying the following properties:
\begin{itemize}
\item[(i)] One has the positive definite inequalities 
\begin{align}
g_{1,1}, g_{\partial_\omega,\partial_\omega} &>0 \label{sa}\\
g_{\partial_r,\partial_r} &> \frac{ \left(\frac{1}{2} \partial_r g_{1,1}\right)^2 + g_{1,i\partial_r}^2}{g_{1,1}}\label{sa-2}
\end{align}
pointwise on $H_1$.
\item[(ii)] The fields $g_{1,1}, g_{\partial_r,\partial_r}, g_{\partial_\omega,\partial_\omega}, g_{1,i\partial_r}$ scale like $\rho^{-\frac{4}{p-1}}$, $\rho^{-\frac{4}{p-1}-2}$, $\rho^{-\frac{4}{p-1}-2}$, and $\rho^{-\frac{4}{p-1}-1}$ respectively.
\item[(vi)]  One has the mass conservation law
\begin{equation}\label{gtr}
\partial_{t} g_{1,1} = 2\left(\partial_r + \frac{d-1}{r}\right) g_{1,i\partial_r}
\end{equation}
and momentum conservation law
\begin{equation}\label{gtr-2}
 4 \left(\partial_r + \frac{d-1}{r}\right) g_{\partial_r,\partial_r} 
= 4 \frac{d-1}{r} g_{\partial_\omega,\partial_\omega} + \partial_r \left(\partial_r^2 + \frac{d-1}{r} \partial_r\right) g_{1,1} + 2 \partial_t g_{1,i\partial_r} 
\end{equation}
with removable singularity at $r=0$.
\item[(vii)]  The functions $g_{1,1}, g_{\partial_r, \partial_r}, g_{\partial_\omega, \partial_\omega}$ are even in $r$, while $g_{1,i\partial_r}$ is odd in $r$.  Furthermore, $g_{\partial_r,\partial_r} - g_{\partial_\omega,\partial_\omega}$ vanishes on the time axis $r=0$.
\end{itemize}
\end{theorem}

Let us now see how Theorem \ref{main-7} implies Theorem \ref{main-6}.  Suppose that $g_{1,1}, g_{\partial_r,\partial_r}, g_{\partial_\omega,\partial_\omega}, g_{1,i\partial_r}, v$ obeys the properties claimed by Theorem \ref{main-7}.  Let $\eps>0$ be a small quantity to be chosen later, and introduce the modified fields
\begin{align*}
\tilde g_{1,1} &\coloneqq g_{1,1} \\
\tilde g_{\partial_r,\partial_r} &\coloneqq g_{\partial_r,\partial_r} - \frac{p-1}{2} \eps \rho^{-\frac{4}{p-1}-2} \\
\tilde g_{\partial_\omega,\partial_\omega} &\coloneqq g_{\partial_\omega,\partial_\omega} - \frac{p-1}{2} \eps \rho^{-\frac{4}{p-1}-2} \\
\tilde g_{1,i\partial_r} &\coloneqq g_{1,i\partial_r} \\
\tilde v &\coloneqq \eps \rho^{-\frac{4}{p-1}-2}.
\end{align*}
The properties (ii), (v), (vii) of Theorem \ref{main-6} are easily verified to be obeyed by these new fields.  Using \eqref{gtr}, \eqref{gtr-2} and the definitions of $T_{00}, T_{0r}, T_{rr}, T_{r\omega}$ in Theorem \ref{main-6}(vi), we see that the conservation laws \eqref{ping-2r}, \eqref{ping-3r} are obeyed by the original fields $g_{1,1}, g_{\partial_r,\partial_r}, g_{\partial_\omega,\partial_\omega}, g_{1,i\partial_r}$ (with $v=0$), and hence by the new fields $\tilde g_{1,1}, \tilde g_{\partial_r,\partial_r}, \tilde g_{\partial_\omega,\partial_\omega}, \tilde g_{1,i\partial r}, \tilde v$ since the stress-energy densities $T_{00}, T_{0r}, T_{rr}, T_{\omega\omega}$ for these new fields are identical to those for the original fields.  By using compactness as in the previous section, we also see that the positive definite inequalities (i) will also be obeyed if $\eps$ is small enough, and the claim follows.

It remains to prove Theorem \ref{main-7}.  This will be the objective of the remaining sections of the paper.

\begin{remark}  The reduction to the case $v=0$ does \emph{not} mean that the finite time blowup in Theorem \ref{main} is arising from a vanishing potential $F=0$, and indeed such a vanishing is not possible since the linear Schr\"odinger equation will not create singularities in finite time from smooth, compactly supported data.  Instead, the $v=0$ case roughly speaking corresponds to the case where $F(x)$ is very close to zero when $x$ lies in the range of the solution map $u: H_d \to \C^m$, but is allowed to be much larger than zero elsewhere; in particular, $\nabla F(x)$ does not need to vanish or be small on the range of $u$.
\end{remark}

\section{Eliminating the radial stress}

Having eliminated the potential energy density $v$ from the problem, we now turn our attention to eliminating the radial stress $g_{\partial_r,\partial_r}$.  To motivate this reduction, assume for the moment that the hypotheses and conclusions of Theorem \ref{main-7} hold.
Multiplying the momentum conservation law \eqref{gtr-2} by $\frac{1}{4} r^{d-1}$, we arrive at the identity
\begin{equation}\label{god}
 \partial_r ( r^{d-1} g_{\partial_r,\partial_r} ) = S_2
\end{equation}
where $S_2: H_1 \to \R$ is the function
\begin{equation}\label{tfdef}
S_2 \coloneqq (d-1) r^{d-2} g_{\partial_\omega,\partial_\omega} + S_1
\end{equation}
and $S_1: H_1 \to \R$ is the function 
\begin{equation}\label{FDEF}
 S_1 \coloneqq \frac{1}{4} r^{d-1} \left( \partial_r \left(\partial_r^2 + \frac{d-1}{r} \partial_r\right) g_{1,1} + 2 \partial_t g_{1,i\partial_r} \right ).
\end{equation}  
As $r^{d-1} g_{\partial_r,\partial_r}$ vanishes on the time axis $r=0$, we can therefore solve for $g_{\partial_r,\partial_r}(t,R)$ for $(t,R)$ in the interior of $H_1$ by the formula
$$
g_{\partial_r,\partial_r}(t,R) \coloneqq \frac{1}{R^{d-1}} \int_0^R S_2(t,r)\ dr,$$
noting that the right-hand side has a removable singularity at $R=0$ since the integral vanishes to order at least $d-1$ there; a Taylor expansion at $R=0$ then also reveals that $g_{\partial_r,\partial_r} - g_{\partial_\omega,\partial_\omega}$ vanishes at $R=0$.
However, it is not immediately clear that the right-hand side will extend smoothly to the boundary $\{ (0,R): R \neq 0 \}$ of $H_1$, due to the singularity of the integrand at the spacetime origin.  As in Section \ref{energy-sec}, this requires an additional ``good'' hypothesis on the asymptotic expansion of the right-hand side of \eqref{god}.  More precisely, we can deduce Theorem \ref{main-7} from

\begin{theorem}[Seventh reduction]\label{main-8}   Then there exist smooth fields $g_{1,1}, g_{\partial_\omega,\partial_\omega}, g_{1,i\partial_r}: H_1 \to \R$ obeying the following properties:
\begin{itemize}
\item[(i)] One has the positive definite inequalities \eqref{sa} pointwise on $H_1$.
\item[(ii)] The fields $g_{1,1}, g_{\partial_\omega,\partial_\omega}, g_{1,i\partial_r}$ scale like $\rho^{-\frac{4}{p-1}}$, $\rho^{-\frac{4}{p-1}-2}$, and $\rho^{-\frac{4}{p-1}-1}$ respectively.
\item[(vi)]  One has the conservation law \eqref{gtr}
with removable singularity at $r=0$.
\item[(vii)]  The functions $g_{1,1}, g_{\partial_\omega, \partial_\omega}$ are even in $r$, while $g_{1,i\partial_r}$ is odd in $r$.  
\item[(viii)]  There is an $\eps>0$ such that one has the pointwise inequality
$$\frac{1}{R^{d-1}} \int_0^R S_2(1,r)\ dr \geq \frac{ \left(\frac{1}{2} \partial_r g_{1,1}\right)^2 + g_{1,i\partial_r}^2}{g_{1,1}}(1,R) + \eps \rho(1,R)^{-\frac{4}{p-1}-2} 
$$
for all $R > 0$, where $S_2: H_1 \to \R$ is the function defined by \eqref{tfdef}.
\item[(ix)]  Let $k \geq -1$ be the largest integer such that
\begin{equation}\label{dang}
 d-3-\frac{4}{p-1}-2k \geq 0.
\end{equation}
As $S_2$ is smooth and scales like $\rho^{-\frac{4}{p-1}+d-4}$, there is an asymptotic of the form
\begin{equation}\label{sso}
 S_2(1,r)= r^{d-4-\frac{4}{p-1}} (R_k(1/r^2) + O( r^{-2k-2} ))
\end{equation}
as $r \to \infty$  for some polynomial $R_k$ of degree at most $k$ (this forces $R$ to vanish in case $k=-1$, where we adopt the convention that $0$ has degree $-\infty$).  If strict inequality holds in \eqref{dang}, we require that
\begin{equation}\label{satyr}
 \int_0^\infty (S_2(1,r) - r^{d-4-\frac{4}{p-1}} R_k(1/r^2))\ dr = 0 
\end{equation}
(note that the integrand is absolutely integrable by \eqref{dang}, \eqref{sso}, and the smoothness of $S_2$).  If instead equality holds in \eqref{dang} (which can only occur if $k \geq 0$), we require that the degree $k$ coefficient of $R_k$ vanishes, so that $R_k$ actually has degree at most $k-1$.
\end{itemize}
\end{theorem}

Let us now see how Theorem \ref{main-8} implies Theorem \ref{main-7}.  Let $d \geq 3$ and $p > 1 + \frac{4}{d-2}$, and let $g_{1,1}, g_{\partial_\omega,\partial_\omega}, g_{1,i\partial_r}: H_1 \to \R$ be as in Theorem \ref{main-8}.  The function $S_2$ defined in \eqref{tfdef} scales like $\rho^{-\frac{4}{p-1}+d-4}$, vanishes to order at least $d-2$ at $r=0$, and has the same parity in $r$ as $r^{d-2}$.  We may then define
\begin{equation}\label{sodd}
 g_{\partial_r,\partial_r}(t,R) \coloneqq \frac{1}{R^{d-1}} \int_0^R S_2(t,r)\ dr
\end{equation}
for $(t,R)$ in the interior of $H_1$.  The integral $\int_0^R S_1(t,r)\ dr$ vanishes to order at least $d-1$ at the time axis $R=0$, so there is a removable singularity on that axis; by Taylor expansion we see that $g_{\partial_r,\partial_r}-g_{\partial_\omega,\partial_\omega}$ vanishes.  It is also easy to see that $g_{\partial_r,\partial_r}$ is even in $r$ and scales like $\rho^{-\frac{4}{p-1}-2}$, and from the fundamental theorem of calculus we see that $g_{\partial_r,\partial_r}$ obeys \eqref{god}.  If we could show that $g_{\partial_r,\partial_r}$ extends smoothly to the boundary $\{ (0,R): R \neq 0 \}$ of $H_1$, then from Theorem \ref{main-8}(viii) we obtain \eqref{sa-2} at $(1,R)$ for all $R > 0$ (with a gap of at least $\eps \rho^{-\frac{4}{p-1}-2}$); using scaling, symmetry and a limiting argument we would obtain \eqref{sa-2} throughout $H_1$, and we would obtain all the requirements for Theorem \ref{main-7}.

Thus the only remaining difficulty is to ensure the smooth extension.  We argue as in Section \ref{energy-sec}.  By scaling and symmetry it suffices to show that the function $t \mapsto g_{\partial_r,\partial_r}(t,1)$ extends smoothly to $t=0$.  If strict inequality occurs in \eqref{dang}, then from \eqref{satyr}, \eqref{sodd} we can write
$$ g_{\partial_r,\partial_r}(1,R) = \frac{1}{R^{d-1}} \int_0^R r^{d-4-\frac{4}{p-1}} R_k(1/r^2)\ dr - \frac{1}{R^{d-1}} \int_R^\infty
(S_2(1,r) - r^{d-4-\frac{4}{p-1}} R_k(1/r^2))\ dr $$
and hence by rescaling
\begin{align*}
 g_{\partial_r,\partial_r}(t,1) &= t^{-\frac{2}{p-1}-1} g_{\partial_r,\partial_r}(1,t^{-1/2}) \\
&= Y_1(t) + Y_2(t)
\end{align*}
where the functions $Y_1,Y_2: (0,+\infty) \to \R$ are defined by the formulae
$$ Y_1(t) \coloneqq t^{\frac{d-3}{2} - \frac{2}{p-1}} \int_0^{t^{-1/2}} r^{d-4-\frac{4}{p-1}} R_k(1/r^2)\ dr$$
and
$$ Y_2(t) \coloneqq -t^{\frac{d-3}{2} - \frac{2}{p-1}} 
\int_{t^{-1/2}}^\infty (S_2(1,r) - r^{d-4-\frac{4}{p-1}} R_k(1/r^2))\ dr.$$
The function $Y_1$ is a polynomial and thus smoothly extends to $t=0$.  As for $Y_2$, we make the change of variables $r = (st)^{-1/2}$ to write
$$ Y_2(t) = -\frac{1}{2}
\int_0^1 \frac{S_2(st,1) - R_k(st)}{s^{k+1}}\ s^{\frac{2}{p-1} - \frac{d-3}{2} + k} ds.
$$
As in Section \ref{energy-sec}, $\frac{S_2(st,1) - R_k(st)}{s^{k+1}}$ is smooth in $t \in [0,1]$ uniformly in $s \in [0,1]$, and the weight $s^{\frac{2}{p-1} - \frac{d-3}{2} + k}$ is absolutely integrable, so we obtain a smooth extension to $t=0$ as required.  The case when equality occurs in \eqref{dang} is treated by adding a monomial term $A t^{\frac{d-3}{2} - \frac{2}{p-1}}$ to $Y_2$  precisely as in Section \ref{energy-sec}.

It remains to prove Theorem \ref{main-8}.  This will be the objective of the remaining sections of the paper.

\section{Eliminating the angular stress}

Now we turn to eliminating the angular stress $g_{\partial_\omega,\partial_\omega}$ from the problem.  It will be natural to divide into the \emph{stress-subcritical} case $d-3-\frac{4}{p-1} < 0$, the \emph{stress-critical} case $d-3-\frac{4}{p-1}=0$, and the \emph{stress-supercritical} case $d-3-\frac{4}{p-1} > 0$ (note that all three of these cases can occur in the energy-supercritical regime \eqref{energy-supercrit}).

Assume that all the conclusions of Theorem \ref{main-8} are satisfied.  In the stress-subcritical case $d-3-\frac{4}{p-1} < 0$, the exponent $k$ in Theorem \ref{main-8}(ix) is equal to $-1$, thus $R_k$ vanishes, $S_2$ is absolutely integrable, and the condition \eqref{satyr} becomes
$$\int_0^\infty S_2(1,r)\ dr = 0.$$
From Theorem \ref{main-8}(viii) we thus have
$$
\frac{1}{R^{d-1}} \int_R^\infty S_2(1,r)\ dr 
\leq -\frac{ \left(\frac{1}{2} \partial_r g_{1,1}\right)^2 + g_{1,i\partial_r}^2}{g_{1,1}}(1,R) - \eps \rho(1,R)^{-\frac{4}{p-1}-2} 
$$
for any $R>0$. Applying \eqref{tfdef}, \eqref{sa}, we obtain the constraint
$$
\frac{1}{R^{d-1}} \int_R^\infty S_1(1,r)\ dr \leq -\frac{ \left(\frac{1}{2} \partial_r g_{1,1}\right)^2 + g_{1,i\partial_r}^2}{g_{1,1}}(1,R) - \eps \rho(1,R)^{-\frac{4}{p-1}-2} 
$$
on the fields $g_{1,1}, g_{1,i\partial_r}$ for all $R>0$.  By scale invariance, we then have
\begin{equation}\label{con}
\frac{1}{R^{d-1}} \int_R^\infty S_1(t,r)\ dr 
\leq -\frac{ \left(\frac{1}{2} \partial_r g_{1,1}\right)^2 + g_{1,i\partial_r}^2}{g_{1,1}}(t,R) - \eps \rho(t,R)^{-\frac{4}{p-1}-2} 
\end{equation}
for all $t,R>0$.

Now suppose we are in the stress-critical case $d-3 - \frac{4}{p-1}=0$.  Then $k=0$, and from Theorem \ref{main-8}(ix) we have
$$ \lim_{r \to \infty} r S_2(1,r) = 0.$$
As $S_1$ scales like $\rho^{\frac{4}{p-1}+d-4} = \rho^{-1}$, the limit $\lim_{r \to \infty} r S_1(1,r)$ exists; from \eqref{tfdef}, \eqref{sa}, we conclude the constraint
$$ \lim_{r \to \infty} r S_1(1,r) \leq 0.$$

Finally, in the stress-supercritical case $d-3 - \frac{4}{p-1} > 0$, there is no obvious way to extract a constraint on $g_{1,1}, g_{1,i\partial_r}$ from the properties in Theorem \ref{main-8} that involve $g_{\partial_\omega,\partial_\omega}$.

As it turns out, the obstructions listed above to eliminating $g_{\partial_\omega,\partial_\omega}$ are essentially the only ones.  More precisely, Theorem \ref{main-8} is a consequence of

\begin{theorem}[Eighth reduction]\label{main-9}   Then there exist smooth fields $g_{1,1}, g_{1,i\partial_r}: H_1 \to \R$ obeying the following properties:
\begin{itemize}
\item[(i)] One has the positive definite inequality $g_{1,1} > 0$ pointwise on $H_1$.
\item[(ii)] $g_{1,1}$ and $g_{1,i\partial_r}$ scale like $\rho^{-\frac{4}{p-1}}$ and $\rho^{-\frac{4}{p-1}-1}$ respectively.
\item[(vi)]  One has the conservation law \eqref{gtr} on $H_1$ with removable singularity at $r=0$.
\item[(vii)]  The function $g_{1,1}$ is even in $r$, while $g_{1,i\partial_r}$ is odd in $r$. 
\item[(x)]  In the stress-subcritical case, we have the constraint \eqref{con} for all $R,t > 0$ and some $\eps>0$, where $S_1$ is defined by \eqref{FDEF}.  In the stress-critical case, we have the constraint 
$$ \lim_{r \to \infty} r S_1(1,r) < 0.$$
In the stress-supercritical case, we impose no constraint here.
\end{itemize}
\end{theorem}

In the remainder of this section we show how Theorem \ref{main-8} implies Theorem \ref{main-7}.   Let $g_{1,1}, g_{1,i\partial_r}$, be as in Theorem \ref{main-8}.  It will suffice to locate a smooth field $g_{\partial_\omega,\partial_\omega}: H_1 \to \R$, scaling like $\rho^{-\frac{4}{p-1}-2}$, even in $r$, and vanishing at $r=0$, which is strictly positive and such that the function $S_2: H_1 \to \R$ defined by \eqref{tfdef} obeys the properties claimed in Theorem \ref{main-8}(viii), (ix).  

We begin with the stress-critical case $d-3-\frac{4}{p-1}=0$, which is the simplest.  From Theorem \ref{main-8}(x) we can write
\begin{equation}\label{rafc}
 \lim_{r \to \infty} r S_1(1,r) = -c 
\end{equation}
for some $c>0$.  Let $\psi: \R \to [0,1]$ be a smooth even function, supported on $[-2,2]$, that equals one on $[-1,1]$, and choose
$$ g_{\partial_\omega,\partial_\omega}(t,r) \coloneqq \frac{c}{d-1} \rho^{1-d} + A t^{\frac{1-d}{2}} \psi( r / t^2 )$$
for some large $A>0$ to be chosen later.  Clearly $g_{\partial_\omega,\partial_\omega}$ is strictly positive, smooth, even in $r$, and scales like $\rho^{-\frac{4}{p-1}-2} = \rho^{1-d}$.  From \eqref{tfdef}, \eqref{rafc} we see that
$$  \lim_{r \to \infty} r S_1(1,r) = 0.$$
It remains to establish the property in Theorem \ref{main-8}(viii) with (say) $\eps=1$.  That is to say, we need to show that
\begin{equation}\label{bound}
A \frac{1}{R^{d-1}} \int_0^R r^{d-2} \psi(r)\ dr \geq f(R)
\end{equation}
for all $R>0$, where
$$ f(R) \coloneqq 
\frac{ \left(\frac{1}{2} \partial_r g_{1,1}\right)^2 + g_{1,i\partial_r}^2}{g_{1,1}}(1,R) + \rho(1,R)^{-\frac{4}{p-1}-2}
- \frac{1}{R^{d-1}}  \int_0^R (S_1(1,r) + c r^{d-2} \rho^{1-d})\ dr.$$
The function $S_1(1,r) + c r^{d-2} \rho^{1-d}$ scales like $\rho^{-1}$, and by \eqref{rafc} we have
$$ \lim_{r \to \infty} r (S_1(1,r) + c r^{d-2} \rho^{1-d}) = 0,$$
so we have an asymptotic of the form
$$ S_1(1,r) + c r^{d-2} \rho^{1-d} = O(1/r^3)$$
as $r \to \infty$.  In particular, the integral $\int_0^R (S_1(1,r) + c r^{d-2} \rho^{1-d})\ dr$ is bounded in $R$.  The first two terms in the definition of $f(R)$ come from evaluating smooth functions scaling like $\rho^{-\frac{4}{p-1}-2} =\rho^{1-d}$ at $(1,R)$.  As such we conclude a bound of the form
$$ f(R) = O( (1+R)^{1-d} )$$
for all $R > 0$, where the implied constant does not depend on $A$.  On the other hand, from the construction of $\psi$, the expression $\frac{1}{R^{d-1}} \int_0^R r^{d-2} \psi(r)\ dr$ is bounded below by $\frac{1}{d-1}$ when $R \leq 1$ and by $\frac{1}{(d-1)R^{d-1}}$ for $R \geq 1$, so we obtain the required bound \eqref{bound} by choosing $A$ large enough.

A similar argument lets us treat the stress-supercritical case in which $d-3-\frac{4}{p-1}=2k$ for some positive integer $k$, as follows.  The function $S_1$ is smooth and scales like $\rho^{d-4-\frac{4}{p-1}} = \rho^{2k-1}$, and thus we have an asymptotic of the form
\begin{equation}\label{indo}
 S_1(1,r) = r^{2k-1} ( R_{k-1}(1/r^2) + c_k / r^{2k} + O( r^{-2k-2} ) )
\end{equation}
as $r \to +\infty$, for some real number $c_k$ and some polynomial $R_{k-1}$ of degree at most $k-1$.  Let $\psi: \R \to [0,1]$ be a smooth cutoff as before, let $A > 0$ be a large parameter to be chosen later, and set
$$ g_{\partial_\omega,\partial_\omega}(t,r) \coloneqq \left(-\frac{c_k}{d-1} |r|^{1-d} t^k + A |r|^{1-d+2k} \right) (1 - \psi(r/t^2)) + 
 + A t^{\frac{1-d}{2}+k} \psi( r / t^2 ).$$
If $A$ is large enough, it is easy to verify that $g_{\partial_\omega,\partial_\omega}$ is strictly positive, smooth, even in $r$, and scales like $\rho^{-\frac{4}{p-1}-2} = \rho^{1-d+2k}$.  From \eqref{tfdef}, \eqref{indo} we have the asymptotic
$$ S_1(1,r) = r^{2k-1} ( A + R_{k-1}(1/r^2) + O( r^{-2k-2} ) )$$
as $r \to +\infty$.  

It remains to establish the property in Theorem \ref{main-8}(viii) with (say) $\eps=1$.  As before, we rewrite this desired inequality as
\begin{equation}\label{bound-2}
A \frac{1}{R^{d-1}} \int_0^R (r^{d-2} \psi(r) + r^{2k-1} (1-\psi(r)))\ dr \geq f_k(R)
\end{equation}
for all $R>0$, where
$$ f_k(R) \coloneqq 
\frac{ \left(\frac{1}{2} \partial_r g_{1,1}\right)^2 + g_{1,i\partial_r}^2}{g_{1,1}}(1,R) + \rho(1,R)^{-\frac{4}{p-1}-2}
- \frac{1}{R^{d-1}}  \int_0^R \left(S_1(1,r) - c_k \frac{1-\psi(r/t^2)}{r}\right)\ dr.$$
As before, the first two terms $f_k(R)$ come from evaluating a smooth function scaling like $\rho^{-\frac{4}{p-1}-2} =\rho^{1-d+2k}$ at $(1,R)$, while the integrand $S_1(1,r) - c_k \frac{1-\psi(r/t^2)}{r}$ is of size $O( (1+r)^{2k-1} )$.  We conclude that
$$ f_k(R) = O( (1+R)^{1-d+2k} )$$
(with implied constant independent of $A$), while from direct computation we have
$$ \frac{1}{R^{d-1}} \int_0^R (r^{d-2} \psi(r) + r^{2k-1} (1-\psi(r)))\ dr \geq c (1+R)^{1-d+2k}$$
for all $R>0$ and some quantity $c>0$ depending on $d,k,\psi$.  The claim then follows by taking $A$ large enough.

It remains to prove Theorem \ref{main-9}.  This will be the objective of the final section of the paper.

\section{Conclusion of the argument}\label{conclude}

The mass conservation law \eqref{gtr} can be rewritten as
$$\partial_{t} (r^{d-1} g_{1,1}) = 2 \partial_r (r^{d-1} g_{1,i\partial_r}).$$
It is thus clear that this law will be satisfied for $r \neq 0$ (with removable singularity at $r=0$) if one uses the ansatz
\begin{align}
g_{1,1} &= 2 r^{1-d} \partial_r (r^d W) = 2 r \partial_r W + 2 d W \label{wand}\\
g_{1,i\partial_r} &= r^{1-d} \partial_t (r^d W) = r \partial_t W \label{wand-2}
\end{align}
for some smooth function $W: H_1 \to \R$.  In order to obey the conditions (i), (ii), (vii) of Theorem \ref{main-9}, we should impose the following conditions on $W$:
\begin{itemize}
\item[(i)]  One has $\partial_r (r^d W(t,r)) > 0$ for all $r > 0$ and $t \geq 0$.  Furthermore, $W(1,0) > 0$.
\item[(ii)]  $W$ scales like $\rho^{-\frac{4}{p-1}}$.
\item[(vii)]  $W$ is even in $r$.
\end{itemize}
It is clear that if $W$ is smooth and obeys the above properties (i), (ii), (vii), and $g_{1,1}, g_{1,i\partial_r}$ are then defined by \eqref{wand}, \eqref{wand-2}, then the properties (i), (ii), (vi), (vii) of Theorem \ref{main-9} are satisfied.  Such a function $W$ is easy to construct, indeed one can just take $W(t,r) \coloneqq \rho^{-\frac{4}{p-1}}$ (noting from the energy supercriticality hypothesis \eqref{energy-supercrit} that $d - \frac{4}{p-1} > 2 > 0$, hence the derivative $\partial_r (r^d W) = (\frac{d}{r} - \frac{4}{p-1} \frac{r^3}{\rho^4}) r^d W$ is positive for $r > 0$).  This already establishes Theorem \ref{main-9} in the stress-supercritical case $d-3 - \frac{4}{p-1} > 0$.

It remains to handle the stress-critical case $d-3 - \frac{4}{p-1} = 0$ and the stress-subcritical case $d-3 - \frac{4}{p-1} < 0$.  Here the difficulty is that there is an additional constraint in Theorem \ref{main-9}(x) that needs to be satisfied.  If one sets $W^0 \coloneqq \rho^{-\frac{4}{p-1}}$ and defines the initial fields $g^0_{1,1}, g^0_{1,i\partial_r}$ by the formulae \eqref{wand}, \eqref{wand-2}, that is to say that
\begin{align}
g_{1,1}^0 &= 2 r \partial_r W^0 + 2 d W^0 \label{wand-3}\\
g_{1,i\partial_r}^0 &= r \partial_t W^0 \label{wand-4}
\end{align}
and then defines the initial field $S_1^0$ by the analogue of \eqref{FDEF}, namely
$$ S_1^0 \coloneqq \frac{1}{4} r^{d-1} \left( \partial_r \left(\partial_r^2 + \frac{d-1}{r} \partial_r\right) g^0_{1,1} + 2 \partial_t g^0_{1,i\partial_r} \right ),$$
then there is no guarantee that the constraint in Theorem \ref{main-9}(x) will be obeyed for these choices of $g_{1,1}, g_{1,i\partial_r}$.  Instead, we select a smooth function $\psi: \R \to \R$ supported on $[-1,1]$, such that $\psi''(t) \geq 0$ for all $t \geq 0$ and $\psi''(t) = 1$ for $0 \leq t \leq 1/2$, let $\delta>0$ be an even smaller parameter, and let $W : H_1 \to \R$ be the function defined for all $(t,r) \in H_1$ by the formula
$$ W(t,r) \coloneqq W^0(t,r) - \delta^{3/2} \rho^{-\frac{4}{p-1}} \psi\left( \frac{t}{\delta \rho} \right),$$
and then define $g_{1,1}$, $g_{1,i\partial_r}$, $S_1$ by \eqref{wand}, \eqref{wand-2}, \eqref{FDEF}.
Clearly $W$ obeys the required properties (ii), (vii).  We now claim that the property (i) also holds if $\delta$ is small enough.  Note that $W(t,r)$ is equal to $W^0(t,r)$ unless $t = O(\delta \rho)$, thus it suffices to verify (i) in the regime $t = O(\delta \rho)$.  By rescaling we may normalise $r=1$ and $t=O(\delta)$.  In this regime we have
$$ \partial_r (r^d W(t,r)) = \partial_r (r^d W^0(t,r)) + O(\delta^{1/2})$$
and from the fact that $W^0$ obeys (i), the quantity $\partial_r (r^d W^0(t,r))$ is bounded away from zero uniformly in $\delta$ in the regime $r=1$, $t=O(\delta)$, so the claim follows.

We now claim that Theorem \ref{main-9}(x) holds for $\delta$ small enough.  In the stress-supercritical case there is nothing to prove.  In the remaining cases, we need to study the quantity $S_1(t,r)$.  By construction, this quantity is equal to $S_1^0(t,r)$ except in the regime $r = O(\delta \rho)$.  Now we rescale and study $S_1(t,1)$ in the regime $r = O(\delta)$.  From \eqref{wand}, \eqref{wand-2} we have
$$ g_{1,1} = g_{1,1}^0 - 2 r \delta^{3/2} \partial_r \left(\rho^{-\frac{4}{p-1}} \psi\left( \frac{t}{\delta \rho} \right)\right) + 2 d \delta^{3/2} \rho^{-\frac{4}{p-1}} \psi\left( \frac{t}{\delta \rho} \right)$$
and
$$ g_{1,i\partial_r} = g_{1,i\partial_r}^0 + r \delta^{3/2} \partial_t \left( \rho^{-\frac{4}{p-1}} \psi\left( \frac{t}{\delta \rho} \right)\right).$$
Using the identities $\partial_t \rho = \frac{t}{2\rho^3}; \quad \partial_r \rho = \frac{r^3}{\rho^3}$ we can obtain the bounds
\begin{align}
\partial_r^j g_{1,1} &= \partial_r^j g_{1,1}^0 + O( \delta^{3/2}) \label{hip}\\
g_{1,i\partial_r} &= g_{1,i\partial_r}^0 + O(\delta^{1/2}) \label{hop}\\
\partial_t g_{1,i\partial_r} &= - 2 \delta^{-1/2} \rho^{-\frac{4}{p-1}-2} \psi''\left(\frac{t}{\delta \rho}\right) + O(1)\label{hup}
\end{align}
for $j=0,1,2,3$ in the regime $r=1$, $t = O(\delta)$.  In particular, from \eqref{FDEF} we have the bounds
\begin{equation}\label{s11t}
 S_1(1,t) = -\delta^{-1/2} \rho^{-\frac{4}{p-1}-2} \psi''\left(\frac{t}{\delta \rho}\right)  + O(1)
\end{equation}
in the region $r=1$, $t=O(\delta)$; this bound is also true in the larger range $r=1$, $t=O(1)$ since $S_1 = S_1^0$ and $\psi'' = 0$ when $t$ is much larger than $\delta$.  In particular, for $\delta$ small enough we have
$$ \lim_{t \to 0^+} S_1(t,1) < 0;$$
as $S_1$ scales like $\rho^{-\frac{4}{p-1} + d-2}$, this is equivalent to
which by rescaling is equivalent to
$$ \lim_{r \to \infty} \rho^{\frac{4}{p-1}-d+2} S_1(1,r) < 0.$$
In the stress-critical case $d-3 - \frac{4}{p-1}=0$, this gives Theorem \ref{main-9}(x).  Now suppose we are in the stress-subcritical case $d-3 - \frac{4}{p-1}<0$.  From \eqref{s11t}, we have the bounds
$$ - \frac{1}{t} \int_0^t S_1(t',1) \left(\frac{t'}{t}\right)^{\frac{2}{p-1}-\frac{d-1}{2}}\ dt' \gg \delta^{-1/2} \left(1 + \frac{t}{\delta}\right)^{\frac{d-3}{2} - \frac{2}{p-1}} - O(1)$$
for all $0 < t \leq 1$ (note in the stress-subcritical case that the exponent $\frac{2}{p-1}-\frac{d-1}{2}$ is at least $-1$).  We have
$$ \delta^{-1/2} \left(1 + \frac{t}{\delta}\right)^{\frac{d-3}{2} - \frac{2}{p-1}} \gg \delta^{\frac{2}{p-1} - \frac{d-2}{2}}.$$
By energy supercriticality \eqref{energy-supercrit}, the exponent here is negative, and thus if $\delta$ is small enough we have
$$ - \frac{1}{t} \int_0^t S_1(t',1) \left(\frac{t'}{t}\right)^{\frac{2}{p-1}-\frac{d-1}{2}}\ dt' \gg \delta^{\frac{2}{p-1}-\frac{d-2}{2}}$$
for all $0 < t \leq 1$.  As $S_1$ scales like $\rho^{-\frac{4}{p-1}+d-2}$, this bound is equivalent to
$$ - \frac{1}{R^{d-1}} \int_R^\infty S_1(1,r)\ dr \gg \delta^{\frac{2}{p-1}-\frac{d-2}{2}} \rho^{-\frac{4}{p-1}-2}$$
for $1 \leq R < \infty$; since $S_1(1,r)=S_1^0(1,r) = O(1)$ when $0 \leq R \leq 1$, we conclude that this bound also holds for $0 < R < 1$ if $\delta$ is small enough.  Meanwhile, from \eqref{hip}, \eqref{hop} we have
$$ \frac{ \left(\frac{1}{2} \partial_r g_{1,1}\right)^2 + g_{1,i\partial_r}^2}{g_{1,1}}(t,1) = O(1)$$
for $0 < t \leq 1$, and hence by rescaling
$$ \frac{ \left(\frac{1}{2} \partial_r g_{1,1}\right)^2 + g_{1,i\partial_r}^2}{g_{1,1}}(1,R) = O(\rho^{-\frac{4}{p-1}-2})$$
for $1 \leq R < \infty$; since $g_{1,1}(1,R) = g_{1,1}^0(1,R) \gg 1$, $g_{1,i\partial_r}(1,R) = g_{1,i\partial_r}^0(1,R) = O(1)$, and  
$\partial_r g_{1,1}(1,R) = \partial_r g_{1,1}^0(1,R) = O(1)$ for $0 \leq R \leq 1$, this bound also holds for $0 < R \leq 1$.  We conclude that for $\delta$ small enough, the conclusion of Theorem \ref{main-9}(x) holds (with $\eps=1$) in the stress sub-critical case.  This covers all the cases required for Theorem \ref{main-9}, and thus (finally!) completes the proof of Theorem \ref{main}.

\appendix

\section{Proof of Nash-type embedding theorem}\label{gad-app}

The purpose of this appendix is to prove Proposition \ref{gadc}.

We can use the hypothesis in Proposition \ref{gadc}(iv) to make a ``gauge transformation'' to reduce to the case when the components $G_{1,i\partial_{x_j}}$ vanish:

\begin{proposition}\label{gauge}  In order to prove Proposition \ref{gadc}, it suffices to do so under the additional hypothesis that $G_{1,i \partial_{x_j}}$ vanishes identically for all $j=1,\dots,d$, and in which we now require $\alpha=0$ in \eqref{uts}.
\end{proposition}

We remark from \eqref{div-3} that the vanishing of $G_{1,i\partial_{x_j}}$ also implies the vanishing of $G_{\partial_{x_j}, i \partial_{x_k}}$.

\begin{proof}  Let the hypotheses be as in Proposition \ref{gadc}, and let $\vec g: H_d \to \R^d$ denote the vector field
$$ \vec g \coloneqq \left(\frac{G_{1,i \partial_{x_j}}}{G_{1,1}}\right)_{j=1}^d.$$ 
From hypothesis (iv) we know that $\vec g$ is curl-free, so in particular
$$ \int_\gamma \vec g(t,x) \cdot ds = 0$$
for all $t > 0$ and all closed curves $\gamma$ in $\R^d$, where $ds$ is the length element.  Taking limits as $t \to 0$, we conclude that
$$ \int_\gamma \vec g(0,x) \cdot ds = 0$$
for all $t > 0$ and all closed curves $\gamma$ in $\R^d \backslash \{0\}$. In particular, $\vec g(0,\cdot)$ is exact, and so we can find a smooth function $P_0: \R^d \backslash \{0\} \to \R$ such that
\begin{equation}\label{vpg}
 \vec g(0,x) = \nabla P_0(x)
\end{equation}
for all $x \in \R^d \backslash \{0\}$.  Observe from \eqref{scaling} that the vector field $\vec g$ has the homogeneity
\begin{equation}\label{hom}
\vec g(4t, 2x) = \frac{1}{2} \vec g(t,x)
\end{equation}
for all $(t,x) \in H_d$.  In particular, \eqref{vpg} continues to hold when $P_0$ is replaced by the rescaling $x \mapsto P_0(2x)$.  Integrating, we conclude that
\begin{equation}\label{hurt}
 P_0(2x) = P_0(x) + \alpha
\end{equation}
for all $x \in \R^d \backslash \{0\}$ and some $\alpha \in \R$.

From \eqref{hom} and the smoothness of $\vec g$ up to the boundary of $H_d$, we see for fixed $t \geq 0$ that one has the asymptotic
$$ \vec g(t,x) - \vec g(0,x) = O(1/|x|^2)$$
as $x \to \infty$, and similarly for all spacetime derivatives of $\vec g$ (in fact one gains additional powers of $|x|$ with each derivative).  If we then define the function $P: H_d \to \R$ by
$$ P(t,x) \coloneqq P_0(x) - \int_\gamma (\vec g(t,x) - \vec g(0,x)) \cdot ds $$
where $\gamma$ is an arbitrary curve from $x$ to $\infty$ in $\R^d \backslash \{0\}$ that is eventually linear, then we see from Stokes' theorem that $P$ is well-defined, and it is clear from construction that $P$ is smooth and obeys the identity
$$ \vec g(t,x) = \nabla P(t,x)$$
for all $(t,x) \in H_d$.  Furthermore, from \eqref{scaling} and \eqref{hurt} we see that
\begin{equation}\label{homogp}
 P(4t, 2x) = P(t,x) + \alpha
\end{equation}
for all $(t,x) \in \R^d$.

We now introduce the ``gauge transformed'' matrix $G' = (G'_{D_1,D_2})_{D_1,D_2 \in {\mathcal D}}$ by setting
\begin{align*}
G'_{1,1} = G'_{i,i} &\coloneqq G_{1,1} \\
G'_{1,i} = G'_{i,1} &\coloneqq 0 \\
G'_{1,D_1} = G'_{D_1,1} = G'_{i,iD_1} = G'_{iD_1,i} &\coloneqq G_{1,D_1} \\
G'_{1,iD_1} = G'_{iD_1,1} = - G_{i,D_1} = -G_{D_1,i} &\coloneqq G_{1,D_1} - G_{1,1} D_1 P \\
G'_{D_1,D_2} = G'_{iD_1,iD_2} &\coloneqq G_{D_1,D_2} - G_{1,iD_2} D_1 P - G_{1,iD_1} D_2 P + (D_1 P) (D_2 P) G_{1,1} \\
G'_{D_1,iD_2} = G_{iD_2,D_1} &\coloneqq G_{D_1,iD_2} - (D_2 P) G_{1,D_1} + (D_1 P) G_{1,D_2}
\end{align*}
for $D_1,D_2 \in {\mathcal D}_\R \backslash \{1\}$.  The motivation for this matrix is that the requirement \eqref{gc} can be seen to be equivalent to the requirement
\begin{equation}\label{gcp}
 G'_{D_1,D_2}(t,x) = \langle D_1 (u e^{iP})(t,x), D_2 (u e^{iP})(t,x) \rangle_{\C^m}
\end{equation}
for $D_1,D_2 \in {\mathcal D}$, as can be seen from many applications of the product and Leibniz rules.

It is easy to see that $G'$ is smooth and real symmetric and obeys the scaling relation \eqref{scaling}.  We observe the identity
$$ \sum_{D_1,D_2 \in {\mathcal D}} G'_{D_1,D_2} a_{D_1} a_{D_2} = \sum_{D_1,D_2 \in {\mathcal D}} G_{D_1,D_2} b_{D_1} b_{D_2}$$
for all real numbers $a_D, D \in {\mathcal D}$, where
\begin{align*}
b_1 &\coloneqq a_1 - \sum_{D \in {\mathcal D}_\R} a_{iD} DP \\
b_i &\coloneqq a_1 + \sum_{D \in {\mathcal D}_\R} a_D DP \\
b_{D_1} &\coloneqq a_{D_1} \\
b_{iD_1} &\coloneqq a_{iD_1}.
\end{align*}
From this we see that $G'$ is strictly positive definite, and thus obeys the property (i).  Routine calculation shows that it also obeys the conditions (ii), (iii), (iv), and that the components $G'_{1,i\partial_{x_j}}$ vanish for $j=1,\dots,d$.  By hypothesis, we may thus find 
a smooth function $u': H_d \to \C^m$ that is nowhere vanishing and obeying the discrete self-similarity \eqref{uts} with $\alpha$ replaced by $0$, such that
$$
 G'_{D_1,D_2}(t,x) = \langle D_1 u'(t,x), D_2 u'(t,x) \rangle_{\C^m}
$$
for all $(t,x) \in H_d$ and all $D_1,D_2\in {\mathcal D}$ other than $(D_1,D_2) = (\partial_t, \partial_t), (i \partial_t, i \partial_t)$.
Furthermore, the function $\theta: H_d/T^\Z \to \mathbf{CP}^{m-1}$, formed by descending the map $\pi \circ u': H_d \to \mathbf{CP}^{m-1}$ to $H_d/T^\Z$, is a smooth embedding.  If we then set $u \coloneqq u' e^{iP}$, one checks from the equivalence of \eqref{gc} and \eqref{gcp} that that $u$ obeys all the properties required for Proposition \ref{gadc}.
\end{proof}

It remains to prove Proposition \ref{gadc} under the additional hypothesis that $G_{1,i \partial_{x_j}}=0$ and with the requirement $\alpha=0$.  It will be convenient to work with a reduced ``basis'' of components of $G$, in order to eliminate the various constraints between the components of $G$.  Let ${\mathcal P} \subset {\mathcal D}^2$ denote the following set of pairs in ${\mathcal D}$:
\begin{align*}
 {\mathcal P} &\coloneqq \{ (1,D): D = 1, i \partial_{x_1}, \dots, i \partial_{x_d}, i \partial_t\} \\
&\quad \cup \{ (\partial_{x_j}, \partial_{x_k}): 1 \leq j \leq k \leq d \} \\
&\quad \cup \{ (\partial_{x_j}, \partial_t): 1 \leq j \leq d \}
\end{align*}
and then define the reduction $G_{\mathcal P}: H_d \to \R^{\mathcal P}$ of the matrix $G$ as
\begin{equation}\label{gr}
 G_{\mathcal P} \coloneqq (G_{D_1,D_2})_{(D_1,D_2) \in {\mathcal P}}
\end{equation}
and the Gram-type matrix $G_{\mathcal P}[u,v]: H_d \to \R^{\mathcal P}$ of two smooth functions $u,v: H_d \to \C^m$ for some $m \geq 1$ by the formula
$$ G_{\mathcal P}[u,v] \coloneqq (\langle D_1 u, D_2 v \rangle_{\C^m})_{(D_1,D_2) \in {\mathcal P}}.$$
Observe from the hypotheses \eqref{div}, \eqref{div-2}, \eqref{div-3}, \eqref{div-4} (as well as the symmetry $G_{D_1,D_2} = G_{D_2,G_1}$) on the matrix $G$, as well as the analogous identities \eqref{da}, \eqref{da-2}, \eqref{bongo}, \eqref{bango} (as well as the symmetry $\langle D_1 u, D_2 u \rangle_{\C^m} = \langle D_2 u, D_1 \rangle_{\C^m}$) on the Gram-type matrix $G[u,u]$, that if $u$ obeyed the equations
\begin{equation}\label{gpu}
 G_{\mathcal P}[u,u] = G_{\mathcal P}
\end{equation}
(that is to say, \eqref{gc} holds for all $(D_1,D_2) \in {\mathcal P}$) then in fact one has \eqref{gc} for all pairs $(D_1,D_2)$ in ${\mathcal D}^2$ other than $(\partial_t,\partial_t)$ and $(i \partial_t, i \partial_t)$.  Thus, our task reduces to that of locating a smooth, nowhere vanishing map $u: H_d \to \C^m$ which obeys the discrete self-similarity \eqref{uts} and the equation \eqref{gpu}.

In order to avoid technicalities involving elliptic theory for manifolds with boundary, it will be convenient to replace the half-space $H_d$ with the punctured spacetime $\R \times \R^d \backslash \{(0,0)\}$, so that the quotient
$$ M \coloneqq (\R \times \R^d \backslash \{(0,0)\})/T^\Z$$
is now a smooth compact manifold without boundary.  More precisely, we will show

\begin{proposition}\label{guip}  Let $G_{\mathcal P} = (G_{D_1,D_2})_{(D_1,D_2) \in {\mathcal P}}$ be a tuple of smooth functions $G_{D_1,D_2}: \R \times \R^d \backslash \{(0,0)\} \to \R$ obeying the scaling law \eqref{scaling}.  Suppose also that the fields $G_{1,i\partial_{x_j}}$ vanish for $j=1,\dots,d$, and that the $d+1 \times d+1$ matrix
\begin{equation}\label{matrix}
 (G_{D_1,D_2})_{D_1,D_2 \in \{1,\partial_{x_1},\dots,\partial_{x_d}\}}
\end{equation}
is strictly positive definite on all of $\R \times \R^d \backslash \{0\}$, where we define
$$ G_{1,\partial_{x_j}} = G_{\partial_{x_j},1} \coloneqq \frac{1}{2} \partial_{x_j} G_{1,1}$$
for $j=1,\dots,d$ and
$$ G_{\partial_{x_k}, \partial_{x_j}} \coloneqq G_{\partial_{x_j}, \partial_{x_k}}$$
for $1 \leq j < k \leq d$.  Then, if $m$ is an integer that is sufficiently large depending on $d$, there exists a smooth nowhere vanishing function $u: \R \times \R^d \backslash \{(0,0)\} \to \C^m$ obeying \eqref{uts} with $\alpha=0$ such that the map $\pi \circ u$ is a smooth embedding of $M$ into $\mathbf{CP}^{m-1}$, and such that
$$ G_{{\mathcal P}}[u,u] = G_{\mathcal P}$$
on all of $\R \times \R^d \backslash \{(0,0)\}$.
\end{proposition}

We now explain why Proposition \ref{guip} gives us Proposition \ref{gadc}.  Let $G_{D_1,D_2}$, $D_1,D_2 \in {\mathcal D}$ be as in that proposition, with $G_{1,i \partial_{x_j}}=0$.  For each $D_1,D_2 \in {\mathcal D}$, the function $\rho^{\frac{4}{p-1} + \operatorname{ord}(D_1) + \operatorname{ord}(D_2)} G_{D_1,D_2}$ is $T$-invariant and may thus be viewed as a smooth function on the quotient space $H_d/T^\Z$.  Using the extension theorem\footnote{One can also use the classical extension theorem of Whitney \cite{whitney-ext}.} of Seeley \cite{seeley}, we may smoothly extend this function to the larger space $(\R \times \R^d \backslash \{(0,0)\})/T^\Z$; lifting this extension back up to $\R \times \R^d \backslash \{(0,0)\}$ and dividing by $\rho^{\frac{4}{p-1} + \operatorname{ord}(D_1) + \operatorname{ord}(D_2)}$, we obtain a smooth extension of $G_{D_1,D_2}$ for $(D_1,D_2) \in {\mathcal P}$ from $H_d$ to $\R \times \R^d \to \{(0,0)\}$ that continues to obey the scaling properties \eqref{scaling}.  Of course we can arrange matters so that one retains the symmetry property $G_{D_1,D_2}=G_{D_2,D_1}$ with this extension, as well as the vanishing property $G_{1,i \partial_{x_j}} = 0$.  By continuity, the matrix \eqref{matrix} will remain strictly positive definite in an open neighbourhood of $H_d$.  By smoothly interpolating the $G_{D_1,D_2}$ with another set of functions for which the matrix \eqref{matrix} is strictly positive definite everywhere (while also still obeying \eqref{scaling}; this is easily achieved by keeping the diagonal terms $G_{1,1}$, $G_{\partial_{x_j}, \partial_{x_j}}$ large and positive), one can assume without loss of generality that \eqref{matrix} is in fact positive definite on \emph{all} of $\R \times \R^d \backslash \{(0,0)\}$.  If one now applies Proposition \ref{guip} and then restricts back to $H_d$, one obtains the claim.

It remains to establish Proposition \ref{guip}.
If we knew that the component $G_{1,i\partial_t}$ of $G$ vanished (in addition to the vanishing of $G_{1,i\partial_{x_j}}$ that is already assumed), one could obtain this claim immediately from Proposition \ref{gad}, by embedding $\R^m$ into $\C^m$ and noting that the inner products $\langle u, i \partial_{x_j} u \rangle_{\C^m}$ and $\langle u, i \partial_t u \rangle_{\C^m}$ automatically vanish if $u$ takes values in $\R^m$.  (In this case, we could also recover the $(\partial_t, \partial_t)$ case of \eqref{gc}.)  Thus the only obstacle to address is the non-vanishing of $G_{1,i\partial_t}$.  Our strategy, inspired by the usual proofs of the Nash embedding theorem, will be to modify $G_{\mathcal P}$ by subtracting the contribution of a suitable ``short map'' that is designed to mostly eliminate the $G_{1,i\partial_t}$ component (while creating only small perturbations in the remaining components of $G_{\mathcal P}$), and then use the perturbative argument\footnote{One could also use the Nash-Moser iteration scheme here, although this would be more complicated techncially.} of Gunther \cite{gunther} to construct a solution $u$ for this perturbative version of $G_{\mathcal P}$.

We turn to the details.  The map $u \mapsto G_{\mathcal P}[u,u]$ defined by \eqref{gr} is quadratic in $u$, rather than linear.  Nevertheless, it does have the following very convenient additivity property: given two maps $u_1: \R \times \R^d \backslash \{(0,0)\} \to \C^{m_1}$ and $u_2: \R \times \R^d \backslash \{(0,0)\} \to \C^{m_2}$ into two finite-dimensional complex vector spaces, one has the identity
\begin{equation}\label{add}
G_{\mathcal P}[(u_1,u_2),(u_1,u_2)] = G_{\mathcal P}[u_1,u_1] + G_{\mathcal P}[u_2,u_2]
\end{equation}
where the \emph{pairing} $(u_1,u_2): \R \times \R^d \backslash \{(0,0)\} \to \C^{m_1+m_2}$ of $u_1, u_2$ is the map defined by the formula
$$ (u_1,u_2)(t,x) \coloneqq (u_1(t,x), u_2(t,x))$$
where we identify $\C^{m_1} \times \C^{m_2}$ with $\C^{m_1+m_2}$ in the obvious fashion.  Note also that if $u_1,u_2$ are smooth and obey \eqref{uts} with $\alpha=0$, then the pairing $(u_1,u_2)$ does also; and if one of $u_1,u_2$ is an embedding and nowhere vanishing and the other is merely a smooth map that is allowed to vanish, then the pairing $(u_1,u_2)$ will be an embedding that is nowhere vanishing.  

Next, we (again inspired by the usual proofs of the Nash embedding theorem) define a smooth map $u: \R \times \R^d \backslash \{(0,0)\} \to \C^m$ to be \emph{free} if, for any $(t,x) \in \R \times \R^d \backslash \{(0,0)\}$, the vectors $u(t,x)$, $\partial_{x_j} u(t,x)$ (for $1 \leq j \leq d$), $\partial_t u(t,x)$, $\partial_{x_j} \partial_{x_k} u(t,x)$ (for $1 \leq j \leq k \leq d$), and $\partial_{x_j} \partial_t u(t,x)$ (for $1 \leq j \leq d$) are all linearly independent over the complex numbers $\C$ in $\C^m$.  We observe that if $m$ is sufficiently large (depending only on $d$), then there is at least one free map into $\C^m$ that obeys the discrete self-similarity \eqref{uts}.  Indeed, from the Whitney embedding theorem there is a smooth embedding $v\colon M \to \R^{m_0}$ whenever $m_0$ is sufficiently large depending on $d$.  If we then define the map $w\colon M \to \R^{1 + m_0 + \binom{m_0}{2}}$ by the formula
$$ w \coloneqq ( 1, (v_j)_{1 \leq j \leq m_0}, (v_j v_k)_{1 \leq j \leq k \leq m_0} )$$
where $v_1,\dots,v_{m_0}\colon \R \times \R^d \backslash \{(0,0)\}/T^\Z \to \R$ are the components of $v$, then one verifies from the chain rule and the immersed nature of $v$ that $w$ is free over $\R$, and hence free over $\C$ if one embeds $\R^{1 + m_0 + \binom{m_0}{2}}$ into $\C^{1 + m_0 + \binom{m_0}{2}}$.  If one then defines the map $u_0: H_d \to \C^{1 + m_0 + \binom{m_0}{2}}$ by the formula
$$ u_0(t,x) \coloneqq \rho^{-\frac{2}{p-1}} w( \pi(t,x) )$$
we see from a further application of the chain rule that $u_0$ is smooth, free, nowhere vanishing, and obeys the discrete self-similarity relation \eqref{uts}. By multiplying $u_0$ by a sufficiently small positive constant (which does not affect the properties of $u$ stated above), and using the compactness of $M$ and the positive definiteness of the $(d+2)\times (d+2)$ matrix-valued function $(G_{D_1,D_2})_{D_1,D_2 \in {\mathcal D}_\R}$, we can also assume that $u_0$ is a \emph{short map} in the sense that the $(d+2) \times (d+2)$ matrix-valued function
$$
\left( G_{D_1,D_2} - \langle D_1 u_0, D_2 u_0 \rangle_{\C^{1 + m_0 + \binom{m_0}{2}}} \right)_{D_1,D_2 \in {\mathcal D}_\R}$$
is strictly positive definite on all of $\R \times \R^d \backslash \{(0,0)\}$.  Applying Proposition \ref{gad}, we see (for $m_1$ sufficiently large depending on $d$) we may find a smooth nowhere vanishing map $u_1: \R \times \R^d \backslash \{(0,0)\} \to \R^{m_1}$, obeying the discrete self-similarity property \eqref{uts} with $\alpha=0$, with $u_1 / \|u_1\|_{\R^m}$ a smooth embedding of $M$ into $S^{m_1-1}$, such that
\begin{equation}\label{gd-ident}
G_{D_1,D_2} - \langle D_1 u_0, D_2 u_0 \rangle_{\C^{1 + m_0 + \binom{m_0}{2}}} = \langle D_1 u_1, D_2 u_1 \rangle_{\C^{m_1}}
\end{equation}
on $\R \times \R^d \backslash \{(0,0)\}$ for all $D_1,D_2 \in {\mathcal D}_\R$.  This identity also is obeyed when $(D_1,D_2) = (1,i \partial_{x_j})$ for some $j=1,\dots,d$, since all three terms in the identity vanish in this case.  On the other hand, \eqref{gd-ident} can fail when $(D_1,D_2) = (1, i\partial_t)$, since $G_{1,i\partial_t}$ is not assumed to vanish.  In particular, the vector-valued function
$$ G_{\mathcal P} - G_{\mathcal P}[u_0,u_0] - G_{\mathcal P}[u_1,u_1]$$
has all components vanishing except for the $(1,i\partial_t)$ component, which is equal to $G_{1,i\partial_t}$.  To address this remaining component, we proceed by the following argument.  Using a smooth partition of unity, we can find a finite number $a_1,\dots,a_k: \R \times \R^d \backslash \{(0,0)\} \to \R$ of smooth functions, each of which is supported in a ball of radius $1/1000$ in the region $\{ (t,x) \in H_d: \frac{1}{2} \leq \rho \leq 2 \}$, such that
\begin{equation}\label{nax}
1 = \sum_{n \in \Z} \sum_{l=1}^k a_l^2(T^{-n}(t,x))
\end{equation}
for all $(t,x) \in H_d$, where $k$ depends only on $d$.  Meanwhile, the function $\rho^{\frac{4}{p-1} + 2} G_{1,i\partial_t}(t,x)$ is $T$-invariant and thus descends to a smooth function of $H_d/T^\Z$.  This function can be written as the difference of two squares $f_+^2-f_-^2$ for some smooth $f_\pm: H_d/T^\Z \to \R$ (e.g. by setting $f_-$ to be a large positive constant and then solving for $f_+$), thus
$$ G_{1,i\partial_t}(t,x) = \rho^{-\frac{4}{p-1} - 2} f_+(\pi(t,x))^2 - \rho^{-\frac{4}{p-1} - 2} f_-(\pi(t,x))^2.$$
Multiplying this with \eqref{nax}, we obtain the decomposition
$$ G_{1,i\partial_t}(t,x) = \sum_{n \in \Z} \sum_{l=1}^k 2^{-(\frac{4}{p-1}+2) n} b_{l,+}^2(T^{-n}(t,x))- 2^{-(\frac{4}{p-1}+2) n} b_{l,-}^2(T^{-n}(t,x)) $$
where
$$ b_{l,\pm}(t,x) \coloneqq a_l(t,x) \rho^{-\frac{2}{p-1}-1} f_\pm(\pi(t,x)).$$
Note that for fixed $l$, the functions $b_{l,+}^2(T^{-n}(t,x))$ have disjoint supports as $n$ varies, and similarly for $b_{l,-}^2(T^{-n}(t,x))$.

Next, let $\eps>0$ be a small parameter to be chosen later, and let $u_{2,\eps}: H_d \to \C^{2k}$ be the map
$$ u_{2,\eps}(t,x) \coloneqq \left( \left( \sum_{n \in \Z} \eps 2^{-\frac{2n}{p-1}} b_{l,+}(T^{-n}(t,x)) e^{i \frac{t}{\eps 4^n}^2} \right)_{l=1}^k, -\left( \sum_{n \in \Z} \eps 2^{-\frac{2n}{p-1}} b_{l,-}(T^{-n}(t,x)) e^{i \frac{t}{\eps 4^n}^2} \right)_{l=1}^k \right).$$
One can check that $u_{2,\eps}$ is smooth and obeys the discrete self-similarity property \eqref{uts}.  Direct computation using \eqref{add} and the chain and product rules gives the identity
$$ G_{\mathcal P} - G_{\mathcal P}[u_0,u_0] - G_{\mathcal P}[u_1,u_1] - G_{\mathcal P}[u_{2,\eps},u_{2,\eps}] = \eps^2 H_{\mathcal P}$$
where $H_{\mathcal P} = (H_{D_1,D_2})_{(D_1,D_2) \in {\mathcal P}}$ is a smooth function from $H_d$ to $\C^{\mathcal P}$ that is independent of $\eps$ and obeys the scaling property \eqref{scaling}.  The precise value of $H_{\mathcal P}$ is not important for our purposes, but for sake of explicitness we can evaluate the components of this matrix to be given by the formulae
\begin{align*}
H_{1,1}(t,x) &= - \sum_\pm \sum_{n \in \Z} \sum_{l=1}^k 2^{-\frac{4}{p-1} n} b_{l,\pm}^2(T^{-n}(t,x)) \\
H_{1,i\partial_{x_j}}(t,x) &= 0 \\
H_{1,i \partial_t}(t,x) &= 0 \\
H_{\partial_{x_j}, \partial_{x_{j'}}}(t,x) &= -\sum_\pm \sum_{n \in \Z} \sum_{l=1}^k 2^{-(\frac{4}{p-1}+2) n} (\partial_{x_j} b_{l,\pm} \partial_{x_{j'}} b_{l,\pm}) (T^{-n}(t,x))\\
H_{\partial_{x_j}, \partial_t}(t,x) &= H_{\partial_t, \partial_{x_j}}(t,x) \\
&= -\sum_\pm \sum_{n \in \Z} \sum_{l=1}^k 2^{-(\frac{4}{p-1}+3) n} (\partial_{x_j} b_{l,\pm} \partial_t b_{l,\pm}) (T^{-n}(t,x))
\end{align*}
for $j,j' = 1,\dots,d$.
It is important here that the pairs $(\partial_t, \partial_t)$, $(i \partial_t, i \partial_t)$ do not appear in ${\mathcal P}$, as these would introduce terms in $H_{\mathcal P}$ that are of order $1/\eps^4$, which is unacceptably large for our purposes.

Proposition \ref{guip} (and hence Proposition \ref{gadc}) may now be deduced from the following perturbative claim:

\begin{proposition}\label{gippy}  Let the notation and hypotheses be as above.  If $\eps>0$ is sufficiently small, then there exists a smooth map $u_{0,\eps}: \R \times \R^d \backslash \{(0,0)\} \to \C^{1 + m_0 + \binom{m_0}{2}}$ obeying the discrete self-similarity property \eqref{uts} with $\alpha=0$, such that
$$ G_{\mathcal P}[u_{0,\eps},u_{0,\eps}] = G_{\mathcal P}[u_0,u_0] + \eps^2 H_{\mathcal P}.$$
\end{proposition}

Indeed, one can now take $u$ to be the tuple $u \coloneqq (u_{0,\eps}, u_1, u_{2,\eps})$ for a sufficiently small $\eps$, giving the claim (for $m$ large enough).  Note that as $u_1$ was already a smooth non-vanishing embedding, $u$ will be also, regardless of how badly $u_{0,\eps}$ and $u_{2,\eps}$ vanish or fail to be an embedding.

It remains to prove Proposition \ref{gippy}.  In order to be able to work on the compact manifold $M$ rather than the non-compact space $\R \times \R^d \backslash \{(0,0)\}$, it will be convenient to normalise $u_0$ and the differential operators in ${\mathcal D}$ and ${\mathcal P}$ to be $T$-invariant.  More precisely, let us introduce the $T$-invariant vector fields
$$ X_j \coloneqq \rho \partial_{x_j}; \quad X_t \coloneqq \rho^2 \partial_t$$
on $\R \times \R^d \backslash \{(0,0)\}$ (or the quotient space $M$)
for $j=1,\dots,d$, where we identify vector fields with first-order differential operators in the usual fashion.  We also introduce the pairs of rescaled differential operators
\begin{align*}
{\mathcal P}' &\coloneqq \{ (1,1) \} \cup \{ (X_j, X_k): 1 \leq j \leq k \leq d \} \\
&\quad \cup \{ (X_j, X_t): 1 \leq j \leq d \} \\
&\quad \cup \{ (1, i X_j): 1 \leq j \leq d \} \\
&\quad \cup \{ (1, i X_t) \}
\end{align*}
and then define
$$ G_{{\mathcal P}'}[u,v] \coloneqq ( \langle D_1 u, D_2 v \rangle_{\C^m} )_{(D_1,D_2) \in {\mathcal P}'}$$
for smooth $u,v: \R \times \R^d \backslash \{(0,0)\} \to \C^m$  Note that the operators in ${\mathcal P}'$ commute with the dilation operator $T$; in particular, if $u$ is $T$-invariant, then so is $G_{{\mathcal P}'}[u,u]$.

Proposition \ref{gippy} is then a consequence of

\begin{proposition}\label{gippy-norm}  Let $m$ be a positive integer.  Let $u: \R \times \R^d \backslash \{(0,0)\} \to \C^m$ be a smooth map which is $T$-invariant and free, and let $H_{{\mathcal P}'}: \R \times \R^d \backslash \{(0,0)\} \to \R$ be smooth and $T$-invariant.   Then, if $\eps>0$ is small enough, there exists a smooth map $u_\eps: \R \times \R^d \backslash \{(0,0)\} \to \C^m$ that is smooth and $T$-invariant, such that
\begin{equation}\label{gap}
 G_{{\mathcal P}'}[u_\eps,u_\eps] = G_{{\mathcal P}'}[u,u] + \eps^2 H_{{\mathcal P}'}.
\end{equation}
\end{proposition}

To see why Proposition \ref{gippy-norm} implies Proposition \ref{gippy}, we observe that if $u_0: \R \times \R^d \backslash \{(0,0)\} \to \C^m$ is smooth and obeys \eqref{uts} with $\alpha=0$, and we set $u: \R \times \R^d \backslash \{(0,0)\} \to \C^m$ to be the map $u \coloneqq \rho^{\frac{2}{p-1}} u_0$, then $u$ is $T$-invariant, and we have the linear relation
$$G_{{\mathcal P}'}[u,u](t,x) = S_{t,x} G_{{\mathcal P}}[u_0,u_0](t,x),$$ 
for some invertible linear transformation $S_{t,x}: \R^{{\mathcal P}} \to \R^{{\mathcal P}'}$.  The exact form of $S_{t,x}$ is not important, but for sake of explicitness we can compute $S_{t,x} (G_{D_1,D_2})_{(D_1,D_2) \in {\mathcal P}} \coloneqq (G'_{D_1,D_2})_{(D_1,D_2) \in {\mathcal P}'}$, where
\begin{align*}
G'_{1,1} &\coloneqq \rho^{\frac{4}{p-1}} G_{1,1} \\
G'_{X_j,X_k} &\coloneqq \rho^{\frac{4}{p-1}} ( \rho^2 G_{\partial_{x_j}, \partial_{x_k}} +  \frac{2}{p-1} \rho (\partial_{x_j} \rho) G_{1, \partial_{x_k}} + \frac{2}{p-1} \rho (\partial_{x_k} \rho) G_{\partial_{x_j},1} +
\frac{4}{(p-1)^2} (\partial_{x_j} \rho) (\partial_{x_k} \rho) G_{1,1}) \\
G'_{X_j,X_t} &\coloneqq \rho^{\frac{4}{p-1}} ( \rho^3 G_{\partial_{x_j}, \partial_t} +  \frac{2}{p-1} \rho^2 (\partial_{x_j} \rho) G_{1, \partial_t} + \frac{2}{p-1} \rho^2 (\partial_t \rho) G_{\partial_{x_j},1} +
\frac{4}{(p-1)^2} \rho (\partial_{x_j} \rho) (\partial_t \rho) G_{1,1}) \\
G'_{1,iX_j} &\coloneqq \rho^{\frac{4}{p-1}} \rho G_{1, i\partial_{x_j}} \\
G'_{1,iX_t} &\coloneqq \rho^{\frac{4}{p-1}} \rho^2 G_{1, i\partial_t}.
\end{align*}
Also, from the product rule we see that $u_0$ is free if and only if $u$ is free.  If one then applies Proposition \ref{gippy-norm} with 
$$H_{{\mathcal P}'}(t,x) \coloneqq S_{t,x} H_{{\mathcal P}}(t,x)$$
(which one verifies to be $T$-invariant), then for $\eps$ small enough, one can find a smooth map $u_\eps: \R \times \R^d \backslash \{(0,0)\} \to \C^m$ that is smooth and $T$-invariant, such that
\begin{equation}\label{eqn}
 G_{{\mathcal P}'}[u_\eps,u_\eps](t,x) = S_{t,x} G_{{\mathcal P}}[u,u](t,x) + \eps^2 S_{t,x} H_{{\mathcal P}}(t,x)
\end{equation}
for all $(t,x) \in \R \times \R^d \backslash \{(0,0)\}$.  If we then define $u_{0,\eps}: \R \times \R^d \backslash \{(0,0)\} \to \C^m$ to be the map $u_{0,\eps} \coloneqq \rho^{-\frac{2}{p-1}} u_\eps$, then $G_{{\mathcal P}'}[u_\eps,u_\eps](t,x) = S_{t,x} G_{{\mathcal P}}[u_{0,\eps},u_{0,\eps}](t,x)$, so on applying $S_{t,x}^{-1}$ to \eqref{eqn} we obtain Proposition \ref{gippy} as claimed.

It remains to prove Proposition \ref{gippy-norm}.  Henceforth the reference solution $u$ will be held fixed, as well as the range dimension $m$.  If we write $u_\eps = u + v$, then we can rewrite the equation \eqref{gap} as
\begin{equation}\label{eq}
 L_u v = \eps^2 H_{{\mathcal P}'} - G_{{\mathcal P}'}[v,v]
\end{equation}
where $L_u$ is the linear operator defined on smooth functions $u: M \to \C^m$ by setting $L_u v: M \to \R^{{\mathcal P}'}$ to be the function
$$ L_u v \coloneqq G_{{\mathcal P}'}[u,v] + G_{{\mathcal P}'}[v,u].$$

Our task is now to find a smooth solution $v: M \to \C^m$ to the equation \eqref{eq}.
In coordinates, we can expand $L_u v = ((L_uv)_{D_1,D_2})_{(D_1,D_2) \in {\mathcal P}'}$ as
\begin{align*}
(L_u v)_{1,1} &\coloneqq 2\langle v, u \rangle_{\C^m} \\
(L_u v)_{X_j,X_k} &\coloneqq \langle X_j v, X_k u \rangle_{\C^m} + \langle X_j u, X_k v \rangle_{\C^m}  \\
&= X_j \langle v, X_k u \rangle_{\C^m} + X_k \langle v, X_j u \rangle_{\C^m} - \langle v, (X_j X_k + X_k X_j) u \rangle_{\C^m} \\
(L_u v)_{X_j,X_t} &\coloneqq \langle X_j v, X_t u \rangle_{\C^m} + \langle X_j u, X_t v \rangle_{\C^m}  \\
&= X_j \langle v, X_t u \rangle_{\C^m} + X_t \langle v, X_j u \rangle_{\C^m} - \langle v, (X_j X_t + X_t X_j) u \rangle_{\C^m} \\
(L_u v)_{1,iX_j} &\coloneqq \langle v, i X_j u \rangle_{\C^m} + \langle u, i X_j v \rangle_{\C^m} \\
&= 2 \langle v, i X_j u \rangle_{\C^m} - X_j \langle v, iu \rangle_{\C^m} \\
(L_u v)_{1,iX_t} &\coloneqq \langle v, i X_t u \rangle_{\C^m} + \langle u, i X_t v \rangle_{\C^m} \\
&= 2 \langle v, i X_t u \rangle_{\C^m} - X_t \langle v, iu \rangle_{\C^m}.
\end{align*}
Observe that the components of $L_u v$ are expressed in terms of the coefficients $\langle v, D u \rangle_{\C^m}$, where $D$ ranges over the collection
\begin{align*}
 {\mathcal F} &\coloneqq \{ 1, i, X_t, iX_t \} \\
&\quad  \cup \{ X_j: 1 \leq j \leq d \} \cup \{ iX_j: 1 \leq k \leq d \} \\
&\quad \cup \{ X_j X_k + X_k X_j: 1 \leq j \leq k \leq d \}
\end{align*}
of $T$-invariant differential operators (which may thus be viewed as differential operators on $M$).
As $u$ is free, we see at each point in $M$ that the vectors $Du, D \in {\mathcal F}$ are linearly independent over $\R$.  By Cramer's rule, we may thus find smooth dual fields $w_D: M \to \C^m$ (depending on $u$), which are pointwise real linear combinations of the $Du, D \in {\mathcal F}$, such that
\begin{equation}\label{w-bas}
\langle w_{D_1}, D_2 u \rangle_{\C^m} = \delta_{D_1,D_2} 
\end{equation}
pointwise on $M$, where $\delta_{D_1,D_2}$ is the Kronecker delta (equal to $1$ when $D_1=D_2$, and zero otherwise).  This provides a zeroth-order right-inverse $Z_u$ to $L_u$, defined on any smooth collection $F = (F_{D_1,D_2})_{(D_1,D_2) \in {\mathcal P}'}$ of functions $F_{D_1,D_2}: M \to \R$ by setting $Z_u F: M \to \C^m$ to be the function
\begin{align*}
Z_u F &\coloneqq \frac{1}{2} F_{1,1} w_1 - \sum_{1 \leq j \leq k \leq d} F_{X_j,X_k} w_{X_j X_k + X_k X_j} \\
&\quad - \sum_{j=1}^d F_{X_j,X_t} w_{X_j X_t + X_t X_j} \\
&\quad - \sum_{j=1}^d F_{1,iX_j} w_{iX_j} \\
&\quad - F_{1,iX_t} w_{iX_t}.
\end{align*}
One can easily check from \eqref{w-bas} and the expansion of $L_u$ in coordinates that $Z_u$ is indeed a right-inverse for $L_u$, that is to say
$$ L_u Z_u F = F$$
for all smooth $F: M \to \C^m$.

One could now try to locate a solution to \eqref{eq} using this left-inverse by solving the equation
$$
v = Z_u \eps^2 H_{{\mathcal P}'} - Z_u G_{{\mathcal P}'}[v,v]$$
which would imply \eqref{eq}.  Here we face the familiar problem of \emph{loss of derivatives}, since the Gram-type operator $G_{{\mathcal P}'}$ is first-order whereas $Z_u$ is zeroth order.  It is possible to recover this loss of derivative problem for $\eps$ small enough using the technique of Nash-Moser iteration as in \cite{nash}.  However, we instead follow the simpler approach of Gunther \cite{gunther}, by obtaining a decomposition of the form
\begin{equation}\label{decomp}
 G_{{\mathcal P}'}[v,v] = L_u Q_0[v,v] + Q_1[v,v] 
\end{equation}
where $Q_0, Q_1$ are ``zeroth order'' operators.  We will then be able to use a contraction mapping argument to obtain a solution to the equation
\begin{equation}\label{iter}
 v = Z_u \eps^2 H_{{\mathcal P}'} - Q_0[v,v] - Z_u Q_1[v,v]
\end{equation}
for $\eps$ small enough; applying $L_u$ to both sides, we obtain a solution to \eqref{eq} as desired.

It remains to obtain the decomposition \eqref{decomp} and solve the equation \eqref{iter}.  We will need an elliptic second order operator $-\Delta$ on $M$.  The precise choice of $-\Delta$ is not important, but for sake of concreteness we will take $\Delta$ to be the Laplace-Beltrami operator on $M$ with the Riemannian metric
$$ ds^2 \coloneqq \sum_{j=1}^d \rho^2 dx_j^2 + \rho^4 dt^2$$
(noting that the right-hand side is $T$-invariant and thus descends to a metric on $M$), with the sign chosen so that $-\Delta$ is positive semi-definite; in particular, one can define the resolvent operator $(1-\Delta)^{-1}$ on smooth functions on $M$.  We can then expand
\begin{equation}\label{gexp}
 G_{{\mathcal P}'}[v,v] = - (1-\Delta)^{-1} F + (1-\Delta)^{-1} Q_2[v,v]
\end{equation}
where
$$ F \coloneqq G_{{\mathcal P}'}[\Delta v, v] + G_{{\mathcal P}'}[v, \Delta v] $$
and
$$ Q_2[v,v] \coloneqq  G_{{\mathcal P}'}[v,v] - \Delta G_{{\mathcal P}'}[v,v] + G_{{\mathcal P}'}[\Delta v,v] + G_{{\mathcal P}'}[v,\Delta v].$$
Observe from the Leibniz rule that $Q_2[v,v]$ takes the schematic form
$$ Q_2[v,v] = \sum_{0 \leq a,b \leq 2} O( \nabla^a v \nabla^b v )$$
where the gradient $\nabla$ is with respect to the Riemannian metric $ds^2$ (and the implied coefficients in the $O()$ notation are smooth on $M$); the point is that the ``carr\'e du champ'' type expression
$$ - \Delta G_{{\mathcal P}'}[v,v] + G_{{\mathcal P}'}[\Delta v,v] + G_{{\mathcal P}'}[v,\Delta v] $$
does not have any terms involving third or higher derivatives after cancelling out the top order terms.  Thus, $Q_2$ is a ``zeroth order operator'', for instance it is a bounded bilinear operator on the H\"older space $C^{2,\alpha}(M)$ for any $0 < \alpha < 1$, as can be seen by classical Schauder estimates.

The components of $F$ can be expanded using the Leibniz rule as
\begin{align*}
F_{1,1} &= 2 \langle \Delta v, v \rangle_{\C^m} \\
F_{X_j,X_k} &= X_j \langle \Delta v, X_k v \rangle_{\C^m} + X_k \langle \Delta v, X_j v \rangle_{\C^m} - \langle \Delta v, (X_j X_k + X_k X_j) v \rangle_{\C^m} \\
F_{X_j,X_t} &= X_j \langle \Delta v, X_t v \rangle_{\C^m} + X_t \langle \Delta v, X_j v \rangle_{\C^m} - \langle \Delta v, (X_j X_t + X_t X_j) v \rangle_{\C^m} \\
F_{1,iX_j} &= - X_j \langle \Delta v, i v \rangle_{\C^m} + 2 \langle \Delta v, i X_j v \rangle_{\C^m} \\
F_{1,iX_t} &= - X_t \langle \Delta v, i v \rangle_{\C^m} + 2 \langle \Delta v, i X_t v \rangle_{\C^m}.
\end{align*}
Comparing this with \eqref{w-bas} and the components of $L_u$, we can then write
\begin{equation}\label{fexp}
 F = L_u Q_3[v,v] + Q_4[v,v]
\end{equation}
where $Q_3[v,v]: M \to \C^m$ is the function
\begin{align*}
Q_3[v,v] &\coloneqq \langle \Delta v, v \rangle_{\C^m} w_1 \\
&\quad + \sum_{k=1}^d \langle \Delta v, X_k v \rangle_{\C^m} w_{X_k} \\
&\quad + \langle \Delta v, X_t v \rangle_{\C^m} w_{X_t} \\
&\quad + \langle \Delta v, iv \rangle_{\C^m} w_i
\end{align*}
and $Q_4[v,v]: M \to \R^{{\mathcal P}'}$ is given in components as
\begin{align*}
Q_4[v,v]_{1,1} &\coloneqq 0 \\
Q_4[v,v]_{X_j,X_k} &\coloneqq - \langle \Delta v, (X_j X_k + X_k X_j) v \rangle_{\C^m}  \\
Q_4[v,v]_{X_j,X_t} &\coloneqq - \langle \Delta v, (X_j X_t + X_t X_j) v \rangle_{\C^m}  \\
Q_4[v,v]_{1,iX_j} &\coloneqq 2 \langle \Delta v, i X_j v \rangle_{\C^m} \\
Q_4[v,v]_{1,iX_t} &\coloneqq 2 \langle \Delta v, i X_t v \rangle_{\C^m}.
\end{align*}
Observe that, as with $Q_2[v,v]$, the expressions $Q_3[v,v]$ and $Q_4[v,v]$ both take the schematic form $\sum_{0 \leq a,b \leq 2} O( \nabla^a v \nabla^b v )$, as they does not contain any terms involving third or higher derivatives.

Using the identity
\begin{align*}
 (1-\Delta)^{-1} L_u &= L_u (1-\Delta)^{-1} + (1-\Delta)^{-1} [L_u, 1-\Delta] (1-\Delta)^{-1} \\
&= L_u (1-\Delta)^{-1} - (1-\Delta)^{-1} [L_u, \Delta] (1-\Delta)^{-1} 
\end{align*}
where $[A,B] = AB - BA$ denotes the commutator of $A,B$, as well as \eqref{gexp}, \eqref{fexp}, we obtain an expansion of the form \eqref{decomp} with
$$ Q_0[v,v] \coloneqq - (1-\Delta)^{-1} Q_3[v,v]$$
and
$$ Q_1[v,v] \coloneqq (1-\Delta)^{-1}( Q_2[v,v] - Q_4[v,v] ) + (1-\Delta)^{-1} [L_u,\Delta] (1-\Delta)^{-1} Q_3[v,v].$$
Observe that the commutator $[L_u,\Delta]$ is a second order differential operator on $M$ with smooth coefficients.  From Schauder theory we then conclude that (after depolarisation) $Q_0, Q_1$ are bounded bilinear operators on the H\"older space $C^{2,\alpha}(M)$ for any fixed $0 < \alpha < 1$.  As such, the contraction mapping theorem then guarantees a solution $v$ to the equation \eqref{iter} in the function space $C^{2,\alpha}(M)$ if $\eps$ is sufficiently small (depending on $u$ and $\alpha$).  We are almost done, except that we have not established that $v$ is smooth.  However, from further application of Schauder theory one can establish estimates of the form
$$ \| Q_i[v,v] \|_{C^{k,\alpha}(M)} \leq C_{u,\alpha} \| v \|_{C^{k,\alpha}(M)} \|v\|_{C^{2,\alpha}(M)} + C_{k,u,\alpha} \|v\|_{C^{k-1,\alpha}(M)}^2$$
for any $k \geq 2$ and $i=1,2$, where the quantities $C_{u,\alpha}, C_{k,u,\alpha}$ depend only on the subscripted parameters.  Crucially, the leading constant $C_{u,\alpha}$ is independent of $k$.  As such, a routine induction argument shows that if $\eps$ is sufficiently small (depending on $u$ and $\alpha$, but not on $k$) that all the iterates used in the contraction mapping theorem to construct $v$, and hence $v$ itself, are bounded in $C^{k,\alpha}(M)$ for any given $k \geq 2$, and so $v$ is smooth as required.  This (finally!) completes the proof of Proposition \ref{gadc}.

\end{document}